\documentclass[12pt]{amsart}
\usepackage[bb=libus]{mathalpha}

\usepackage[a-1b]{pdfx}   
\makeatletter \AtBeginDocument{\let\mathaccentV\AMS@mathaccentV} \makeatother

\usepackage{graphicx}

\usepackage{amsmath, amsthm, amssymb}
\usepackage{fullpage}
\usepackage{color}
\usepackage{soul}
\usepackage{enumitem}

\newtheorem{theorem}{Theorem}
\newtheorem{lemma}{Lemma}
\newtheorem{proposition}[lemma]{Proposition}
\newtheorem{corollary}[lemma]{Corollary}
\newtheorem{definition}[lemma]{Definition}
\newtheorem{remark}[lemma]{Remark}
\newtheorem{conjecture}{Conjecture}
\numberwithin{lemma}{section}

\newcommand{\R}{{\mathbb R}}

\newcommand{\Z}{{\mathbb Z}}

\renewcommand{\R}{\mathbb R}
\newcommand{\bL}{\mathbf L}
\newcommand{\bM}{\mathbf M}
\newcommand{\bbM}{\mathbb M}
\newcommand{\bbm}{\mathbb m}
\newcommand{\bP}{\mathbf P}
\newcommand{\bbP}{\mathbb P}
\newcommand{\bbp}{\mathbb p}
\newcommand{\bE}{\mathbf E}
\newcommand{\bbE}{\mathbb E}
\newcommand{\bbe}{\mathbb e}
\newcommand{\bN}{\mathbf N}
\newcommand{\bB}{\mathbf B}

\newcommand{\bR}{\mathbf R}
\newcommand{\bQ}{\mathbf Q}

\newcommand{\du}{\mathfrak{u}}
\newcommand{\dv}{\mathfrak{v}}

\newcommand{\bI}{\mathbf I}
\newcommand{\bJ}{\mathbf J}
\newcommand{\bK}{\mathbf K}

\newcommand{\bu}{{\bar u}}

\newcommand{\bw}{{\bar w}}

\newcommand{\la}{\langle}
\newcommand{\ra}{\rangle}
\newcommand{\abs}[1]{\lvert #1 \rvert}

\newcommand{\ol}{\overline}
\newcommand{\ms}{\mathbb M^\sharp}
\newcommand{\ps}{\mathbb P^\sharp}
\newcommand{\calR}{\mathcal{R}}

\def\bal{{bal}}
\def\res{res} 
\def\nonres{nres}
\def\low{pert} 
              
\def\unbal{unbal}
\def\rem{{rem}}
\def\himed{{hm}} 
\def\hihi{{hh}} 

\def\dyad{{\mathfrak c}}

\begin{document}

\title{Global solutions for 1D cubic defocusing dispersive equations, Part V: low regularity NLS}

\author{Mihaela Ifrim}
\address{Department of Mathematics, University of Wisconsin, Madison}
\email{ifrim@wisc.edu}

\author{ Ryan Martinez}
\address{Department of Mathematics, University of California at Berkeley}
\email{ryan\_martinez@berkeley.edu}

\author{ Daniel Tataru}
\address{Department of Mathematics, University of California at Berkeley}
\email{tataru@math.berkeley.edu}

\begin{abstract}
This article is motivated by a broad conjecture, formulated by the first and
last authors in earlier work, asserting that \emph{one-dimensional cubic
defocusing dispersive flows with small initial data have global, dispersive
solutions}. The conjecture was first established for a class of semilinear
Schr\"odinger-type models at $L^2$ regularity, the classical cubic NLS among
them. In a complementary direction, Harrop-Griffiths, Killip and
Vi\c{s}an~\cite{HGKV} have recently shown, using the completely integrable
structure, that the cubic NLS is globally well-posed in $H^s$ for every
$-\tfrac12 < s < 0$.

Our aim here is to extend the reach of the global well-posedness conjecture for one dimensional cubic NLS problems to data which is small in negative Sobolev
spaces, and to show that global \emph{dispersive} bounds persist there. We do
so for a broad class of nonlinearities which includes the cubic NLS but which
in general generates flows that are not completely integrable. Our method is
correspondingly robust, resting on density-flux identities, interaction
Morawetz estimates and an implicit normal form transformation rather than on
integrability, and it reaches all the way to the scaling-critical threshold, namely $s > -\tfrac12$.
As in the earlier work, the global bounds we obtain include both $L^6_{t,x}$
Strichartz estimates and bilinear $L^2_{t,x}$ estimates; these are new even for the classical defocusing cubic NLS
at negative Sobolev regularity. There, by scaling, our dispersive bounds also extend to the large data case. 

\end{abstract}

\subjclass[2020]{35Q55;   
35B40   
}
\keywords{NLS problems, defocusing, scattering, interaction Morawetz, global well-posedness}

\maketitle

\setcounter{tocdepth}{1}
\tableofcontents


\section{Introduction}

Our starting point is the following broad conjecture, proposed in earlier work
of the first and last authors.

\begin{conjecture}[\cite{IT-global,IT-conjecture}]
One-dimensional dispersive flows with cubic defocusing nonlinearities and small
initial data have global in time, dispersive solutions.
\end{conjecture}

Here the word dispersive is used to indicate that the solutions satisfy a range of global dispersive bounds; this is 
as opposed to scattering, which would also require the existence of an asymptotic state.

This conjecture was first proved in~\cite{IT-global} for 1D semilinear cubic NLS problems with a bounded nonlinearity and $L^2$ initial data. The solutions
constructed there were moreover shown to disperse at infinity in a sharp,
quantitative sense: they obey both $L^6_{t,x}$ Strichartz estimates and bilinear
$L^2_{t,x}$ estimates, despite the fact that the nonlinearity is non-perturbative
on large time scales. This was already new for the classical, integrable cubic defocusing 
NLS.

In the present paper we extend these results to initial data which is small in negative Sobolev spaces
$H^s$ with $-\tfrac12 < s < 0$, still for bounded nonlinearities. Under a mild additional restriction on the class of allowed symbols for the cubic nonlinearity, we prove that $H^s$ solutions obey
global, uniform bounds and scatter at infinity in the same quantitative sense as above, again through $L^6_{t,x}$ Strichartz and bilinear $L^2_{t,x}$ estimates.
It is important to note here that bounded cubic nonlinearities correspond to an $H^{-\frac12}$ critical Sobolev space, so the threshold $s = -\frac12$ is 
the sharp threshold for our result.

These bounds apply in particular to the cubic NLS at negative regularity, where by scaling one can also consider large data. For
that completely integrable model, global well-posedness in $H^s$ for
$-\tfrac12 < s < 0$ --- existence, uniqueness and continuous dependence --- was
established by Harrop-Griffiths, Killip and Vi\c{s}an~\cite{HGKV} through the
method of commuting flows and a family of microscopic conservation laws tied to
the integrable structure. Their method does not, however, yield dispersive information, so our bounds are new even in this case.
By comparison, the model we treat here is in general not integrable and admits no
exact conservation law.

\subsection{Cubic NLS problems in one space dimension}
One of the most fundamental one-dimensional dispersive
flows is the cubic NLS equation
\begin{equation}\label{nls3}
i u_t + u_{xx} = \pm u |u|^2,  \qquad u(0) = \du_0,
\end{equation}
which comes, according to the choice of sign, in a defocusing $(+)$
and a focusing $(-)$ flavor. Both are of interest in their own right, and as
model problems for more complex one-dimensional dispersive flows, semilinear and
quasilinear alike.

These are integrable problems, so for each 
nonnegative integer $k$ they admit a conserved energy $E^k$ at the $H^k$ level. It is far more difficult to show that conservation laws can be obtained at any $H^s$ level above scaling $s > -\frac12$, see \cite{KT-full} and \cite{KVZ}, as well as preceding partial results in \cite{KT,CCT,KT0}.

Local and then global well-posedness in $L^2$ and then $H^s$ for all $s \geq 0$ for both of these problems has been known for a long time, and can be proved using standard dispersive tools. The much more recent work in \cite{HGKV} uses integrable tools to expand this well-posedness result to its natural scaling limit, i.e. to all $H^s$ spaces with $s > -\frac12$.

However, the above global well-posedness results say little about the long time behavior of solutions, which also  differs sharply between the focusing and the defocusing models. In the focusing case the
equation admits small solitons, so its solutions cannot in general scatter. If
the initial data is in addition localized, one expects the solution to resolve
into a superposition of finitely many solitons together with a dispersive part.
This is the \emph{soliton resolution conjecture}, known in restricted settings
through the method of inverse scattering, see e.g.~\cite{IST-focusing}.

In the defocusing case, inverse scattering again handles localized data and
yields scattering of the global solution, see for instance~\cite{IST}. The same
conclusion can be reached more robustly, without inverse scattering, for small
localized data, see~\cite{IT-NLS} and the references therein. Much less is known
about scattering for nonlocalized $L^2$ data. If additional regularity is
assumed, however, one has the Planchon--Vega estimate~\cite{PV} (see also
Colliander--Grillakis--Tzirakis~\cite{MR2527809})
\begin{equation}\label{PV-est}
\|u\|_{L^6_{t,x}}^6 + \| \partial_x |u|^2\|_{L^2_{t,x}}^2
\lesssim \| \du_0\|_{L^2}^3 \|\du_0\|_{H^1},
\end{equation}
which in particular controls the $L^6_{t,x}$ Strichartz norm of the solution, and so yields a form of dispersive decay.

For problems with just $L^2$ data, the result in \cite{IT-global} shows that the solutions are still dispersive, and in effect satisfy a global $L^6$ bound,
\begin{equation}\label{PV-est-re}
\|u\|_{L^6_{t,x}} 
\lesssim \| \du_0\|_{L^2}.
\end{equation}
But until now, no such bounds are known below the $L^2$ regularity, i.e.  in $H^s$ with $-\frac12 < s < 0$.

\bigskip

Motivated by the above considerations, in this article we focus on defocusing cubic problems at low regularity. Precisely, we consider a cubic nonlinear Schr\"odinger (NLS) type model in one space dimension,
\begin{equation}\label{nls}
i u_t + u_{xx} = C(u,\bar u, u),  \qquad u(0) = \du_0,
\end{equation}
where $u : \mathbb{R}\times\mathbb{R}\rightarrow \mathbb{C}$ and $C$
is a translation invariant trilinear form whose symbol $c(\xi_1,\xi_2,\xi_3)$
may always be taken symmetric in $\xi_1,\xi_3$; see Section~\ref{s:multi} for an
expanded discussion of multilinear forms. The arrangement of the arguments
$u,\bar u, u$ of $C$ ensures that \eqref{nls} has the phase rotation symmetry
$u \mapsto u e^{i\theta}$, as is the case in most examples of interest. 

For the above problem, in the spirit of the defocusing global well-posedness conjecture, we ask whether it is globally well-posed, with dispersive solutions, for initial data which is small in well chosen negative Sobolev spaces $H^s$. 
We begin by reviewing the earlier results in \cite{IT-global} and \cite{IT-general}, which frame the question in the present article.

\subsection{$L^2$ data: the results of~\cite{IT-global}}
This article considers symbols $c(\xi_1,\xi_2,\xi_3)$ satisfying a set of assumptions as follows:

\begin{enumerate}[label=(H\arabic*)]

\item \label{h:c_smooth} Bounded and regular:
\begin{equation*}
|\partial_\xi^\alpha c(\xi_1,\xi_2,\xi_3)| \leq c_\alpha,
\qquad \xi_1,\xi_2,\xi_3 \in \R,\,   \mbox{  for every  multi-index $\alpha$}.
\end{equation*}

\item\label{h:c_conserv} Conservative:
\begin{equation*}
\Im c(\xi,\xi,\xi) = 0, \qquad \xi\in \R,
\quad \Im \nabla c(\xi,\xi,\xi) = 0 \qquad \xi\in \R.
\end{equation*}

\item\label{h:c_defocus} Defocusing:
\begin{equation*}
c(\xi,\xi,\xi) \geq c_0 > 0, \qquad \xi \in \R,  \mbox{ for some constant } c_0 \in \mathbb{R^+}.
\end{equation*}

\end{enumerate}

The simplest such trilinear form is $C = 1$, corresponding to the classical
cubic NLS. This problem is completely integrable, and so possesses infinitely
many conservation laws. By
contrast, the assumptions above guarantee no exact conservation law, at the $L^2$ level or at any other level.

The main $L^2$ result of~\cite{IT-global} asserts global well-posedness for
\eqref{nls} with small $L^2$ data, together with global $L^6_{t,x}$ and bilinear
$L^2_{t,x}$ control, as follows.

\begin{theorem}[\cite{IT-global}]\label{t:main-L2}
Under the above assumptions (H1), (H2) and (H3) on the symbol of the cubic form
$C$, any small initial data $\du_0$ with
\[
\|\du_0\|_{L^2} \leq \epsilon \ll 1,
\]
yields a unique global solution $u$ of \eqref{nls}, which satisfies the
following bounds:
\begin{enumerate}[label=(\roman*)]
\item Uniform $L^2$ bound:
\begin{equation}\label{main-L2}
\| u \|_{L^\infty_t L^2_x} \lesssim \epsilon.
\end{equation}

\item Strichartz bound:
\begin{equation}\label{main-Str}
\| u \|_{L^6_{t,x}} \lesssim \epsilon^\frac23.
\end{equation}

\item Bilinear Strichartz bound:
\begin{equation}\label{main-bi}
\| \partial_x (u \bar u(\cdot+x_0))\|_{L^2_t H_x^{-\frac12}} \lesssim \epsilon^2,
\qquad x_0 \in \R.
\end{equation}
\end{enumerate}
\end{theorem}

Taking $x_0 = 0$ recovers the more classical formulation of the
bilinear bound,
\begin{equation}\label{main-bi-diag}
\| \partial_x |u|^2\|_{L^2_tH^{-\frac12}_x} \lesssim \epsilon^2;
\end{equation}
keeping \eqref{main-bi} uniform in the translation $x_0$ records its
separate translation invariance, and is convenient in the proofs.

All of these bounds are Galilean invariant: while \eqref{nls} itself is not, the
class to which it belongs is. The estimates of Theorem~\ref{t:main-L2} are only
a representative sample of what is actually proved; in full strength this is a
frequency-envelope bound attached to a decomposition of $u$ on the unit
frequency scale, rather than the more traditional dyadic one. 

Specializing to the cubic NLS \eqref{nls3} and using scaling yields the
following large-data counterpart.

\begin{theorem}[\cite{IT-global}]\label{t:NLS-L2}
Consider the defocusing 1-d cubic NLS problem \eqref{nls3}$(+)$ with $L^2$
initial data $\du_0$. Then the global solution $u$ satisfies the following
bounds:
\begin{enumerate}[label=(\roman*)]
\item Uniform $L^2$ bound:
\begin{equation}\label{main-L2-model}
\| u \|_{L^\infty_t L^2_x} \lesssim \|\du_0\|_{L^2_x}.
\end{equation}

\item Strichartz bound:
\begin{equation}\label{main-Str-model}
\| u \|_{L^6_{t,x}} \lesssim  \|\du_0\|_{L^2_x}. 
\end{equation}

\item Bilinear Strichartz bound:
\begin{equation}\label{main-bi-model}
\| \partial_x |u|^2\|_{L^2_t (\dot H_x ^{-\frac12} + c L^2_x) } \lesssim
\|\du_0\|_{L^2}^2, \qquad c = \| \du_0\|_{L^2}.
\end{equation}
\end{enumerate}
\end{theorem}

The $L^6_{t,x}$ bound here may be compared with the Planchon--Vega
estimate~\eqref{PV-est}, which by contrast requires $H^1$ data\footnote{Or a uniform $H^{\frac12}$ bound for $u$, which can in turn be replaced by 
an $H^{\frac12}$
data bound using the subsequent $H^\frac12$ level conservation law of \cite{KT-full}.}.

We conclude our discussion of this result with two remarks
that will play a role in framing the question in the present work.

\begin{remark}\label{r:L2}
 Examining the scaling properties of the 
simplest model in the above results, i.e. the cubic NLS, one easily sees that the scale invariant Sobolev space corresponding to this set-up is $H^{-\frac12}$. This leads to the interesting question of whether the above result extends to all $H^s$ spaces with $-\frac12 < s < 0$; in that case, a simple scaling argument shows that the natural substitute for the condition (H1) should be
\begin{equation}\label{H1s}
|\partial_\xi^\alpha c(\xi_1,\xi_2,\xi_3)| \leq c_\alpha
\la\xi_1\ra^{2s\alpha_1}\la\xi_2\ra^{2s\alpha_2}\la\xi_3\ra^{2s\alpha_3},
\qquad \xi_1,\xi_2,\xi_3 \in \R,\,   \mbox{  for every  multi-index $\alpha$}.
\end{equation}
\end{remark}

\begin{remark}\label{r:Hs}
    The symbol smoothness on the unit scale in Theorem~\ref{t:NLS-L2} is directly connected to the 
    unit scale frequency decomposition which is used in the proof of the result. In this fashion, the condition \eqref{H1s} corresponding to $H^s$ solutions would 
    be naturally associated to a frequency decomposition on the scale $\delta \xi \approx \la \xi \ra^{-2s}$.
        In the limit at the scaling regularity $s \to -\frac12$,
    the frequency localization becomes dyadic. This is important because a lower frequency decomposition scale leads to a better $L^6_{t,x}$ bound for the solutions.
    \end{remark}

\subsection{Scale invariant data: The results of \cite{IT-general}}

The aim of the article \cite{IT-general} 
was  to provide the first proofs of the global well-posedness conjecture 
for general dispersion relations, i.e. not necessarily of Schr\"odinger type.
However, it also applies to the Schr\"odinger dispersion relation, 
though with a different symbol class:

\begin{enumerate}[label=(H\arabic*)']

\item \label{h:c_smooth-crit} Bounded and regular:
\begin{equation*}
|\partial_\xi^\alpha c(\xi_1,\xi_2,\xi_3)| \leq c_\alpha
\la\xi_1\ra^{\delta - \alpha_1}\la \xi_2\ra^{\delta -\alpha_2}\la\xi_3\ra^{\delta-\alpha_3},
\qquad \xi_1,\xi_2,\xi_3 \in \R,\,   \mbox{  for every  multi-index $\alpha$}.
\end{equation*}
\item\label{h:c_conserv-crit} Conservative:
\begin{equation*}
\Im c(\xi,\xi,\xi) = 0, \quad \Im \nabla c(\xi,\xi,\xi) = 0 \qquad \xi\in \R.
\end{equation*}

\item\label{h:c_defocus-crit} Defocusing:
\begin{equation*}
c(\xi,\xi,\xi) \geq c_0 \la \xi\ra^{3\delta} > 0, \qquad \xi \in \R  \mbox{ for some constant } c_0 \in \mathbb{R^+}.
\end{equation*}
\end{enumerate}
The critical Sobolev exponent associated 
to this problem is 
\[
s_c = \frac{3\delta-1}2,
\]
and the result of \cite{IT-general} specialized to this case gives:

\begin{theorem}[\cite{IT-general}]\label{t:main-crit}
Let $\dfrac13 < \delta < 1$. Under the above assumptions (H1)', (H2)' and (H3)' on the symbol of the cubic form
$C$, any small initial data $\du_0$ with
\[
\|\du_0\|_{H^{s_c}} \leq \epsilon \ll 1,
\]
yields a unique global solution $u$ of \eqref{nls}, which satisfies 
(i) uniform $H^{s_c}$ bounds, (ii) 
$L^6_{t,x}$ Strichartz bounds, and 
(iii) bilinear $L^2_{t,x}$ bounds.
\end{theorem}

We conclude our discussion of this result with two remarks
which broadly parallel Remarks~\ref{r:L2}, \ref{r:Hs} associated 
to Theorem~\ref{t:NLS-L2}.

\begin{remark} \label{r:crit}
Compared to \cite{IT-global}, here we are able to work in the critical Sobolev spaces, though for a different class of symbols, which are now assumed to be smooth on the dyadic scale. The upper constraint on $\delta$ is simply in order to ensure that the problem is semilinear. The lower constraint $\dfrac13$, on the other hand, is more of a technical nature. The interesting open question, in this case, would be to extend the range of $\delta$ to $(0,1)$.
This is not the same as the open question in Remark~\ref{r:L2}, though we note that the (possibly forbidden) endpoints of the two open problems are identical, corresponding to $\delta = 0$
and $s = s_c = -\frac12$.
\end{remark}

\begin{remark}
 The symbol $c$ in the above theorem is assumed to be smooth on the dyadic scale. This is natural since the 
 result applies to initial data in critical Sobolev spaces,
 see also Remark~\ref{r:crit}. If instead one were to consider solutions in higher regularity spaces, then it 
 should be possible to work with narrower frequency localizations.
\end{remark}

\subsection{$H^s$ data, $-\tfrac12 < s < 0$: the new results}

We come now to the main contribution of the paper, the counterpart of
Theorem~\ref{t:main-L2} in negative Sobolev spaces. Throughout we fix
$-\tfrac12 < s < 0$; the endpoint $s = -\tfrac12$ is the scaling-critical
exponent for the cubic NLS, and marks the natural limit of our method: it is
approached but never attained.

In terms of the assumptions on the symbol, we retain the condition (H3), but we strengthen (H1) in two ways, (i) by requiring symbol regularity on the dyadic scale, and (ii) by asking 
for some\footnote{Examining the proofs, one 
can see that the extra regularity is only needed in the most interesting range
$-\frac12 < s < -\frac13$.} additional regularity in the case of low frequency outputs. By contrast, we are able to weaken the condition (H2), by eliminating any requirement on $\nabla_\xi c$.
Precisely, we will work with 

\begin{enumerate}[label=(H\arabic*s)]

\item \label{h:c_smooth-s} Enhanced symbol regularity: 
for all frequencies $\xi_1,\xi_2,\xi_3$ and all multiindices $\alpha$,
\begin{equation}\label{normal-reg}
|\partial_\xi^\alpha c(\xi_1,\xi_2,\xi_3)| \leq c_\alpha
\la\xi_1\ra^{-\alpha_1}\la\xi_2\ra^{-\alpha_2}\la\xi_3\ra^{-\alpha_3}
\end{equation}
and
\begin{equation}\label{extra-reg}
|\partial_\xi^\alpha c(\xi_1,\xi_2,\xi_3)| \leq c_\alpha
\la \xi_1 \ra^{-\alpha_1-\frac{\alpha_3}2}\la\xi_2\ra^{-\alpha_2}
 \la\xi_3\ra^{-\frac{\alpha_3}2} \mbox{  when }
 |\xi_1 - \xi_2 + \xi_3| \ll |\xi_3| \ll |\xi_1| \sim |\xi_2| 
\end{equation}
\item\label{h:c_conserv+} Conservative:
\begin{equation*}
\Im c(\xi,\xi,\xi) = 0, \qquad \xi\in \R.
\end{equation*}
\end{enumerate}
By symmetry the bound \eqref{extra-reg} holds with $\xi_1$ and $\xi_3$ interchanged.
We remark that the bound \eqref{extra-reg} strengthens
the symbol regularity (but not size) in \eqref{normal-reg} exactly in the nonperturbative regime $high \times high \times medium \to low$.  

Concerning \ref{h:c_conserv+} we note that it 
relaxes and simplifies the original condition \eqref{h:c_conserv}.
\medskip

To state our main result we need to introduce some notations. At $L^2$ regularity the bilinear content was recorded in a single estimate on
the product itself: by~\eqref{main-bi-diag}, $\partial_x|u|^2$ lies in
$L^2_t H^{-1/2}_x$, uniformly across frequencies. This is no longer possible once
$s<0$. The scale at which the bilinear $L^2_{t,x}$ bound holds then depends on the
frequencies involved 
(see Section~\ref{s:global}), and the two basic
interactions are governed by different weights: the diagonal high--high
interactions and the off-diagonal low--high interactions can no longer be
measured in a common norm of $u\bar u$, and the low--high contribution in fact
diverges logarithmically as the low frequency ranges below the high one. One
must therefore separate these interactions before measuring them, through the
paraproduct decomposition
\[
u\,\bar u \;=\; T_{\bar u} u + T_u \bar u + \Pi(u,\bar u),
\qquad
T_f g = \sum_{j \ll k} P_j f\, P_k g,
\qquad
\Pi(f, g) = \sum_{j \sim k} P_j f\, P_k g,
\]
where $P_j$ denotes the Littlewood--Paley projections of
Section~\ref{ss:lattice_decomp}. The high--high part $\Pi(u,\bar u)$ is then
controlled symmetrically, with a common weight $\la D\ra^{s-\frac14}$ on each
factor; the low--high part $T_{\bar u}u$ instead calls for unequal weights on its
two factors, together with an arbitrarily small loss $\delta_0>0$ on the low one
to absorb the logarithm. The bilinear bounds of Theorem~\ref{t:main} are stated
in exactly these terms.

We are now ready to state the main result of the paper, the analogue of
Theorem~\ref{t:main-L2} at negative regularity.
\begin{theorem}\label{t:main}
Under the assumptions \ref{h:c_smooth-s}, \ref{h:c_conserv+} and \ref{h:c_defocus}
on the symbol of the cubic form $C$, and for $-\tfrac12 < s < 0$, any $L^2$ initial
data $\du_0$ satisfying the smallness condition
\[
\|\du_0\|_{H^s} \leq \epsilon \ll 1,
\]
yields a unique global solution $u$ of \eqref{nls}, which satisfies the
following bounds:
\begin{enumerate}[label=(\roman*)]
\item Uniform $H^s$ bound:
\begin{equation}\label{main-Hs}
\| u \|_{L^\infty_t H^s_x} \lesssim \epsilon.
\end{equation}
\item Strichartz bound:
\begin{equation}\label{main-Str-Hs}
\| \la D\ra^{-\frac{1-4s}{6}} u \|_{L^6_{t,x}} \lesssim \epsilon^{\frac23}.
\end{equation}
\item Bilinear $L^2_{t,x}$ bounds: for every $\delta_0 > 0$,
\begin{equation}\label{main-bi-Hs}
\begin{aligned}
\big\| \partial_x\, T_{\la D\ra^{s-\delta_0}\bar u}\, 
    \la D\ra^{s-\frac12} u
\big\|_{L^2_{t,x}} &\lesssim \epsilon^2,\\[2pt]
\big\| \partial_x\, \Pi\big(\la D\ra^{s-\frac14} u,\, 
    \la D\ra^{s-\frac14}\bar u\big)
\big\|_{L^2_{t,x}} &\lesssim \epsilon^2,
\end{aligned}
\end{equation}
with implicit constants depending on $\delta_0$.
\end{enumerate}
\end{theorem}

As before, the estimates above are only a representative sample. The sharper
frequency-envelope statement, including a frequency-localized,
translation-uniform transversal bilinear estimate, is Theorem~\ref{t:boot} in
Section~\ref{s:boot}, from which \eqref{main-Hs}--\eqref{main-bi-Hs} are
recovered in Section~\ref{s:global}. We continue with several remarks:

\begin{remark}
In the above theorem we require the initial data to be in $L^2$ (though possibly with a large $L^2$ norm) due to the fact that this is where we have a local well-posedness result. However, should an $H^s$ local well-posedness result become available in the future, our bounds would 
directly apply to the $H^s$ solutions.  This is the case 
for instance in the completely integrable case $c=1$, where such a result is already available in \cite{HGKV}.
\end{remark}

\begin{remark}
Our condition (H1s) requires symbol regularity on the dyadic scale, rather than on the narrower scale discussed in 
Remark~\ref{r:Hs}. The reader should regard this as a 
minor compromise for expository reasons, rather than as 
something fundamental. In any case, the scales in Remark~\ref{r:Hs} converge to the dyadic scales in the limit $s \to -\frac12$. The price we pay is a slightly weaker $L^6_{t,x}$ bound  in \eqref{main-Str-Hs}, with the exponent $\frac{1-4s}6$ rather than the expected exponent $-s$. Again, the difference between the two decays to zero as $s \to -\frac12$.
\end{remark}

Finally, we note that comparing with the two earlier results in Theorems~\ref{t:main-L2},\ref{t:main-crit}, our new result can be 
best interpreted as the natural low regularity expansion of  Theorem~\ref{t:main-L2}. However,
one may also adopt the opposite viewpoint, 
which leads us to

\begin{conjecture}
With a minor enhancement of the symbol class for $C$, the result in Theorem~\ref{t:main-crit} extends to the maximal range $0< \delta <1$.    
\end{conjecture}

Specialized to the integrable defocusing cubic NLS case,
Theorem~\ref{t:main} proves new global $L^6_{t,x}$ Strichartz and bilinear 
$L^2_{t,x}$ bounds for the negative-regularity solutions of~\cite{HGKV}. By scaling, 
these even extend to the large data case, with implicit constants depending 
polynomially on the data size $\|\du_0\|_{H^s}$:

\begin{theorem}\label{t:NLS-Hs}
Consider the defocusing 1-d cubic NLS problem \eqref{nls3}$(+)$ with 
    $H^s$ initial data $\du_0$ for $-1/2 < s < 0$. 
    Then the global solution $u$
    satisfies the following bounds:
\begin{enumerate}[label=(\roman*)]
\item Uniform $H^s$ bound:
\begin{equation}\label{main-Hs-model}
    \| u \|_{L^\infty_t H^s_x} 
    \lesssim \|\du_0\|_{H^s_x}(1 + \|\du_0\|_{H^s_x}^{\frac{-2s}{1+2s}}).
\end{equation}

\item Strichartz bound:
\begin{equation}\label{main-s-Str-model}
    \|\la D\ra^{-\frac{1-4s}{6}} u \|_{L^6_{t,x}} 
    \lesssim \|\du_0\|_{H^s_x}^{2/3}(1 + \|\du_0\|_{H^s_x}^{\frac{1-4s}{3+6s}}). 
\end{equation}

\item Bilinear Strichartz bound: for every $\delta_0 > 0$ 
\begin{equation}\label{main-bi-Hs-model}
\begin{aligned}
    \big\| \partial_x\, T_{\la D\ra^{s-\delta_0}\bar u}\, 
    \la D\ra^{s-\frac12} u
    \big\|_{L^2_{t,x}} &\lesssim \|\du_0\|_{H^s}^{2}
    (1 + \|\du_0\|_{H^s_x}^{\frac{1 - 4s}{1+2s}}),\\[2pt]
    \big\| \partial_x\, \Pi\big(\la D\ra^{s-\frac14} u,\, 
    \la D\ra^{s-\frac14}\bar u\big)
\big\|_{L^2_{t,x}} &\lesssim \|\du_0\|_{H^s}^2
    (1 + \|\du_0\|_{H^s_x}^{\frac{1 - 4s}{1+2s}}).
\end{aligned}
\end{equation}

\end{enumerate}
\end{theorem}

We note here that the large powers of 
$\|\du_0\|_{H^s}$ go to infinity as $s$ approaches
$-1/2$. In fact, these large powers come from low frequencies rather than high 
frequencies and so may be removed in norms adapted to the scale on which 
the transition from low frequencies to high frequencies takes place. This is discussed in more detail in Proposition 
\ref{prop:large} in Section \ref{s:global}.

\subsection{Further discussion}
The fundamental difficulty in working below $L^2$ is that the conserved mass no
longer controls the solution at the level of the $H^s$ norm. Relative to that
norm a dyadic frequency $N$ component of $u$ carries the amplified mass
$\|P_N u\|_{L^2}\sim N^{-s}\|P_N u\|_{H^s}$, so for $s<0$ it is the high
frequencies --- and above all the transfer of energy from high to low
frequencies --- that must be controlled, a transfer that grows  as
$s\to-\tfrac12$. We organize the analysis around a splitting of the nonlinearity
into a \emph{balanced} part, in which all frequencies are comparable, and an
\emph{unbalanced} part, in which the output frequency is separated from the
inputs. The balanced part is treated through a density-flux and interaction
Morawetz analysis, carried out in the nonlocal, frequency localized setting
forced on us by a general nonlinearity. Peeling off perturbative components of the unbalanced component, we are left with its $hhm \to l$ part, which at negative regularity is genuinely non-perturbative, and is removed by an implicit normal form transformation. The whole is then closed by a bootstrap, run on a dyadic
frequency decomposition, that propagates energy, Strichartz and bilinear $L^2_{t,x}$
bounds at once. Five ideas, each used in a somewhat nonstandard way, drive this
scheme. The last of them, the implicit normal form transformation, is the
principal novelty of this article.

\medskip

\emph{1. Density-flux identities for nonlocal multilinear forms.} For the
balanced nonlinearity we write the mass and momentum balances not as energy
identities but in density-flux form, the densities and fluxes now being
translation invariant multilinear forms. The conservative hypothesis (H2) forces
the symbol of the leading quartic source term to vanish on the resonant set, and a division lemma
then separates it into three structurally distinct pieces: a genuine flux, an
energy correction (the part divisible by $\Delta^4\xi^2$), and a resonant
remainder with a null condition, i.e. a multilinear expression of the type $(\partial|u|^2)^2$, which is
therefore amenable to bilinear $L^2_{t,x}$ control. The division is carried out in
linear coordinates adapted to the resonant set, and is localized to adjacent
dyadic frequency intervals.

\medskip

\emph{2. Energy corrections.} To make these density-flux identities usable we
correct the frequency localized mass and momentum densities and their associated fluxes by carefully chosen
quartic terms $B^4$ respectively $R^4$, in the spirit of the $I$-method
\cite{I-method,I-method2}, but implemented at the level of the density-flux
identities themselves, in a manner closer to \cite{KT}. The corrections are
chosen so that the residual source is controlled by the very norms we propagate:
the $L^6_{t,x}$ Strichartz norm and the bilinear $L^2_{t,x}$ norm.

\medskip

\emph{3. Interaction Morawetz estimates.} Built from the corrected, frequency
localized mass and momentum density-flux pairs, our interaction Morawetz
functionals come in two forms. In the diagonal case the sixth order term has
diagonal symbol $p_j^4(\xi)\,c^{\bal}(\xi,\xi,\xi)$, which is \emph{positive}
exactly by the defocusing condition (H3); this positivity is the source of the
$L^6_{t,x}$ Strichartz bound. In the transversal case, with two separated
frequencies and uniformly in a relative translation, the same functionals yield
the bilinear $L^2_{t,x}$ bounds. The method originates with the three-dimensional
NLS analysis of \cite{MR2053757}, see also \cite{MR2415387,MR2288737}. Our
one-dimensional version is closer to \cite{PV}, recast in the language of
nonlocal multilinear forms.

\medskip

\emph{4. Frequency envelopes on the dyadic scale.} 
We track the distribution of
energy across frequencies by means of frequency envelopes, in the spirit of Tao
\cite{Tao-WM,Tao-BO}, but asking for the usual ``slowly varying'' condition 
to act more strongly from high frequencies to low frequencies than from low 
to high. 
This allows us to apply an $L^2$ based local theory which preserves $\dot H^s$. 
Unlike in the work \cite{IT-global}, we do not use maximal envelopes. This 
corresponds to our use of dyadic frequency intervals as opposed to unit 
frequency intervals. The
bootstrap propagates the energy, Strichartz and bilinear estimates
simultaneously, as in the authors' \cite{IT-BO}; the bilinear bounds are kept
uniform in a relative translation, which is what permits the multilinear symbols
to be estimated through their integrable kernels.

\medskip

\emph{5. An implicit normal form transformation.} The unbalanced nonresonant
interactions are the genuine obstruction at negative regularity: already with a
single cubic correction the high-frequency sum diverges for $s<-\tfrac13$, and
canceling the cubic term by an ordinary normal form merely shifts the problem to
quintic and higher order. We therefore pass to an \emph{implicit} normal form
\[
v = u + \sum_{n=1}^N B^{2n+1}(v, \bar v, \ldots, v),
\]
of finite order $2N+1$, which resumes into a single change of variables the entire
infinite cascade of corrections an explicit transformation would generate. The
construction is organized around a hierarchy of corrections $B^{n}_h$ and a
notion of \emph{good} remainder, and the order is governed by the regularity
through the relation $s > -\tfrac{N}{2N+1}$, so that every $s>-\tfrac12$ is
reached by choosing $N$ large enough. In the new variable the equation reduces to
a balanced cubic Schr\"odinger equation with a source term $N^{\unbal}(v)$ satisfying\footnote{ Here $\dyad > 1$ replaces $2$ as the dyadic step parameter, so that $P_j$ projects to frequency $\dyad^j$, see the discussion in the next section.} 
\[
    \sup_{x_0}\|P_j N^{\unbal}(v)\,P_j\bar v(\cdot+x_0)\|_{L^1_{t,x}} \lesssim \epsilon^4\, c_j^2\dyad^{-2sj},
\]
precisely the bound the energy estimate can absorb. This is the analytic heart of
the paper, and is carried out in Section~\ref{sec:unbalanced_corrections}.

\subsection{An outline of the paper}

Section~\ref{s:not} fixes notation for the function spaces and multilinear
forms, and---more importantly---introduces the class of admissible frequency
envelopes which will preserve both $L^2_x$ and $\dot H^s_x$ norms. 

In Section~\ref{s:local} we carry out a preliminary step in the proof of the
main result, namely the large $L^2_x$ data local well-posedness theorem. It is
independent of the global theory, and rests on a standard contraction argument, 
where we emphasize that the local solution remains controlled by the frequency 
envelope of the initial data, assuming smallness in $H^s$.

Section~\ref{s:energy} recasts the mass and momentum identities in density-flux
form. We refine this in two further steps: first by passing to frequency
localized mass and momentum densities, and then by improving their accuracy
through a carefully chosen quartic correction.

In Section~\ref{s:Morawetz} we begin from the classical interaction Morawetz
identities for the linear Schr\"odinger flow, and then use the density-flux
identities for the sharp frequency localized mass and momentum to derive refined
interaction Morawetz identities for our problem. For clarity we treat separately
the diagonal case, where equal frequency components interact, and the transversal
case, where the frequencies are separated.

Section~\ref{sec:unbalanced_corrections} is devoted to the unbalanced part of
the nonlinearity, whose nonperturbative part we remove by the implicit normal form transformation
described above, constructed to all the relevant orders. This is the longest and
most technical part of the paper.

The global result is then obtained through an intricate bootstrap argument, in
which energy, Strichartz and bilinear $L^2_{t,x}$ bounds are propagated together
in a frequency localized setting governed by frequency envelopes. The bootstrap
is set up in Section~\ref{s:boot}, which also contains the sharper
frequency-envelope form of our result, Theorem~\ref{t:boot}; the estimates that
close it are carried out in Section~\ref{s:fe-bounds}, on the basis of the
density-flux and interaction Morawetz identities obtained earlier.

Finally, in Section~\ref{s:global} we pass from the frequency localized bounds to
their global counterparts, completing the proof of the main result.%

\subsection*{Acknowledgements}
\leavevmode{
M.I. gratefully acknowledges support from the National Science Foundation
through grant DMS-2348908, from a Miller Visiting Professorship at UC Berkeley
during the Fall semester of 2023, from the Simons Foundation through a Simons
Fellowship in the Spring semester of 2024, and from a Vilas Associate
Fellowship.
R.M. gratefully acknowledges support from The Professor B.C. Wong Endowment in Mathematics as well as from the NSF grant DMS-2054975.
D.T. was supported by the NSF grants DMS-2054975 and DMS-2554866 and by a Simons Fellowship from the Simons Foundation.%


\section{Notations and preliminaries} \label{s:not}

\subsection{Littlewood-Paley decomposition and frequency envelopes}
\label{ss:lattice_decomp}

Throughout our analysis, we would like to localize our 
functions to inhomogeneous dyadic intervals $I_j$. 
However, at certain points, we would 
like for the sum of two points in a dyadic interval to be in a 
different dyadic interval. Thus we choose a constant
\[
    1 < \dyad
\]
such that 
\[
    \dyad - 1 \ll 1.
\]
It will not be important exactly what this number is, so it will suffice 
to choose $\dyad = 1+ 1/10$. However, to ease the complexity 
of formulas we will always 
prefer to leave $\dyad$ as a variable.

We let $p_0$ be a smooth function which is $1$ on $[-1,1]$
and is supported on $[-\dyad, \dyad]$.

Then we construct the usual Littlewood-Paley symbols 
\[
    p_{j+1}(\xi) = p_0(\xi/\dyad^{j+1}) - p_0(\xi/\dyad^{j})
\]
and we will denote by 
$P_j$ the projector whose symbol is $p_j$.

\bigskip

Corresponding to these dyadic intervals we will use the following notation 
to indicate how different dyadic regions compare. 

\begin{definition}
    We will say 
    \begin{itemize}
        \item $j \ll k$ (or $\dyad^j \ll \dyad^k$) if $k - j \geq 4$.
        \item $j \gg k$ if $k \ll j$ 
        \item $j \sim k$ if $j \not \ll k$ and $j \not \gg k$
        \item $j \cong k$ if $\abs{j - k} \leq 1$
    \end{itemize}
\end{definition}
\begin{remark}
    We make these definitions so that if $j \cong k \gg \ell$ and 
    $\xi$ is in the support of $p_j$ and $\eta$ is in the support of 
    $p_\ell$, then $\abs{\xi - \eta} \geq_{\dyad} \dyad^{k-2}$. In other words
    there is a buffer between adjacent intervals and separated intervals.
\end{remark}

Corresponding to these definitions we define
\[
    P_{\cong j} := \sum_{k \cong j} P_k 
\]
and so on.

\bigskip

Next we construct appropriate frequency envelopes for our problem
subordinate to our dyadic decomposition. We 
will want our frequency envelopes to be slowly varying as in \cite{Tao-WM,Tao-BO}, 
which is a natural condition in these problems 
which will help us control the flow of ``energy'' between frequencies. 

In our problem, what will be most important is controlling 
how much energy at high frequency flows to lower frequencies, and 
dual to this story, we will want to keep track of how fast the 
frequency envelopes decay at high frequency. This motivates 
the following definition:

\begin{definition}\label{def:admissible_envelope}
    Let $\delta > 0$ be small.
    We say a frequency envelope $\{c_j\}$ 
    is $\delta$-\emph{one-sided slowly varying} 
    or simply \emph{admissible}
    if for all $j_1 \geq j_2$
    \[
    c_{j_1} \lesssim c_{j_2}\dyad^{\delta(j_1-j_2)}.
    \]
    and if $j_2 \geq j_1$
    \[
        c_{j_1} \lesssim c_{j_2}\dyad^{(\delta -s)(j_2-j_1)}.
    \]
\end{definition}

We will consider initial data which is small in $H^s$, but potentially
large in $L^2$. Every such function $f \in H^s \cap L^2$ has an associated 
frequency envelope:
\[
    c_j = \|P_j f\|_{H^s}
\]
with the property that 
\[
    \|c_j\|_{\ell^2_j} = \|f\|_{H^s}
\]
and 
\[
    \|\dyad^{-sj}c_j\|_{\ell^2_j} \sim \|f\|_{L^2}
\]

An important observation is that the frequency envelopes with
these properties can always be placed under a $\delta$-one-sided slowly
varying frequency envelope with comparable properties.

\begin{proposition}\label{prop:freq_envelopes}
    Any $\ell^2$ frequency envelope $c$ can be 
    placed under a comparable $\delta$-one-sided slowly
varying envelope $\tilde c$. In other words,
    \[
        c \leq \tilde c, \qquad \|\tilde c\|_{\ell^2} \sim_\delta \|c\|_{\ell^2}
    \]
    Further, if $-1/2 < s < 0$ and  
    \[
        \|\dyad^{-sj}c_j\|_{\ell^2} < \infty, 
    \]
    then $\tilde c$ can be chosen such that
    \[
        \|\dyad^{-sj}\tilde c_j\|_{\ell^2} \sim_\delta
        \|\dyad^{-sj}c_j\|_{\ell^2} 
    \]
\end{proposition}
\begin{proof}
    We simply convolve $c$ with an appropriate kernel.

    Let 
    \[
        K(k) = \begin{cases}
            \dyad^{-(\delta-s) k} & \text{ if } k \geq 0\\
            \dyad^{\delta k} & \text{ if } k < 0\\
        \end{cases}.
    \]
    The important properties of $K$ are that, for all $k$ and $\ell$, 
    we have 
    \[
        K(k)K(\ell) \leq K(k + \ell),
    \]
   \[
        \sum_{k \in \Z} K(k) \lesssim 1,
    \] 
    and that 
    \[
        \sum_{k \in \Z} \dyad^{-sk} K(k) \lesssim 1.
    \]
    
    The first fact is obvious when $k$ and $\ell$ have the same sign
    and also when either is 0. Otherwise, 
    without loss of generality assume $k > 0 > \ell$. If $k + \ell > 0$, then 
    \[
        K(k)K(\ell) = \dyad^{-(\delta-s) k}\dyad^{\delta \ell}  
        = \dyad^{-(\delta-s)(k + \ell)}\dyad^{(2\delta-s) \ell}
        \leq K(k + \ell)
    \]
    and if $k + \ell < 0$, then 
    \[
        K(k)K(\ell) = \dyad^{-(\delta-s) k}\dyad^{\delta \ell}  
        = \dyad^{\delta(k + \ell)}\dyad^{-(2\delta-s) k}
        \leq K(k + \ell).
    \]
    The second fact follows from $\delta > 0$ and $-(\delta - s) < 0$.
    The third fact follows from $\delta - s > 0$ and $-(\delta - s) - s < 0$.

    Then we set 
    \[
        \tilde c_j = \sum_{k \geq 0} K(j-k)c_k.
    \]
    First, we see that $\tilde c_j$ is admissible since
    \[
        K(j_2 - j_1)\tilde c_{j_1} = \sum_{k\geq 0} K(j_2 - j_1)K(j_1 - k) c_{k}  
        \leq \sum_{k\geq 0} K(j_2 - k ) c_{k} = \tilde c_{j_2}
    \]

    Since $K(0) = 1$ we have 
    \[
        c_j \leq \sum_{k} K(j-k)c_k = \tilde c_j
    \]
    which also shows 
    $\|c_j\|_{\ell^2} \leq \|\tilde c_j\|_{\ell^2}$
    and 
    $\|\dyad^{-sj}c_j\|_{\ell^2} \leq \|\dyad^{-sj}\tilde c_j\|_{\ell^2}$.
    By Young's inequality and the second and third properties of $K$, we also 
    have 
    $\|\tilde c_j\|_{\ell^2} \lesssim_\delta \| c_j\|_{\ell^2}$
    and 
    $\|\dyad^{-sj}\tilde c_j\|_{\ell^2} \lesssim_\delta 
    \|\dyad^{-sj} c_j\|_{\ell^2}$
    which completes the proof.

\end{proof}

\subsection{Multilinear forms and symbols}\label{s:multi}

A key notion which is used throughout the paper is that 
of multilinear form. 
All our multilinear forms are invariant with respect to 
translations, and have as arguments either complex-valued 
functions or their complex conjugates. 

For an integer $k \geq 2$, we will use 
translation invariant $k$-linear  forms 
\[
(\mathcal D(\R))^{k} \ni (u_1, \cdots, u_{k}) \to     L(u_1,\bu_2,\cdots) \in \mathcal D'(\R),
\]
where the nonconjugated and conjugated entries are alternating.

Such a form is uniquely described by its symbol $\ell(\xi_1,\xi_2, \cdots,\xi_{k})$
via
\[
\begin{aligned}
L(u_1,\bu_2,\cdots)(x) = (2\pi)^{-k} & 
\int e^{i(x-x_1)\xi_1} e^{-i(x-x_2)\xi_2}
\cdots 
\ell(\xi_1,\cdots,\xi_{k})
\\ & \qquad 
u_1(x_1) \bu_2(x_2) \cdots  
dx_1 \cdots dx_{k} d\xi_1\cdots d\xi_k,
\end{aligned}
\]
or equivalently on the Fourier side
\[
\mathcal F L(u_1,\bu_2,\cdots)(\xi)
= (2\pi)^{-\frac{k-1}2} \int_{D}
\ell(\xi_1,\cdots,\xi_{k})
\hat u_1(\xi_1) \bar{\hat u}_2(\xi_2) \cdots  
d\xi_1 \cdots d\xi_{k-1},
\]
where, with alternating signs, 
\[
D = \{ \xi = \xi_1-\xi_2 + \cdots \}.
\]

They can also be described via their kernel
\[
L(u_1,\bu_2,\cdots)(x) =  
\int K(x-x_1,\cdots,x-x_{k})
u_1(x_1) \bu_2(x_2) \cdots  
dx_1 \cdots dx_{k},
\]
where $K$ is defined in terms of the  
Fourier transform  of $\ell$
\[
K(x_1,x_2,\cdots,x_{k}) = 
(2\pi)^{-\frac{k}2} \hat \ell(-x_1,x_2,\cdots,(-1)^k x_{k}).
\]

All the symbols in this article will be 
assumed to be smooth, bounded and with bounded derivatives.

We remark that our notation is slightly nonstandard because of the alternation of complex conjugates, which is consistent with the set-up of this paper. Another important remark is that, for $k$-linear forms, the cases of odd $k$, respectively even $k$ play different roles here, as follows:

\medskip

i) The $2k+1$ multilinear forms will be thought of as functions, e.g. those which appear 
in some of our evolution equations.

\medskip

ii) The $2k$ multilinear forms will be thought of as densities, e.g. which appear 
in some of our density-flux pairs.

\medskip
Correspondingly,  
to each $2k$-linear form $L$ we will associate
a $2k$-linear functional $\bL$ defined by 
\[
\bL(u_1,\cdots,u_{2k}) = \int_\R L(u_1,\cdots,\bu_{2k})(x)\, dx,
\]
which takes real or complex values.
This may be alternatively expressed 
on the Fourier side as 
\[
\bL(u_1,\cdots,u_{2k}) = (2\pi)^{1-k} \int_{D}
\ell(\xi_1,\cdots,\xi_{2k})
\hat u_1(\xi_1) \bar{\hat u}_2(\xi_2) \cdots  
\bar{\hat u}_{2k}(\xi_{2k})d\xi_1 \cdots d\xi_{2k-1},
\]
where, with alternating signs, the diagonal $D_0$ is given by
\[
D_0 = \{ 0 = \xi_1-\xi_2 + \cdots \}.
\]
Note that in order to define the multilinear functional $\bL$ we only need to know the symbol $\ell$ on $D_0$. There will be however 
more than one possible smooth extension of 
$\ell$ outside $D_0$. This will play a role in our story later on.

\subsection{Separation of variables}

In this section we will explain a tool we will rely on heavily in our 
estimates. In particular, we will often want to convert multilinear 
estimates into product estimates for some multilinear form $L$ of 
either even or odd multiplicity:
\[\|L(u_1, \ldots, u_k)\|_{L^1_x} 
\lesssim 
\sup_{x_1,\ldots,x_k}\|u_1(x - x_1)
    \cdots  u_k(x - x_k) \|_{L^1_x}
\]

\begin{lemma}\label{lem:sep}
    Let $\ell$ be the symbol associated to a $n$ linear form $L$ and suppose that 
    $\ell$ is smooth on the dyadic scale with some size $M$
    near $\xi_1 \sim \dyad^{j_1}, \ldots \xi_n \sim \dyad^{j_n}$. That is 
    \begin{equation}\label{eq:smooth_on_scales}
        |p_{j_1}(\xi_1) \cdots p_{j_n}(\xi_n)
        \partial^\alpha \ell(\xi_1, \ldots, \xi_n)| 
        \leq M\dyad^{-\alpha_1 j_1 - \cdots - \alpha_nj_n}
    \end{equation}
for all $\abs{\alpha} \leq n + 2.$
Then the kernel of the multilinear form $L$ restricted to 
    $\xi_1 \sim \dyad^{j_1}, \ldots \xi_n \sim \dyad^{j_n}$
    given by
$$\check \ell(x_1,\ldots, x_n) = \int e^{ix_1\xi_1 + \cdots + ix_n\xi_n} 
    \ell(\xi_1,\ldots, \xi_n)p_{j_1}(\xi_1) \cdots p_{j_n}(\xi_n) 
    d\xi_1\cdots d\xi_n$$
has 
$$\int \abs{\check \ell(x_1, \ldots, x_n)} dx_1\cdots dx_n \lesssim M$$
where the constant depends only on $n$ and not on $M$ or the $j_i$'s.

Further, any multilinear form whose kernel is integrable, such as the ones 
    satisfying \eqref{eq:smooth_on_scales} after localization, has the estimate
\begin{align*}
    \|L(u_1,\bar u_2, u_3, \ldots )\|_{L^q_x}
    \lesssim M \sup_{x_1,\ldots,x_n}\|u_1(x - x_1)\bar u_2 (x+x_2) 
    u_3(x - x_3)\cdots \|_{L^q_x}.
\end{align*}
    In particular, multilinear forms which satisfy \eqref{eq:smooth_on_scales} 
    satisfy
\begin{align*}
    \|L(P_{j_1}u_1,P_{j_2}\bar u_2, P_{j_3}u_3, \ldots )\|_{L^q_x}
    \lesssim M \sup_{x_1,\ldots,x_n}\|P_{j_1}u_1(x - x_1)
    P_{j_2}\bar u_2 (x+x_2) P_{j_3} u_3(x - x_3)\cdots \|_{L^q_x}.
\end{align*}
\end{lemma}
\begin{remark}
    In Section \ref{sec:unbalanced_corrections} we will consider 
    arbitrarily high order multilinear forms of the form 
    \[
        L_1(L_2(\cdots(L_m(u_1, \cdots, u_{k_m}), u_{k_m+1}, 
        \cdots u_{k_m + k_{m-1}}) \cdots u_{k_m + \cdots k_2})
        \cdots u_{k_m + \cdots + k_1})
    \]
    where each $L_k$ is at most of fixed finite order $N$
    satisfying \eqref{eq:smooth_on_scales}, and $m$ is potentially 
    large. Thus, we must analyze the dependence of the constant on $m$. 

    Fortunately, the symbol corresponding to a composition of this form 
    corresponds to a product of the symbols:
    \[
        \ell_m(\xi_1, \cdots, \xi_{k_m})
        \ell_{m-1}(\xi_1 - \xi_2 + \cdots \pm \xi_{k_m}, 
        \xi_{k_m+1}, \cdots, \xi_{k_m + k_{m-1}}) \cdots.
    \]
    Now, the kernel corresponding to this product is the convolution 
    of the individual symbols, each of which satisfies \eqref{eq:smooth_on_scales}
    after a linear change of variables. Further the convolution of $L_1$ 
    kernels is bounded, and so the growth rate of the constant will be at most
    $C(N)^m$.

    This will be sufficient in Section \ref{sec:unbalanced_corrections}.
\end{remark}
\begin{proof}
The proof follows from non-stationary phase. We expect $\check \ell$ 
    to be localized in a $\dyad^{-j_1} \times \cdots \times \dyad^{-j_n}$ 
    sized box near the origin because of \eqref{eq:smooth_on_scales}. 
    Thus, we estimate $\check \ell$ inside the box by 
the $L^\infty$ norm of $\ell$ and by using polynomial decay outside of the box.
We see
\[
\abs{\check \ell} = 
    \abs{\int e^{ix_1\xi_1 + \cdots + ix_n\xi_n} \ell(\xi_1,\ldots, \xi_n)
    p_{j_1}(\xi_1) \cdots p_{j_n}(\xi_n) d\xi_1\cdots d\xi_n} \lesssim 
    M\dyad^{j_1 + \cdots + j_n},
\]
and 
\begin{align*}
    \abs{(\dyad^{2j_1}x_1^2 + \cdots + \dyad^{2j_n}x_n^2)^\alpha \check \ell} 
    &= \abs{\int e^{ix\cdot \xi} (\dyad^{2j_1}\partial_1^2 + \cdots + 
    \dyad^{2j_n}\partial_n^2)^\alpha 
    \Big(\ell(\xi_1,\ldots, \xi_n)p_{j_1}(\xi_1) \cdots p_{j_n}(\xi_n) \Big)d\xi}\\ 
    &\lesssim M\dyad^{j_1 + \cdots + j_n}.
\end{align*}
    Then, letting $z_k= \dyad^{j_k}x_k$
\begin{align*}
\int\abs{\check \ell(x)}dx &\leq \dyad^{-j_1} 
    \cdots \dyad^{-j_n} 
    \left(\int_{\abs{z} \leq 1} \abs{\check \ell}dz + 
    \int_{\abs{z} \geq 1} \frac{\abs{(\dyad^{2j_1}x_1^2 + \cdots + \dyad^{2j_n}x_n^2)^
    \alpha \check \ell}}{\abs{z}^{2\alpha}}dz\right)\\
\lesssim M
\end{align*}
as long as $2\alpha > n.$
\end{proof}

In our estimates we will commonly use Lemma \ref{lem:sep} by restricting 
a multilinear form to dyadic intervals to replace multilinear estimates 
with dyadic sums over localized product estimates. For example 
\begin{align*}
    \|C(u, \bar u, u)\|_{L^q}
    &\leq \sum_{j_1, j_2, j_3} \|C(P_{j_1} u, P_{j_2} \bar u, P_{j_3} u)\|_{L^q}\\
    &\lesssim \sum_{j_1, j_2, j_3} \sup_{x_1,x_2, x_3}
    \|P_{j_1} u(\cdot - x_1) P_{j_2} \bar u(\cdot + x_2) P_{j_3} u(\cdot - x_3)\|_{L^q}
\end{align*}

The more general idea of integrable kernels 
will be invaluable when we study the unbalanced part of 
the nonlinearity in Section~\ref{sec:unbalanced_corrections} since 
symbols there will live on several scales where it will not always be 
straightforward to apply Lemma~\ref{lem:sep}.


\subsection{Cubic interactions in Schr\"odinger flows}
Given three input frequencies $\xi_1, \xi_2,\xi_3$ for 
our cubic nonlinearity, the output will be at frequency 
\[
\xi_4 = \xi_1-\xi_2+\xi_3.
\]
This relation can be described in a more symmetric fashion as 
\[
\Delta^4 \xi = 0, \qquad \Delta^4 \xi := \xi_1-\xi_2+\xi_3-\xi_4 .
\]
This is a resonant interaction if and only if we have a similar relation for the associated time frequencies, namely 
\[
\Delta^4 \xi^2 = 0, \qquad \Delta^4 \xi^2 := \xi_1^2-\xi_2^2+\xi_3^2-\xi_4^2 .
\]
Hence, we define the resonant set in a symmetric fashion as 
\[
\calR := \{ \Delta^4 \xi = 0, \ \Delta^4 \xi^2 = 0\}.
\]
It is easily seen that this set may be characterized as
\[
\calR = \{ \{\xi_1,\xi_3\} = \{\xi_2,\xi_4\}\}.
\]

Further, we will want to characterize how transverse the interaction is. That is 
the largest separation between the frequencies. We will use the notation 
\[
    (\xi_o - \xi_e)^2 := (\xi_1 - \xi_2)(\xi_4 - \xi_3) + 
    (\xi_1 - \xi_4)(\xi_2 - \xi_3).
\]
Note that $(\xi_o - \xi_e)^2$ measures the square distance 
of a point in $\{\Delta^4 \xi = 0\}$ to the ``diagonal'', where all four frequencies 
are exactly equal.


\section{Local well-posedness }
\label{s:local}
In this section we will prove a large data $L^2$ result which 
will be the starting point of our bootstrap argument in Section
\ref{s:boot}. The main reason we reprove this result is to establish 
suitable frequency envelope 
controlled $L^\infty_tL^2_x$, $L^6_{t,x}$, and bilinear $L^2_{t,x}$ 
bounds which we will need for the global result. We have the following theorem

\begin{theorem}\label{thm:lwp}
    Let $\du_0 \in L^2_x$ with $\|\du_0\|_{L^2_x} = M$. Then there exists a time $T(M)$ and 
    a unique solution $u \in C^0([0,T], L^2) \cap L^6_{t,x}$ to \eqref{nls}.
    In addition, suppose 
    \[
        \|\du_0\|_{H^s} = \epsilon
    \]
    and let
    $c_j$ be a normalized one sided slowly varying 
    frequency envelope such that 
    \[
        \|P_j \du_0\|_{L^2} \lesssim \epsilon\dyad^{-sj} c_j
    \]
    and 
    \[
        \|\epsilon\dyad^{-sj}c_j\|_{\ell^2_j} \sim M.
    \]
    (Note such an envelope always exists by Proposition \ref{prop:freq_envelopes}.)
    Then we also have the following frequency localized 
    bounds on $[0,T]$:
    \[
        \|P_j u\|_{L^\infty_t L^2_x} \lesssim \epsilon \dyad^{-sj} c_j
    \]
    \[
        \|P_j u\|_{L^6_{t,x}} \lesssim \epsilon \dyad^{-sj} c_j
    \]
    \[
        \|\partial(P_j u P_k \bar u(\cdot + x_0))\|_{L^2_{t,x}} 
        \lesssim (\dyad^{j/2} + \dyad^{k/2})\epsilon^2 
        \dyad^{-sj}\dyad^{-sk}c_j c_k,
    \]
uniformly in $x_0 \in \R$.
\end{theorem}
\begin{proof}
    The existence follows from a standard contraction using Strichartz
    estimates and a Coifman-Meyer type estimate for the trilinear 
    form $C$. We write the Duhamel form of \eqref{nls}:
    \[ 
        \Psi(u) = e^{it\Delta}\du_0 -i \int_0^t e^{i(t-s)\Delta} C(u,\bar u , u)ds
    \]
    Then 
    \[
        \|\Psi(u)\|_{L^\infty_t L^2_x \cap L^6_{t,x}}
        \leq \|\du_0\|_{L^2} + \|C(u,\bar u, u)\|_{L^1_t L^2_x}
        \leq M + T^{1/2}\|u\|_{L^\infty_t L^2_x \cap L^6_{t,x}}^3.
    \]
    And the contraction follows from the trilinearity of $C$ and 
    by choosing $T$ small depending on $M$.

    \bigskip

    For the frequency localized estimates we follow the Picard iterates.
    Let $u^0 = e^{it\Delta}\du_0$ and 
    $u^n = \Psi^n(u_0)$.
Then the standard energy estimate for the linear Schr\"odinger 
    equation gives
    \[\|P_j u^0\|_{L^2_x} = \|P_j \du_0\|_{L^2} \lesssim \epsilon \dyad^{-sj}c_j; \]
    the Strichartz estimate for the linear Schr\"odinger equation gives
    \[
    \|P_j u^0\|_{L^6_{t,x}} \lesssim\|P_j \du_0\|_{L^2} \lesssim \epsilon \dyad^{-sj}c_j; 
    \]
    and bilinear estimates for the linear Schr\"odinger equation 
    give 
    \[
    \|\partial_x \left(P_j u^0P_k \ol{u^0(\cdot +x_0)}\right)\|_{L^2_{t,x}} 
    \lesssim   (\dyad^{j/2} + \dyad^{k/2})\|P_j \du_0\|_{L^2}\|P_k \du_0\|_{L^2} \lesssim 
       \epsilon^2 (\dyad^{j/2} + \dyad^{k/2})\dyad^{-sj}\dyad^{-sk}c_jc_k. 
    \]
    For higher iterates, we simply use induction. Assume we have 
    our estimates for $u$. We have by
    Strichartz estimates that 
    \begin{align*}
        \|P_j \Psi(u)\|_{L^\infty_t L^2_x \cap L^6_{t,x}}
        &\lesssim \|P_j \du_0\|_{L^2_x} + \|P_j C(u,\bar u, u)\|_{L^1_tL^2_x}\\
        &\lesssim \|P_j \du_0\|_{L^2_x} + T^{1/2}\|P_j C(u,\bar u, u)\|_{L^2_{t,x}}
    \end{align*}
    For the bilinear estimate we have, setting $v = u(\cdot + x_0)$
    \begin{align*}
        \partial_x(P_j \Psi(u)P_k \ol{\Psi(v)})
        =&  \ \partial_x(e^{it\Delta}P_j \du_0 e^{-it\Delta} P_k\overline{\dv_0})
        + i \partial_x(e^{it\Delta}P_j \du_0 \cdot \int_0^t e^{-i(t-s)\Delta} 
            P_k\ol{C(v,\bar v, v)} ds)\\
        & - i \partial_x(e^{-it\Delta}P_k \ol{\dv_0} \cdot \int_0^t e^{i(t-s)\Delta} 
            P_jC(u,\bar u, u) ds)\\
        &+ \partial_x(\int_0^t e^{i(t-s)\Delta} 
            P_jC(u,\bar u, u) ds \cdot \int_0^t e^{-i(t-s)\Delta} 
            P_k\ol{C(v,\bar v, v)} ds)
    \end{align*}
    so that by the standard bilinear estimates for the linear Schr\"odinger 
    operator
\begin{align*}
        \|\partial_x(P_j \Psi(u)P_k \ol{\Psi(v)})\|_{L^2_{t,x}}
        \lesssim & \ (\dyad^{j/2} + \dyad^{k/2})( \|P_j \du_0\|_{L^2}
        + \|P_j C(u,\bar u, u)\|_{L^1_tL^2_x})
        \\ & \ ( \|P_k \du_0\|_{L^2}
        +\|P_k C(u,\bar u, u)\|_{L^1_tL^2_x})
        \\
        \lesssim & \ (\dyad^{j/2} + \dyad^{k/2})( \|P_j \du_0\|_{L^2}
        + T^\frac12 \|P_j C(u,\bar u, u)\|_{L^2_{x,t}})
        \\ & \ ( \|P_k \du_0\|_{L^2}
        + T^\frac12 \|P_k C(u,\bar u, u)\|_{L^2_{x,t}})
\end{align*}        
    
By choosing $T$ small compared to $M$, we can close the induction step by showing that
\begin{equation}
    \|P_j C(u,\bar u, u)\|_{L^2_{t,x}} \lesssim  M^2 \epsilon \dyad^{-sj} c_j
\end{equation}    
Expanding the inputs of $C$, it will suffice to estimate each term 
separately with an off-diagonal gain 
\begin{equation}\label{tri-L2}
    \|P_j C(P_{j_1}u,P_{j_2}\bar u, P_{j_3} u)\|_{L^2_{t,x}} \lesssim  M^2 \epsilon \dyad^{-sj} c_j \dyad^{\gamma(j_{min}-j_{max})}
\end{equation}
where $j_{min} = \min\{j, j_1,j_2,j_3\}$,  $j_{max} = \max\{j, j_1,j_2,j_3\}$, and $\gamma > 0$.   
    Assuming the trilinear bound \eqref{tri-L2}, expanding the square of the
$L^2$ norm and summing over the frequency triples controls the integral: 
    \begin{align*}
        \|P_j C(u,\bar u, u)\|_{L^2_{t,x}}^2 &\lesssim 
        \sum_{j_1, \ldots, j_6} \int P_j(P_{j_1}uP_{j_2}\bar uP_{j_3}u)
        P_j(P_{j_4}\bar uP_{j_5} uP_{j_6}\bar u) \, dx dt 
        \lesssim M^4 \epsilon^2 \dyad^{-2sj} c_j^2
    \end{align*}
    Note that for a particular term in the above sum to be non-zero
    it must be that the alternating sum of the  three input  frequencies lands in the region corresponding to 
    $\dyad^j$. In particular, we will consider three cases for each 
    $P_j(P_{j_1}uP_{j_2}\bar uP_{j_3}u)$. 
    \begin{itemize}
        \item All three frequencies are comparable to $j$
        \item One frequency is comparable to $j$, one is much less than $j$ 
        and the third is comparable or less than $j$.
        \item The highest frequency is much larger than $j$, in which case the 
        second highest must be comparable, and the highest and the lowest must be separated  by the high frequency.
    \end{itemize}
    Throughout the following argument we will use the 
    fact that 
    \begin{equation}\label{use-L2}
        \epsilon \dyad^{-sj} c_j \lesssim M
    \end{equation}
    where it is convenient, which in practice will be for the 
two lowest frequencies.

\medskip

    In the first case we simply use three $L^6_{t,x}$ estimates, and \eqref{use-L2} twice,
    \begin{align*}
        \|P_j(P_{j_1}uP_{j_2}\bar u P_{j_3}u)\|_{L^2_{t,x}} 
        &\lesssim M^2 \epsilon\dyad^{-sj} c_j.
    \end{align*}

    \medskip
    
    For the second case we order the frequencies $j_1 < j_2 < j_3$
    where $j_3$ is close to $j$, and $j_1$ is smaller. We can estimate 
    the trilinear form in two ways. With one bilinear $L^2_{t,x}$ for the unbalanced pair $(j_1,j_3)$ and one $L^6_{t,x}$ bound, we get
   \begin{align*}
        \|P_j(P_{j_1}uP_{j_2}\bar u P_{j_3}u)\|_{L^\frac32_{t,x}} 
        &\lesssim M^2 \epsilon\dyad^{-sj} c_j \dyad^{-\frac12 j}.
    \end{align*} 
With two $L^6_{t,x}$ bounds for the frequencies $j_2,j_3$ and one 
$L^\infty_{t,x}$ bound for $j_1$ via Bernstein, we obtain
\begin{align*}
        \|P_j(P_{j_1}uP_{j_2}\bar u P_{j_3}u)\|_{L^3_{t,x}} 
        &\lesssim M^2 \epsilon\dyad^{-sj} c_j \dyad^{\frac12 j_1}.
    \end{align*} 
Interpolating the two, we obtain \eqref{tri-L2} with $\gamma = \frac14$.    
\medskip

Finally we consider the third case, where we assume again 
that the three frequencies are ordered $j_1 \leq j_2 \leq j_3$,
with $j_2$ and $j_3$ close and separation between $j_1$ and $j_3$. 
 We further split this into two cases based on how the lowest frequency $j_1$ compares to $j$. 

 \medskip
 
 If $j \leq j_1$, then we use a bilinear $L^2_{t,x}$ bound pairing the frequencies $j_1$ and $j_3$, and energy at frequency $j_2$, followed
 by Bernstein at frequency $j$,
\begin{align*}
        \|P_j(P_{j_1}uP_{j_2}\bar u P_{j_3}u)\|_{L^2_{t,x}} 
        &\lesssim
     \dyad^{\frac{j}2}    \|P_j(P_{j_1}uP_{j_2}\bar u P_{j_3}u)\|_{L^2_{t}L^1_x} \lesssim
        M^2 \epsilon\dyad^{-sj_3} c_{j_3} \dyad^{\frac12 (j-j_3)}.
    \end{align*} 
 Using the slowly varying condition to replace 
    \[c_{j_3} \lesssim c_j \dyad^{\delta(j_3-j)}\]
we obtain 
\[
\|P_j(P_{j_1}uP_{j_2}\bar u P_{j_3}u)\|_{L^2_{t,x}} \lesssim 
 M^2 \epsilon\dyad^{-sj} c_{j} \dyad^{(\frac12+s-\delta )(j-j_3)}.
\] 
which gives \eqref{tri-L2} with $\gamma = \frac12+s-\delta > 0$.

\medskip

Finally we consider the scenario $j_1 < j$. Here we replicate 
the strategy in the second case, obtaining on one hand
\begin{align*}
        \|P_j(P_{j_1}uP_{j_2}\bar u P_{j_3}u)\|_{L^\frac32_{t,x}} 
        &\lesssim M^2 \epsilon\dyad^{-sj_3} c_{j_3} \dyad^{-\frac12 j_3} \lesssim M^2 \epsilon\dyad^{-sj} c_{j} \dyad^{-\frac12 j} \dyad^{(\frac12+s-\delta)(j-j_3)} 
    \end{align*} 
and on the other hand    
\begin{align*}
        \|P_j(P_{j_1}uP_{j_2}\bar u P_{j_3}u)\|_{L^3_{t,x}} 
        &\lesssim M^2 \epsilon\dyad^{-sj_1} c_{j_1} \dyad^{\frac12 j_1}
        \lesssim M^2 \epsilon\dyad^{-sj} c_{j} \dyad^{\frac12 j}
        \dyad^{(\frac12-\delta)(j_1-j)} .
    \end{align*} 
Interpolating, we obtain \eqref{tri-L2} with $2\gamma = \min\{\frac12+s-\delta, \frac12-\delta\} > 0$.

    This concludes the proof of \eqref{tri-L2}, which closes all our estimates. Finally, since the projectors 
    $P_j$ are bounded on the relevant spaces, we may pass 
    to the limit to arrive at the desired bounds for the solution 
    $u$.

\end{proof}

\begin{remark}
    In the nonlinear argument, we will only be able to 
    close the global in time $L^6_{t,x}$ bound with the following estimate
    \[
        \|P_j u\|_{L^6_{t,x}} \lesssim \epsilon^{2/3} \dyad^{(1-4s)j/6} c_j^{2/3}.
    \]
    Here we remark that the local in time bound in Theorem \ref{thm:lwp} is strictly better 
    since $\epsilon < 1$, $\|c_j\|_{\ell^2} \sim 1$, and 
    \[
        -s < -s + \frac{1+2s}6 = \frac{1 - 4s}6.
    \]
\end{remark}


\section{Energy estimates and conservation laws}
\label{s:energy}

\subsection{Conservation laws for the linear problem}
We begin our discussion with the linear Schr\"odinger equation
\begin{equation}
i\partial_t u + \partial_x^2 u = 0, \qquad u(0) = \du_0.    
\end{equation}
For this we consider the following three conserved quantities, the mass
\[
\bM(u) = \int |u|^2 \,dx,
\]
the momentum 
\[
\bP(u) = - 2 \int \Im ( \bar u 
    \partial_x  u) \,dx,
\]
as well as the energy
\[
\bE(u) = 4 \int |\partial_x u|^2\, dx. 
\]

To these quantities we associate corresponding densities
\[
\bbM(u) = |u|^2, 
\qquad \bbP(u) = i ( \bar u \partial_x 
     u -  u \partial_x  \bar u),
 \] 
    \[\bbE(u) = - \bar u \partial_x^2 
    u  +  2  |\partial_x u|^2 -  u \partial_x^2 
\bar u.
\]

The choice of densities here is not entirely straightforward. Symmetry is clearly a criterion, but further motivation is provided by the conservation law computation,
\begin{equation}\label{df-lin}
\partial_t \bbM(u) = \partial_x \bbP(u), \qquad \partial_t \bbP(u) = 
    \partial_x \bbE(u).
\end{equation}
The symbols of these densities viewed as bilinear forms are
\[
\bbm(\xi,\eta) = 1 , \qquad \bbp(\xi,\eta) = -(\xi+\eta), 
\qquad \bbe(\xi,\eta) = (\xi+\eta)^2.
\]

In this work, we will also consider frequency localized versions 
of these symbols: 
\[
    \bbm_j(\xi,\eta) = p_j(\xi)p_j(\eta) , 
    \qquad \bbp_j(\xi,\eta) = -(\xi+\eta)p_j(\xi)p_j(\eta), 
    \qquad \bbe_j(\xi,\eta) = (\xi+\eta)^2p_j(\xi)p_j(\eta).
\]

Then a direct computation yields the density flux relations
\[
\frac{d}{dt} \bbM_j(u,\bar u) = \partial_x \bbP_j(u,\bar u), \qquad 
\frac{d}{dt} \bbP_j(u,\bar u) = \partial_x \bbE_j(u,\bar u).
\]


\subsection{Nonlinear density flux identities for the mass and momentum} \label{s:nonlin_den_flux}
Here we will develop the nonlinear counterpart of the above discussion focusing 
on the balanced part of the nonlinearity. We will see in Section \ref{s:Morawetz} 
how the balanced part of the nonlinearity, with the conservative assumption (H2s) and the defocusing condition (H3),
gives us access to the Strichartz norm $\|u\|_{L_{t,x}^6}$, following the strategy developed in \cite{IT-global}, \cite{IT-qnls}, \cite{IT-conjecture}.

For the remainder of the section we will consider the nonlinear equation 
\begin{equation}\label{eq:nls_bal}
    iv_t + v_{xx} = C^{\bal}(v, \bar v, v) + N^{\unbal}(v).
\end{equation}
Here we assume $C^\bal$ is a trilinear form whose symbol satisfies 
the assumptions \ref{h:c_smooth} - \ref{h:c_defocus} and is localized 
to the set where all three input frequencies are comparable to 
the output frequency: hence the 
name balanced.

The nonlinear term $N^{\unbal}$  will capture the rest of the nonlinearity 
as well as further source terms induced by corrections we will make in Section 
\ref{sec:unbalanced_corrections} all of which will 
be at trilinear or higher order in $v$. For now, we 
assume we will have good control over $N^{\unbal}(v)$.

\subsubsection{The modified mass} 
To motivate what follows, we begin with a simpler computation for the $L^2$ norm of 
a solution $v$ of \eqref{eq:nls_bal}: 
\begin{align*}
\frac{d}{dt} \| v\|_{L^2}^2 
&= \int - i C^{\bal}(v,\bar v,v) \cdot \bar v  + i v \cdot 
\ol{C^{\bal}(v,\bar v,v)} \, dx + \int -iN(v)\cdot \bar v + iv \cdot \ol{N(v)} \, dx\\
&:= \int C^{4,\bal}_m(v,\bar v, v, \bar v) \, dx + 
\int N^{\geq 4, \unbal}_m(v) \, dx.
\end{align*}

A-priori the symbol of the quartic form $C^{4,\bal}_m$, defined on the diagonal $\Delta^4 \xi = 0$,
is given by
\[
c^{4,\bal}_m(\xi_1,\xi_2,\xi_3,\xi_4) = 
- i c^{\bal}(\xi_1,\xi_2,\xi_3) 
+ i \bar c^{\bal}(\xi_2,\xi_3,\xi_4)
\]
We can further symmetrize and replace $c^{4,\bal}$ by 
\begin{align*}
    c^{4,\bal}_m(\xi_1,\xi_2,\xi_3,\xi_4) = \frac{i}2 \Big(
-  c^{\bal}(\xi_1,\xi_2,\xi_3) 
- c^{\bal}(\xi_1,\xi_4,\xi_3)
+    \bar c^{\bal}(\xi_2,\xi_3,\xi_4)
+   \bar c^{\bal}(\xi_2,\xi_1,\xi_4)
 \Big).
\end{align*}
\begin{remark}\label{r:real}
There are three symmetries satisfied by $c^{4}_{bal}$, it is (i) symmetric in $\xi_1,\xi_3$, (ii)  symmetric in $\xi_2,\xi_4$, and (iii) hermitian symmetric when switching the pairs $\xi_1,\xi_3$ and $\xi_2,\xi_4$,
\begin{equation}\label{sym-real}
c^{4,\bal}_m(\xi_2,\xi_1,\xi_4,\xi_3)
= \overline{c}^{4,\bal}_m(\xi_1,\xi_2,\xi_3,\xi_4).
\end{equation}
It is this last condition which ensures that the quartic functional $C^{4}_{bal}(u,\bu,u,\bu)$ is real valued.
\end{remark}

In particular we are interested in the behavior of $c^{4,\bal}_m(\xi_1,\xi_2,\xi_3,\xi_4)$ on the resonant set
\[
\calR = \{(\xi_1, \xi_2, \xi_3, \xi_4)\in \mathbb{R}^4 \, / \,\Delta^4 \xi = 0, \,  \Delta^4 \xi^2 = 0\} = \{ \{ \xi_1,\xi_3\} = \{\xi_2,\xi_4\} \}.
\]

On this set we compute
\begin{align*}
    c^{4,\bal}_m(\xi_1,\xi_1,\xi_3,\xi_3) &= \frac{i}2 \Big(
    - c^{\bal}(\xi_1,\xi_1,\xi_3) 
    - c^{\bal}(\xi_1,\xi_3,\xi_3)\\
&\qquad + \bar c^{\bal}(\xi_1,\xi_3,\xi_3)
+  \bar c^{\bal}(\xi_1,\xi_1,\xi_3)
 \Big)\\
& =  \Im( c^{\bal}(\xi_1,\xi_1,\xi_3) 
+ c^{\bal}(\xi_1,\xi_3,\xi_3)),
\end{align*}
which is symmetric in $\xi_1$ and $\xi_3$.
Then we observe that our (H2s) assumption on $C^{\bal}$ shows that this expression 
vanishes when 
$\xi_1 = \xi_3$ (when the frequencies are balanced). 
Thus, we have cancellation when $\xi_1 = \xi_3$ and by the symmetry 
in $\xi_1$ and $\xi_3$ we can expect to factor out two 
derivatives when 
$\xi_1 \neq \xi_3$ which will be amenable to bilinear estimates.

In particular since $c^{4,\bal}_m$ vanishes on the diagonal and we 
may take it to be symmetric in $\xi_1$ and $\xi_3$ as well 
as $\xi_2$ and $\xi_4$ 
means that we can make the smooth division
\[
    q^{4,\bal}_m(\xi_1,\xi_3) =  \frac{c^{4,\bal}_m(\xi_1,\xi_1,\xi_3,\xi_3)}{
    (\xi_1-\xi_3)^2}.
\]
Now if we set 
\[
    c^{4,\res}_m(\xi_1,\xi_2,\xi_3,\xi_4) = ((\xi_1 - \xi_2)(\xi_4-\xi_3) 
    + (\xi_1 - \xi_4)(\xi_2 - \xi_3))q^{4,\bal}_m(\xi_1,\xi_3)
\]
we see that $c^{4,\res}_m$ and $c^{4,\bal}_m$ agree on 
the resonant set $\calR$. The reason we have chosen to write this 
term in this way is because the 
symbol $((\xi_1 - \xi_2)(\xi_3-\xi_4) + (\xi_1 - \xi_4)(\xi_3 - \xi_2))$ 
exactly corresponds to the multiplier 
\[(\partial (v \bar v))^2\]
which will be controllable using bilinear $L^2_{t,x}$ estimates. To emphasize this 
we will write 
\[
    Q^{4,\bal}_m((\partial |v|^2)^2) = C^{4,\res}_m(v,\bar v, v ,\bar v)
\]
on the physical side to emphasize this connection. We will also use the notation
$(\xi_o - \xi_e)^2$ to refer to the symbol 
$(\xi_1 -\xi_2)(\xi_4 - \xi_3) + (\xi_1 - \xi_4)(\xi_2 - \xi_3)$ since 
the symbol consists of odd factors minus even factors.

Note that while $c^{4,\res}_m$ defined above was chosen to agree with 
$c^{4,\bal}_m$ on the resonant set $\calR$, it is not the unique symbol with 
this property off of the resonant set. In fact, 
later we will want to localize $c^{4,\res}_m$ to dyadic regions which 
will alter the specifics of the definition, see Proposition~\ref{prop:symbols}. 

If we define 
\[
    c^{4,\bal}_m = c^{4,\res}_m + c^{4,\rem}_m,
\]
we see that $c^{4,\rem}_m$ vanishes on the resonant set $\calR$. Thus, 
on the diagonal $\Delta^4\xi = 0$, we will be able to divide by $\Delta^4\xi^2$.
In particular 
we can smoothly divide (see Lemma~\ref{l:division})
\[
    b^{4,\bal}_m(\xi_1,\xi_2,\xi_3,\xi_4) = \frac{i c^{4,\rem}_m(\xi_1,\xi_2,\xi_3,\xi_4)}{\Delta^4 \xi^2}
\]
on $\Delta^4 \xi = 0$.

Note that there is an ambiguity between $q^{4,\bal}_m$ and $b^{4,\bal}_m$. In 
particular we may replace
\[
    b^{4,\bal}_m \to b^{4,\bal}_m + f (\xi_o - \xi_e)^2, \qquad
    q^{4,\bal}_m \to q^{4,\bal}_m + f \Delta^4\xi^2
\]
for any smooth $f$ on $\Delta^4\xi = 0$.

We now use $\bB^{4,\bal}_m$ as an energy correction. Then we obtain the modified
energy relation
\begin{equation}\label{eq:modified-mass}
    \begin{aligned}
        \frac{d}{dt} (\| v\|_{L^2_x}^2+ \bB^{4,\bal}_m(v,\bar v, v,\bar v) )  
        &= \bQ^{4,\bal}_m((\partial |v|^2)^2) + \bR^{\geq 6}_m(v)\\
        &+ \bN^{\geq 4,\unbal}_m(v),
    \end{aligned}
\end{equation}
where $\bR^{\geq 6}_m$ is given by
\begin{equation}\label{R6-m-bal} 
\begin{aligned}
    \bR^{\geq 6}_m(v) &= 
    \bB^{4,\bal}_m(-iC^{\bal}(v,\bar v, v), \bar v, v, \bar v) 
    + \cdots 
    + \bB^{4,\bal}_m(v, \bar v, v, i\ol{C^{\bal}(v,\bar v, v)})\\
    &+\bB^{4,\bal}_m(-iN^{\unbal}(v), \bar v, v, \bar v) 
    + \cdots 
    + \bB^{4,\bal}_m(v, \bar v, v, i\ol{N^{\unbal}(v)}).
\end{aligned}
\end{equation}

In \eqref{eq:modified-mass} the left hand side may be viewed as a modified
energy, while the right hand side
can potentially be estimated using the $L^6_{t,x}$ norm of $v$ and 
bilinear $L^2_{t,x}$ estimates. 

\subsubsection{ The modified mass and momentum density-flux pairs}

The key idea here is that, corresponding to the above modified mass, we also want to write a conservation law for an associated mass density
\begin{equation}\label{m-sharp}
\ms(v) = \bbM(v) + B^{4,\bal}_m(v,\bar v, v,\bar v).
\end{equation}
However, when doing this, we remark that the symbol of $B^{4,\bal}_m$ was previously 
defined only on the diagonal $\Delta^4 \xi = 0$, whereas in order for the above 
expression to be well defined we need to extend it everywhere. For the purpose 
of this computation we simply assume that we have chosen some smooth extension.
A more careful choice will be considered later in Subsection \ref{ss:choice}.

Now we compute 
\begin{align*}
    \partial_t \ms(v) &= \partial_x \bbP(v) 
    + i ( \Delta^4 \xi^2 B^{4,\bal}_m)(v, \bar v, v, \bar v) 
    + C^{4,\rem}_m(v,\bar v, v, \bar v) + Q^{4,\bal}_m((\partial |v|^2)^2)\\ 
    &+ R^{\geq 6}_m(v)
    + N_m^{\geq 4, \unbal}(v).
\end{align*}
By the choice of $B^{4,\bal}_m$, the symbol of the quartic term 
$c^{4,\rem}_m +  i  \Delta^4 \xi^2  b^{4,\bal}_m$ vanishes 
on the diagonal $\{\Delta^4 \xi = 0\}$, therefore we can express it smoothly in the form
\begin{equation*}
c^{4,\rem}_m +  i  \Delta^4 \xi^2 \,  b^{4,\bal}_m = i \Delta^4 \xi\, r^{4,\bal}_m. 
\end{equation*}
Unrolling the definition of $c^{4,\rem}$ we have 
\begin{equation}\label{choose-R4m}
c^{4,bal}_m +  i  \Delta^4 \xi^2 \,  b^{4,\bal}_m = i \Delta^4 \xi\, r^{4,\bal}_m
    + (\xi_o - \xi_e)^2 q^{4,\bal}_m. 
\end{equation}

Hence the above relation can be written in the better form
\begin{equation}\label{dens-flux-m}
\begin{aligned}
    \partial_t \ms(v) &= \partial_x (\bbP(v) + R^{4,\bal}_m(v,\bar v, v,\bar v)) +
    Q^{4,\bal}_m((\partial |v|^2)^2)\\ 
    &+ R^{\geq 6}_m(v) +
    N_m^{\geq 4, \unbal}(v).
\end{aligned}
\end{equation}
Note that while in an integral over all of $\R$ the derivative term vanishes, we 
keep track of $R^{4,\bal}_m$ for Interaction Morawetz estimates which have a spatial 
cutoff.

One may view here the relation \eqref{choose-R4m} as a division problem,
where $c^{4,\rem}_m$ vanishes on the resonant set $\calR$. The symbols 
$b^{4,\bal}_m$, $r^{4,\bal}_m$, and $q^{4,\bal}_m$ 
are not uniquely determined by the relation 
\eqref{choose-R4m}, as we can change them by 
\begin{align*}
    &b^{4,\bal}_m \to b^{4,\bal}_m + f \Delta^4 \xi + g (\xi_o - \xi_e)^2,\\
    &r^{4,\bal}_m \to r^{4,\bal}_m + f \Delta^4 \xi^2 + h (\xi_o - \xi_e)^2,\\
    &q^{4,\bal}_m \to q^{4,\bal}_m + g \Delta^4 \xi^2 - h \Delta^4\xi
\end{align*}
for any smooth $f,g,$ and $h$. However, this is the only 
ambiguity. In particular $r^{4,\bal}_m$ is uniquely determined on the 
set $(\xi_o - \xi_e)^2 = \Delta^4 \xi^2 = 0$ , $b^{4,\bal}_m$ is uniquely determined 
on the set $(\xi_o - \xi_e)^2 = \Delta^4 \xi = 0$ and $q^{4,\bal}_m$ is uniquely 
determined on the set $\Delta^4 \xi = \Delta^4\xi^2 = 0$.

\bigskip

One could carry out a similar computation for the momentum,
where the starting point is the relation
\[
    \partial_t \bbP(v) = \partial_x \bbE(v) + C^{4,\bal}_p(v,\bar v, v, \bar v) 
    + N^{\geq 4, \unbal}_p(v).
\]
Precisely, the symbol of $C^{4,\bal}_p$ is initially given by
\[
    c^{4,\bal}_p(\xi_1,\xi_2,\xi_3,\xi_4) =  i (\xi_1-\xi_2+\xi_3+\xi_4) c^{\bal}(\xi_1,\xi_2,\xi_3) - i 
    (\xi_1+\xi_2-\xi_3+\xi_4)\bar c^{\bal}(\xi_2,\xi_3,\xi_4).
\]
However, we can further symmetrize it and split it exactly as in the case of 
$C_m^{4,\bal}$. 
In particular we define 
\[
C^{4,\bal}_{p} = C^{4,\res}_p + C^{4,\rem}_p
\]
where the symbol of $C^{4,\res}_p$ captures the transversal part of the nonlinearity 
which we can control and is given by
\[
c^{4,\res}_p(\xi_1,\xi_2,\xi_3,\xi_4) = ((\xi_1 - \xi_2)(\xi_4-\xi_3) 
    + (\xi_1 - \xi_4)(\xi_2 - \xi_3)) \frac{c^{4,\bal}_p(\xi_1,\xi_1,\xi_3,\xi_3)}{
    (\xi_1-\xi_3)^2}.
\]
Then,
$C^{4,\rem}_p$ vanishes 
on the resonant set $\calR$, so it admits a 
(nonunique) representation of the form
\begin{equation}\label{choose-R4p}
c^{4,\rem}_p +  i  \Delta^4 \xi^2 b^{4,\bal}_p = i \Delta^4 \xi r^{4,\bal}_p. 
\end{equation}

Hence, as in the case of the mass, we define a quartic correction for the momentum density
\[
\ps(v) = \bbP(v) + B^{4,\bal}_p(v,\bar v, v,\bar v).
\]
This satisfies a conservation law of the form
\begin{equation}\label{dens-flux-p}
    \begin{aligned}
        \partial_t \ps(v) &= \partial_x (\bbE(v) + R^{4,\bal}_p(v,\bar v, v,\bar v)) 
        +Q^{4,\bal}_p((\partial|v|^2)^2)\\ 
        &+R^{\geq 6}_{p}(v)
        + N^{\geq 4, \unbal}_p(v) .
    \end{aligned}
\end{equation}

\bigskip


\subsection{Localized density-flux identities for mass and momentum}

In our analysis later on, we will not use density-flux pairs for global estimates, but instead we will use them only in a frequency localized setting. 

In previous work \cite{IT-global}, we were motivated to 
take very general localization symbols to cover arbitrary intervals. Here 
we will use dyadic intervals. Recall that by scaling the optimal 
size of these intervals would be $\dyad^{-2sj}$ when $\xi \sim \dyad^j$. Since 
$\dyad^{j} \geq \dyad^{-2sj}$ for all $j \geq 0$, we get slightly less control
over energy transfer between frequencies. However, this will dramatically 
simplify the exposition as we will not need to separate the analysis 
of energy transfer between one and several dyadic scales.

We use the Littlewood Paley projections $p_j$ from 
Section \ref{s:not}. Corresponding to 
this frequency projection we define 
\[
    \bbm_j(\xi,\eta) = p_j(\xi)p_j(\eta), \qquad \bbp_j(\xi,\eta) = -(\xi+\eta)
\bbm_j(\xi,\eta), \qquad \bbe_j(\xi,\eta) = (\xi+\eta)^2
\bbm_j(\xi,\eta),
\]
A direct computation yields the relation
\begin{equation}\label{eq:deriv_of_mass}
\partial_t \bbM_j(v) = \partial_x\bbP_j(v) +
    C_{m,j}^{4,\bal}(v) + N^{\geq 4, \unbal}_{m,j}(v),   
\end{equation}
where the symbol $C_{m,j}^{4,\bal}$ is given by 
\[
\begin{aligned}
    c_{m,j}^{4,\bal}(\xi_1,\xi_2,\xi_3,\xi_4) = -\frac{i}{2} 
    [ & \ c^{\bal}(\xi_1,\xi_2,\xi_3) \bbm_j(\xi_1-\xi_2 +\xi_3,\xi_4) 
    +c^{\bal}(\xi_1,\xi_4,\xi_3) \bbm_j(\xi_1-\xi_4 +\xi_3,\xi_2)
   \\ & - \bar{c}^{\bal}(\xi_2,\xi_3,\xi_4) \bbm_j(\xi_1, \xi_2-\xi_3+\xi_4)
   - \bar{c}^{\bal}(\xi_2,\xi_1,\xi_4) \bbm_j(\xi_3, \xi_2-\xi_1+\xi_4)].
\end{aligned}
\]
and $N^{\geq 4, \unbal}_{m,j}$ is defined to be 
\[N^{\geq 4,\unbal}_{m,j}(v) = -i P_jN^{\unbal}(v) P_j\bar v 
+ i P_j v P_j\ol{N^{\unbal}(v)}.\]
Consider the first term in $c^{4,\bal}_{m,j}$: 
$c^{\bal}(\xi_1,\xi_2,\xi_3) m_j(\xi_1-\xi_2+\xi_3,\xi_4)$.
Since $c^{\bal}$ is localized near where all input frequencies are equal
to the output frequency, then in fact all frequencies $\xi_1, \ldots, \xi_4$ 
are localized 
near $\dyad^j$. See Definition \ref{def:c_splitting} for the precise 
details. For now, we note that $C^{4,\bal}_{m,j}$ has 
all frequencies localized in the dyadic region corresponding to $\dyad^j$.

A similar identity applies in the case of the localized momentum,
where we simply replace the symbol $\bbm_j$ by $\bbp_j$.

Again, by 
the condition (H2s), the resonant parts of $c^{4,\bal}_{m,j}$
and $c^{4,\bal}_{p,j}$ are amenable to 
bilinear $L^2_{t,x}$ estimates. The remainder will vanish on the
resonant set $\calR.$ Overall we 
can represent it as in the division relation \eqref{choose-R4m}
\begin{equation}\label{eq:choose_R4mj}
    c^{4,\bal}_{m,j} +  i\Delta^4 \xi^2 b^{4,\bal}_{m,j} = 
    i \Delta^4 \xi r^{4,\bal}_{m,j}
    + (\xi_o - \xi_e)^2q^{4,\bal}_{m,j},
\end{equation}
as well as
\begin{equation}\label{eq:choose_R4pj}
    c^{4,\bal}_{p,j} +  i  \Delta^4 \xi^2 b^{4,\bal}_{p,j} = 
    i \Delta^4 \xi r^{4,\bal}_{p,j}
    +(\xi_o - \xi_e)^2q^{4,\bal}_{p,j},
\end{equation}
Then, defining $\ms_j$ and $\ps_j$ as before, 
\begin{equation}\label{ma-sharp}
    \ms_j(v) = \bbM_j(v) + B^{4,\bal}_{m,j}(v,\bar v, v,\bar v),
\end{equation}
\begin{equation}\label{pa-sharp}
    \ps_j(v) = \bbP_j(v) + B^{4,\bal}_{p,j}(v,\bar v, v,\bar v),
\end{equation}
we obtain density-flux identities akin to
\eqref{dens-flux-m},  namely 
\begin{equation}\label{eq:dens_flux_mj}
 \partial_t \ms_j(v) = \partial_x(\bbP_{j}(v)
    + R^{4,\bal}_{m,j}(v)) + Q^{4,\bal}_{m,j}(|\partial|v|^2|^2) 
    + R^{\geq 6}_{m,j}(v)
    + N^{\geq 4, \unbal}_{m,j}(v),
\end{equation}
and 
\begin{equation}\label{eq:dens_flux_pj}
 \partial_t \ps_j(v) = \partial_x(\bbE_{j}(v)
    + R^{4,\bal}_{p,j}(v)) + Q^{4,\bal}_{p,j}(|\partial|v|^2|^2)
    + R^{\geq 6}_{p,j}(v)
    + N^{\geq 4, \unbal}_{p,j}(v).
\end{equation}

To use these density flux relations we need to have appropriate 
bounds for our symbols:

\begin{proposition} 
\label{prop:symbols}
    Let $j \geq 0$. Then the relations \eqref{eq:choose_R4mj}
and \eqref{eq:choose_R4pj} hold with symbols $b^{4,\bal}_{m,j}$, 
$b^{4,\bal}_{p,j}$, $r^{4,\bal}_{m,j}$,  $r^{4,\bal}_{p,j}$,
$q^{4,\bal}_{m,j}$, and $q^{4,\bal}_{p,j}$
which can be chosen to have the following properties:

\begin{enumerate}[label=\roman*)]
    \item Support: $b^{4,\bal}_{m,j}$ and  
$b^{4,\bal}_{p,j}$, may be chosen to be supported on a region where 
all frequencies are in adjacent dyadic intervals which are comparable to 
$\dyad^j$.
$r^{4,\bal}_{m,j}$, $r^{4,\bal}_{p,j}$,
$q^{4,\bal}_{m,j}$, and $q^{4,\bal}_{p,j}$ will have support in a region 
where all frequencies are merely comparable to $\dyad^j$.

    \item Size and Regularity: 
\[
    |\partial^\alpha b^{4,\bal}_{m,j}| \lesssim \dyad^{-2j - |\alpha| j},
    \qquad 
    |\partial^\alpha b^{4,\bal}_{p,j}| \lesssim \dyad^{-j - |\alpha| j},
\]
\[
|\partial^\alpha r^{4,\bal}_{m,j}| \lesssim \dyad^{-j - |\alpha| j},
\qquad
|\partial^\alpha r^{4,\bal}_{p,j}| \lesssim \dyad^{- |\alpha| j},
\]
\[
    |\partial^\alpha q^{4,\bal}_{m,j}| \lesssim \dyad^{-2j - |\alpha| j},
    \qquad 
    |\partial^\alpha q^{4,\bal}_{p,j}| \lesssim \dyad^{-j - |\alpha| j}.
\]
\end{enumerate}
\end{proposition}

\begin{proof}
 This is easily done by applying Lemma~\ref{l:division} below,
 see also Remark~\ref{r:division}.

Here we remark on the $i$ factors in the relations \eqref{choose-R4m}, \eqref{eq:choose_R4mj}, \eqref{eq:choose_R4pj}. These are not present in the Lemma~\ref{l:division}, because they are immaterial there, and can be freely added in the above relations; 
but they are needed in \eqref{choose-R4m}, \eqref{eq:choose_R4mj}, \eqref{eq:choose_R4pj}, in order to insure that the
symbols $b^4$, $r^4$ and $q^4$  have the 
hermitian symmetry \eqref{sym-real}, and thus the associated quartic forms are real valued, see Remark~\ref{r:real}. 
 
\end{proof}

\subsection{The choice for the density-flux corrections} \label{ss:choice}

Here we consider the division problem in \eqref{eq:choose_R4mj}, and ask
what should be a good balance between the symbols $B_{m,j}^{4,\bal}$,
 $R_{m,j}^{4,\bal}$, and $Q^{4,\bal}_{m,j}$.
We recall that $b_{m,j}^{4,\bal}$ is uniquely determined on the diagonal 
$\Delta^4 \xi = 0$,
but we can choose it freely away from the diagonal.

In the next lemma we prove Proposition~\ref{prop:symbols} for $c^{4,\bal}_{m,j}$. However we will use this lemma again in Section~\ref{sec:unbalanced_corrections} so we state it in slightly greater generality.

\begin{lemma}\label{l:division}
Let $\lambda = \dyad^j \geq 1$ be an arbitrary dyadic scale
and let 
$c(\xi_1,\xi_2,\xi_3,\xi_4)$ be a smooth symbol with 
    the following properties. 
    $c$ is symmetric in $\xi_1, \xi_3$ and 
$\xi_2, \xi_4$; $c$  
vanishes on the diagonal $\xi_1=\xi_2=\xi_3=\xi_4$; 
    $c$ is supported where all frequencies are comparable to 
$\lambda$; and $c$ is smooth on scale $\lambda$: 
\[
|\partial^\alpha c| \lesssim_\alpha K \lambda^{-|\alpha|}.
\]
Then it admits a representation
\begin{equation}\label{c4-rep-lem}
c^4 =  \Delta^4 \xi\, r^4 - \Delta^4 \xi^2 \, b^4 + 
    (\xi_o - \xi_e)^2q^4,
\end{equation}
where $b^4$ is supported in a region where different 
frequencies are all in mutually adjacent dyadic intervals; 
    where $r^4$ and $q^4$ are 
    supported in the same region as $c$;
and all of these are 
smooth on scale $\lambda$ with size according to their divisions. Explicitly, 
we have the following size properties

\begin{equation} \label{rb4}
\begin{aligned}
    |\partial^\alpha  r^4|  \lesssim & \  \frac{K}{\lambda}\lambda^{-|\alpha|}, 
\qquad 
|\partial^\alpha b^4|  \lesssim & \  \frac{K}{\lambda^2}\lambda^{-|\alpha|},
\qquad
|\partial^\alpha q^4|  \lesssim & \  \frac{K}{\lambda^2}\lambda^{-|\alpha|}.
\end{aligned}   
\end{equation}    

\end{lemma}

Here and later in the paper by ``smooth on scale $\lambda$" we mean that the above functions and all their derivatives are bounded, with bounds as in \eqref{rb4}, and where the implicit constant is allowed to depend on $\alpha$, but not on anything else. As usual, only finitely many derivatives are needed in our analysis, but we do  not take the extra step of determining how many.

\begin{remark}\label{r:division}
Compared to the earlier article \cite{IT-global}, here we use a simpler version of the division Lemma,  where 
balanced corrections are localized to nearby dyadic intervals, as opposed to lattice type scale, see also Remark~\ref{r:Hs}.
\end{remark}

\medskip

\begin{proof}
To simplify the notation, we introduce new linear coordinates $(\eta_1, \eta_2, \eta_3, \eta_4)$ where
\[
\eta_1 = \frac{\Delta^4 \xi}{\lambda}, \quad \eta_2 = \frac{\xi_1+\xi_2-\xi_3-\xi_4}{\lambda}, \quad \eta_3 = 
\frac{\xi_1-\xi_2-\xi_3+\xi_4}\lambda, \]

For $\eta_4$ we can choose in a symmetric fashion
\[
\eta_4 = \frac{\xi_1+\xi_2+\xi_3+\xi_4}\lambda.
\]

Note that 
\[
    \frac12 \lambda^2 (\eta_1\eta_4 + \eta_2 \eta_3) = \Delta^4 \xi^2,
\]
\[
    \frac14\lambda^2(\eta_2 - \eta_1)(\eta_2 + \eta_1) = (\xi_2-\xi_3)(\xi_1-\xi_4),
\]
and
\[
    \frac14\lambda^2(\eta_3 - \eta_1)(\eta_3 + \eta_1) = (\xi_4-\xi_3)(\xi_1-\xi_2).
\]
In these coordinates, the condition that $c$  vanishes on the diagonal
becomes the 
condition that 
$c$  vanishes on $\eta_1 = \eta_2 = \eta_3 = 0$. 
The condition that $c$ is symmetric in $\xi_1,\xi_3$ and $\xi_2,\xi_4$ becomes the 
condition that 
\[
    c(\eta_1,\eta_2,\eta_3,\eta_4) = c(\eta_1,-\eta_2,-\eta_3,\eta_4) = 
    c(\eta_1,\eta_3,\eta_2,\eta_4) = 
    c(\eta_1,-\eta_3,-\eta_2,\eta_4).
\]
Further, the condition that $c$ is smooth and compactly 
supported in frequencies comparable to $\lambda$ in these coordinates 
becomes a parallel statement with $\lambda$ replaced by $1$.

First, we eliminate the part of $c$ which is not localized to 
mutually adjacent dyadic intervals. We do this by defining an even 
cutoff $\chi(x)$ which is 1 when $|x|>1$, 0 when $|x| < 1/2$ and smooth 
on scale one. Then on the support of $\chi(16 \eta_1)$, 
$|\eta_1| > 1/32$, and so division by $\eta_1$ is favorable. We do similarly
for $\eta_2$ and $\eta_3$:
\[
    c(\eta_1,\eta_2,\eta_3,\eta_4) = 
    (1-\chi(16\eta_1)+\chi(16\eta_1))(1-\chi(2\eta_2)+\chi(2\eta_2))
    (1-\chi(2\eta_3)+\chi(2\eta_3))c(\eta_1,\eta_2,\eta_3,\eta_4)
\]
Then we write 
    \begin{align*}
        c(\eta_1,\eta_2,\eta_3,\eta_4) &= 
    (1-\chi(16\eta_1))(1-\chi(2\eta_2))
    (1-\chi(2\eta_3))c(\eta_1,\eta_2,\eta_3,\eta_4)\\
        &+ \lambda \eta_1 r^{far}(\eta_1,\eta_2,\eta_3,\eta_4) + 
        \frac14\lambda^2 (\eta_2^2 + \eta_3^2-2\eta_1^2) q^{far}(\eta_1,\eta_2,\eta_3,\eta_4) 
    \end{align*}
    where 
    \[
        r^{far} = \frac{\chi(16\eta_1)}{\lambda\eta_1}
        c(\eta_1,\eta_2,\eta_3,\eta_4)
    \]
    \[
        q^{far} = 4(1-\chi(16\eta_1))\frac{\chi(2\eta_2)
        + \chi(2\eta_3)-\chi(2\eta_2)\chi(2\eta_3))}
        {\lambda^2(\eta_2^2 + \eta_3^2-2\eta_1^2)} c(\eta_1,\eta_2,\eta_3,\eta_4),
    \]
with the additional remark that $\eta_2^2+\eta_3^2-2\eta_1^2\sim\eta_2^2+\eta_3^2 \gtrsim 1$ within the support of $q^{far}$,
so that the last division above is smooth.

Note that on the support of $(1-\chi(16\eta_1))(1-\chi(2\eta_2))
    (1-\chi(2\eta_3))$ we have for instance 
\[
    |\xi_1 - \xi_2| = \frac{\lambda}2 |\eta_1 + \eta_3|
    \leq \frac{\lambda}{2}
\]
which means that $\xi_1$ and $\xi_2$ are in adjacent dyadic regions.
Similarly, every pair is in adjacent regions. So what remains 
has the desired localization for $b$.

For what remains we let $\psi$ be a cutoff which is 1 on the support of $c$ 
and smooth on scale 1.
We write 
\begin{align*}
    c(\eta_1,\eta_2,\eta_3,\eta_4) &= 
    \Big(c(\eta_1,\eta_2,\eta_3,\eta_4)-c(0,\eta_2,\eta_3,\eta_4)\Big)\psi\\
                                   &+\Big(c(0,\eta_2,\eta_3,\eta_4))
    -c(0,0,\eta_2+\eta_3,\eta_4)\Big)\psi\\
                                   &+ c(0,0,\eta_2+\eta_3,\eta_4)\psi.
\end{align*}
The first difference can be smoothly divided by $\eta_1$ since it vanishes when 
$\eta_1 =0$. The second difference can be smoothly divided by $\eta_2\eta_3$ since 
it vanishes when either $\eta_2=0$ or $\eta_3=0$ because of the symmetry property:
$c(0,\eta_2,0,\eta_4) = c(0,0,\eta_2,\eta_4)$.
The third term can be divided by $(\eta_2 + \eta_3)^2$ since $c(0,0,0,\eta_4)= 0$ 
and, in addition, the symmetries
$c(\eta_1,\eta_2,\eta_3,\eta_4)=c(\eta_1,\eta_3,\eta_2,\eta_4)
=c(\eta_1,-\eta_3,-\eta_2,\eta_4)$ make $c(0,0,\cdot,\eta_4)$ even, which
gives $\partial_3 c(0,0,0,\eta_4)=0$.
In particular, we set 
\[
    q^{near} 
    = \frac4{\lambda^2} \frac{c(0,0,\eta_2+\eta_3,\eta_4)}{(\eta_2 + \eta_3)^2}\psi,
\]
\[
    b^{near} = -\frac2{\lambda^2} \frac{c(0,\eta_2,\eta_3,\eta_4)
    - c(0,0,\eta_2+\eta_3,\eta_4)}{\eta_2\eta_3}\psi - q^{near},
\]
and 
\[
    r^{near} = \frac1{\lambda} \frac{c(\eta_1,\eta_2,\eta_3,\eta_4)
    - c(0,\eta_2,\eta_3,\eta_4)}{\eta_1}\psi + \frac12 \lambda \eta_4 b^{near}
    + \frac12 \lambda \eta_1 q^{near}.
\]
We then set 
\[
    b = (1-\chi(16\eta_1))(1-\chi(2\eta_2)) (1-\chi(2\eta_3)) b^{near}
\]
\[
    r = (1-\chi(16\eta_1))(1-\chi(2\eta_2)) (1-\chi(2\eta_3)) r^{near}
    + r^{far}
\]
\[
    q = (1-\chi(16\eta_1))(1-\chi(2\eta_2)) (1-\chi(2\eta_3)) q^{near}
    + q^{far}
\]

This ensures that
\begin{align*}
    \Delta^4 \xi\, r - &\Delta^4 \xi^2 \, b + 
\left[(\xi_1-\xi_4)(\xi_2-\xi_3) + (\xi_1-\xi_2)(\xi_4-\xi_3)\right]q\\
&= \lambda\eta_1 r - \frac12\lambda^2(\eta_2\eta_3+\eta_1\eta_4) b + \frac14\lambda^2(\eta_2^2 + \eta_3^2 - 2\eta_1^2)q\\
&= c(\eta_1,\eta_2,\eta_3,\eta_4)\psi - c(0,0,\eta_2+\eta_3,\eta_4)\psi
+ \frac14 \lambda^2(\eta_2+\eta_3)^2q\\
&= c(\eta_1,\eta_2,\eta_3,\eta_4)
\end{align*}

The support property follows from the cutoff $\psi$.
And, since we chose coordinates at scale $\lambda$, we can see more easily the size properties \eqref{rb4}.
This is because $c^4$ is smooth on the unit scale with respect to 
$\eta_1, \ldots, \eta_4$ and thus so are the 
divisions. Further, by the chain rule, we get a factor of $\lambda^{-1}$ for each $\xi_1, \ldots, \xi_4$ derivative,
which give the desired bounds.

\end{proof}



\section{Interaction Morawetz identities}
\label{s:Morawetz}

\subsection{The linear Schr\"odinger equation} 
\label{s:Morawetz-lin}
The interaction Morawetz inequality aims to capture the fact that the 
momentum moves to the right faster than the mass. 
Here the left/right symmetry is broken
due to the sign choice which is implicit in the choice of the momentum.

\subsubsection{A global computation} To warm up,  
we start with two solutions $u$ and $w$ for the linear 
Schr\"odinger equation. To these we associate
the interaction functional 
\[
\bI(u,w) = \int_{x > y} \bbM(u)(x) \bbP(w)(y) - \bbP(u)(x) \bbM(w)(y)  \,dx dy,
\]
and compute $d\bI/dt$ using the conservation laws \eqref{df-lin}. We have
\[
\begin{aligned}
\frac{d}{dt} \bI(u,w) =  & \ \int_{x > y} \partial_x \bbP(u)(x)
\bbP(w)(y) +  \bbM(u)(x) \partial_y  \bbE(w)(y)
\\ & \ \ \ \ 
-\partial_x \bbE(u)(x)
\bbM(w)(y) - \bbP(u)(x) \partial_y  \bbP(w)(y)\,
dx dy 
\\
= &  \int \bbM(u) \bbE(w) + \bbM(w) \bbE(u) 
- 2 \bbP(u)\bbP(w) \,dx:= 
\int J^4(u,\bu,w,\bw)\, dx
.\end{aligned}
\]
Here  $J^4$ can be chosen\footnote{ Recall that 
a-priori the symbol of $j^4$ is only determined uniquely on the diagonal $\Delta^4 \xi = 0$.}
to have  symbol
\[
j^4(\xi_1,\xi_2,\xi_3,\xi_4) =    4(\xi_1-\xi_4)(\xi_2-\xi_3).
\]
This is because of the following computation on the diagonal $\Delta^4 \xi = 0$:
\[
(\xi_1+\xi_2)^2 + (\xi_3+\xi_4)^2 - 2 (\xi_1+\xi_2)(\xi_3+\xi_4) = (\xi_1+\xi_2-\xi_3 - \xi_4)^2 
=  4 (\xi_1 - \xi_4)(\xi_2 - \xi_3).
\]
Thus we have  the positivity
\[
J^4(u,\bu,w,\bw) = 4 |\partial_x (u \bw)|^2.
\]
The above computation is classically done using integration by parts, see \cite{PV}. 
However, it is more interesting to do it at the symbol level because we want to 
apply it in a more general context. Classically this is done with $u=w$, but here 
we find it convenient to break the symmetry. Primarily, our $w$'s will be spatial 
translations of $u$.


\subsubsection{A frequency localized bound}
Here we localize the solutions $u$ and $w$ at first to the 
same dyadic frequency corresponding to $\dyad^j$.
Consider the following frequency localized functional:
\begin{equation}\label{Ia-def}
\bI_j(u,w) =  \int_{x > y} \bbM_j(u)(x) \bbP_j(w)(y)  -   \bbP_j(u) (x) \bbM_j(w)(y) \, dx dy.
\end{equation}

As above, its time derivative is 
\[
\frac{d}{dt} \bI_j(u,w) =  
\bJ^4_j(u,\bu,w,\bw),
\]
where $J_j^4$ has symbol
\[
j_j^4(\xi_1,\xi_2,\xi_3,\xi_4) = 
4p_j(\xi_1)p_j(\xi_2)p_j(\xi_3)p_j(\xi_4)(\xi_1-\xi_4)(\xi_2-\xi_3).
\]
Then $J_j^4$'s integral is a localized version of the unlocalized version:
\[
    \bJ_j^4 (u,w) = 4 \int |\partial(P_j(u)P_j(\bw))|^2 \, dx.
\]

\subsubsection{ Interaction Morawetz for separated velocities} \label{s:AB}

Here we instead choose two dyadic frequencies $j, k$. In the nonlinear 
problem we will choose $j$ and $k$ to be separated, but for the linear 
problem separation will not come into play.
We write the interaction Morawetz functional 
\[
\bI_{jk} = \int_{x > y} \bbM_j(u)(x) \bbP_k(w)(y)    - \bbP_j(u)(x) \bbM_k(w)(y) \, dx dy ,
\]

Then we compute
\[
\frac{d}{dt} \bI_{jk} = 
\int \bbM_j(u) \bbE_k(w) + \bbE_j(u) \bbM_k(w) - 2  \bbP_j(u) \bbP_k(w) \, dx :=
\bJ^4_{jk}(u,\bu,w,\bw), 
\]
where $\bJ^4_{jk}$ has symbol
\[
j^4_{jk}(\xi_1,\xi_2,\xi_3,\xi_4) = 
4p_j(\xi_1)p_j(\xi_2) p_k(\xi_3)p_k(\xi_4)(\xi_1-\xi_4)(\xi_2-\xi_3).
\]
We can write $\bJ^4_{jk}$ as
\[
    \bJ^4_{jk} = 4\int |\partial(P_j(u)P_k(\bar w))|^2 \, dx.
\]

When $j \ll k$ for instance, then
$(\xi_1-\xi_4)$ and $(\xi_2-\xi_3)$ are both size $\dyad^k$ and 
\[
    \bJ_{jk} \sim \dyad^{2k} \|P_j u  P_k \bw\|_{L^2_x}^2.
\]

\subsection{Nonlinear interaction Morawetz estimates}
Here we consider the same interaction Morawetz functional as above, 
but now apply it to (two) solutions $v$ and $w$ of the nonlinear equation 
\eqref{eq:nls_bal}. 

\subsubsection{ A simple case }
As a starting point, here we consider density-flux pairs as in \eqref{dens-flux-m}, 
\eqref{dens-flux-p} to which we associate the nonlinear interaction functional 
\begin{equation}
\bI(v,w) = \iint_{x > y} \ms(v)(x) \ps(w)(y) - \ps(v)(x) \ms(w)(y) \, dx dy .    
\end{equation}
Using the density-flux relations we obtain
\begin{equation}
\frac{d\bI}{dt} = \bJ^4 + \bJ^6 + \bJ^8 + \bK^{\geq 6},
\end{equation}
where $\bJ^4$ is the same as above, while $\bJ^6$ and $\bJ^8$ are given by
\begin{equation}
\begin{aligned}
    \bJ^6(v,w) =  \int & \ \bbM(v) R^{4,\bal}_{p}(w)+ B^{4,\bal}_m(v) \bbE(w)
    - \bbP(v)B^{4,\bal}_p(w)- R^{4,\bal}_m(v)\bbP(w) +
\\ & \ 
    \bbM(w) R^{4,\bal}_{p}(v)+ B^{4,\bal}_m(w) \bbE(v)
    - \bbP(w)B^{4,\bal}_p(v)- R^{4,\bal}_m(w)\bbP(v)\,  dx ,
\end{aligned}
\end{equation}
respectively 
\begin{equation}
    \bJ^8(v,w) =  \int B^{4,\bal}_m(v) R^{4,\bal}_{p}(w) 
    - R^{4,\bal}_{m}(v) B^{4,\bal}_p(w)
    + B^{4,\bal}_m(w) R^{4,\bal}_{p}(v) 
    - R^{4,\bal}_{m}(w) B^{4,\bal}_p(v)\,  dx .
\end{equation}
Finally, we are also left with the double integral
\begin{equation}
\begin{aligned}
    \bK^{\geq 6} &= \iint_{x > y} \ms(v)(x) 
    (Q^{4,\bal}_p(|\partial|w|^2|^2)(y) + N^{\geq 4, \unbal}_p(w)(y) 
    + R^{\geq 6}_{p}(w)(y))\\
    &\qquad + \ps(w)(y) (Q^{4,\bal}_m(|\partial|v|^2|^2)(x) 
    + N^{\geq 4, \unbal}_m(v)(x) 
    + R^{\geq 6}_{m}(v)(x)) \, dx dy \\ 
    \quad &- \iint_{x > y}
    \ms(w)(y) (Q^{4,\bal}_p(|\partial|v|^2|^2)(x) + N^{\geq 4, \unbal}_p(v)(x) 
    +  R^{\geq 6}_{p}(v)(x))\\
    &\qquad + \ps(v)(x) (Q^{4,\bal}_m(|\partial|w|^2|^2)(y) 
    + N^{\geq 4, \unbal}_m(w)(y) 
    +  R^{\geq 6}_{m}(w)(y)) \, dx dy,
\end{aligned}
\end{equation}
whose leading part has order $6$ but also contains terms of orders $8$ and higher;  
but we will treat it all perturbatively later.

\bigskip

    It is instructive to consider the case where $C^{\bal} = 1$, 
    which corresponds to the cubic defocusing NLS. 
There we may take $B^{4,\bal}_m= 0$ and $B^{4,\bal}_p=0$.
Further, $R^{4,\bal}_{m} = 0$ but $R^{4,\bal}_p = 1$.
Thus in particular we get 
\[
\bJ^6(v,v) =2 \int |v|^6 \, dx.
\]
This is where the focusing/defocusing type of the equation comes in, as it determines the sign of $\bJ^6$ 
(relative to the sign of $\bJ^4$).

\subsubsection{Nonlinear interaction Morawetz: the localized diagonal case}

Here we use the frequency localized mass density-flux \eqref{eq:dens_flux_mj}
and the corresponding momentum density-flux \eqref{eq:dens_flux_pj}
in order to produce a localized interaction Morawetz estimate.

Correspondingly, we have the localized mass and momentum densities
\[
    \ms_j = \bbM_j(v,\bar v) + B^{4,\bal}_{m,j}(v,\bar v, v,\bar v),
\]
\[
    \ps_{j} = \bbP_{j}(v,\bar v) + B^{4,\bal}_{p,j}(v,\bar v, v,\bar v),
\]
which satisfy the conservation laws \eqref{eq:dens_flux_mj}
and \eqref{eq:dens_flux_pj} copied below:
\[
 \partial_t \ms_j(v) = \partial_x(\bbP_{j}(v)
    + R^{4,\bal}_{m,j}(v)) + Q^{4,\bal}_{m,j}(|\partial|v|^2|^2) 
    + R^{\geq 6}_{m,j}(v)
    + N^{\geq 4,\unbal}_{m,j}(v),
\]
and 
\[
 \partial_t \ps_j(v) = \partial_x(\bbE_{j}(v)
    + R^{4,\bal}_{p,j}(v)) + Q^{4,\bal}_{p,j}(|\partial|v|^2|^2)
    +R^{\geq 6}_{p,j}(v)+ N^{\geq 4, \unbal}_{p,j}(v).
\]

For these we define the interaction Morawetz functional 
\begin{equation}\label{Ia-sharp-def}
\bI_{j}(v,w) =   \iint_{x > y} \ms_j(v)(x) \ps_{j}(w) (y) -  
\ps_{j}(v)(x) \ms_{j}(w) (y) \, dx dy.
\end{equation}

The time derivative of $\bI_{j}$ is
\begin{equation}\label{interaction-xi}
    \frac{d}{dt} \bI_{j} =  \bJ^4_{j} + \bJ^6_{j} + \bJ^8_{j} + \bK^{\geq 6}_{j}. 
\end{equation}

Here the quartic contribution $\bJ^4_{j}$ is 
the same as in the linear case,
\[
\bJ^4_{j}(v,w) = \int \bbM_j(v) \bbE_{j}(w) 
+ \bbM_j(w) \bbE_{j}(v)
- 2 \bbP_{j}(v) \bbP_{j}(w)\, dx.
\]
The sixth order term $\bJ^6_{j}$ has the form
\begin{equation}\label{eq:J6_def}
\begin{aligned}
    \bJ^6_j(v,w) =  \int & \ \bbM_j(v) R^{4,\bal}_{p,j}(w)
    + B^{4,\bal}_{m,j}(v) \bbE_{j}(w)
    - \bbP_{j}(v)B^{4,\bal}_{p,j}(w)- R^{4,\bal}_{m,j}(v)\bbP_{j}(w) 
\\ & \ 
    + \bbM_j(w) R^{4,\bal}_{p,j}(v)+ B^{4,\bal}_{m,j}(w) \bbE_{j}(v)
    - \bbP_{j}(w)B^{4,\bal}_{p,j}(v)- R^{4,\bal}_{m,j}(w)\bbP_{j}(v)\,  dx .
\end{aligned}
\end{equation}
Next we have 
\begin{equation}\label{eq:J8_def}
\bJ^8_{j}(v,w) =   \int 
    B^{4,\bal}_{m,j}(v) R^{4,\bal}_{p,j}(w) 
    - R^{4,\bal}_{m,j}(v) B^{4,\bal}_{p,j}(w)
    + B^{4,\bal}_{m,j}(w) R^{4,\bal}_{p,j}(v) 
    - R^{4,\bal}_{m,j}(w) B^{4,\bal}_{p,j}(v)
 \, dx .
\end{equation}

Finally the remaining term $\bK^{\geq 6}_{j}$ has the form
\begin{equation}\label{eq:Kg6_def}
\begin{aligned}
    \bK^{\geq 6}_j &= \iint_{x > y} \ms_j(v)(x) 
    (Q^{4,\bal}_{p,j}(|\partial|w|^2|^2)(y) + 
    N^{\geq 4,\unbal}_{p,j}(w)(y) 
    +R^{\geq 6}_{p,j}(w)(y))\\
    &\qquad + \ps_j(w)(y) (Q^{4,\bal}_{m,j}(|\partial|v|^2|^2)(x) 
    + N^{\geq 4, \unbal}_{m,j}(v)(x) 
    + R^{\geq 6}_{m,j}(v)(x)) \, dx dy \\ 
    \quad &- \iint_{x > y}
    \ms_j(w)(y) (Q^{4,\bal}_{p,j}(|\partial|v|^2|^2)(x) 
    + N^{\geq 4,\unbal}_{p,j}(v)(x) 
    + R^{\geq 6}_{p,j}(v)(x))\\
    &\qquad + \ps_j(v)(x) (Q^{4,\bal}_{m,j}(|\partial|w|^2|^2)(y) 
    + N^{\geq 4,\unbal}_{m,j}(w)(y) 
    + R^{\geq 6}_{m,j}(w)(y)) \, dx dy,
\end{aligned}
\end{equation}

Importantly, here we compute the symbol of $\bJ^6_{j}$ on the diagonal
$\xi_1 = \xi_2=\xi_3=\xi_4=\xi_5=\xi_6:=\xi$. The next result, proved in \cite{IT-global}, will be essential later on in order to obtain bounds for the $L^6_{t,x}$ Strichartz norm.

\begin{lemma}
The diagonal trace of the symbol $j^6_{j}$
is 
\begin{equation}\label{good-J6}
j^6_{j}(\xi) = p_j^4 (\xi) c (\xi, \xi, \xi) .
\end{equation}
\end{lemma}


\subsubsection{Nonlinear interaction Morawetz: the transversal case}

Here we return to the setting of Section~\ref{s:AB}
where we have two frequency intervals corresponding to $\dyad^j$ and $\dyad^k$ 
where $k \gg j$.
Our interaction Morawetz functional is given by
\begin{equation}\label{interaction-bi}
 \bI_{jk} = \int_{x > y} \ms_j(v)(x) \ps_{k}(w)(y)  
    - \ps_{j}(v)(x) \ms_k(w)(y)\, dx dy. 
\end{equation}

Using again the frequency localized mass density-flux 
\eqref{eq:dens_flux_mj} and the corresponding momentum density-flux 
\eqref{eq:dens_flux_pj}
we produce a localized interaction Morawetz estimate,
\begin{equation}\label{interaction-xi-AB}
    \frac{d}{dt} \bI_{jk} =  \bJ^4_{jk} + \bJ^6_{jk} + \bJ^8_{jk} + \bK^{\geq 6}_{jk}.
\end{equation}

Here the quartic contribution $\bJ^4_{jk}$ is 
the same as in the linear case,
\[
\bJ^4_{jk} = \int \bbM_j(v)(x) \bbE_{k}(w)(x) + \bbM_k(w)(x) \bbE_{j}(v)(x)-  
2 \bbP_{j}(v)(x) \bbP_{k}(w)(x)\, dx,
\]
and captures the 
bilinear $L^2_{t,x}$ bound.

The sixth order term $\bJ^6_{jk}$ has the form
\begin{equation}\label{eq:J6jk_def}
\begin{aligned}
    \bJ^6_{jk} &= \int -( \bbP_{j}(v) B^{4,\bal}_{p,k}(w)
    + \bbP_{k}(w) R^{4,\bal}_{m,j}(v)) + ( 
    \bbM_j(v) R^{4,\bal}_{p,k}(w) +\bbE_{k}(w) B^{4,\bal}_{m,j}(v))\\  &
    \qquad -( \bbP_{k}(w) B^{4,\bal}_{p,j}(v)
    + \bbP_{j}(v) R^{4,\bal}_{m,k}(w))+ ( 
    \bbM_k(w) R^{4,\bal}_{p,j}(v) +\bbE_{j}(v) B^{4,\bal}_{m,k}(w))\, dx.
\end{aligned}
\end{equation}

Next we have 
\begin{equation}\label{eq:J8jk_def}
    \bJ^8_{jk} =   \int - R^{4,\bal}_{m,j} B^{4,\bal}_{p,k}
 +  
 B^{4,\bal}_{m,j} R^{4,\bal}_{p,k}  + \text{symmetric} \,      dx .
\end{equation}

Finally the remainder term $\bK^{\geq 6}_{jk}$ has the form
\begin{equation}\label{eq:Kg6jk_def}
\begin{aligned}
    \bK^{\geq 6}_{jk} &= \iint_{x > y} \ms_j(u)(x) 
    (Q^{4,\bal}_{p,k}(|\partial|w|^2|^2)(y) + 
    N^{\geq 4,\unbal}_{p,k}(w)(y) 
    + R^{\geq 6}_{p,k}(w)(y))\\
    &\qquad + \ps_k(w)(y) (Q^{4,\bal}_{m,j}(|\partial|v|^2|^2)(x) 
    + N^{\geq 4, \unbal}_{m,j}(v)(x) 
    + R^{\geq 6}_{m,j}(v)(x)) 
    - \text{symmetric} \, dx dy 
\end{aligned}
\end{equation}



\section{Frequency envelopes and the bootstrap argument}

\label{s:boot}

The goal of this section is to develop the structure we will need in 
the remainder of the paper to prove our main result, Theorem~\ref{t:main}.
Although our main goal is to establish a small global $L^\infty_t H^s_x$
bound for solutions with $L^2$ initial data which is small in $H^s$, 
along the way we will also 
establish bilinear $L^2_{t,x}$ and $L^6_{t,x}$ Strichartz bounds for the solutions.
These will both 
play an essential role in the proof of Theorem~\ref{t:main},
and will also establish the dispersive properties of our global solutions.

Since the proof of our estimates loops back in a complex manner, 
it is most convenient to establish the bilinear $L^2_{t,x}$ and the $L^6_{t,x}$ Strichartz 
bounds in the setting of a bootstrap argument, where we already assume that the 
desired bilinear and Strichartz estimates hold but with weaker constants. 

Note that so far, we have developed the tools to handle balanced interactions,
but not yet all unbalanced interactions, the analysis of which is in 
Section \ref{sec:unbalanced_corrections}. The motivation for introducing the 
bootstrap estimates first is that Section \ref{sec:unbalanced_corrections} relies 
heavily on the bilinear $L^2_{t,x}$ and Strichartz bounds in order to understand which interactions
are uncontrollable.

\bigskip

To start with, we assume that the initial data has small $H^s$ size,
\[
\| \du_0\|_{H^s} \lesssim \epsilon,
\]
but potentially large $L^2$ size,
\[
\| \du_0\|_{L^2} < \infty.
\]
We consider a frequency decomposition for the initial data on an
inhomogeneous dyadic spatial scale, 
\[
    \du_0 = \sum_{j \in \Z^{\geq 0}} \du_{0,j}.
\]
Then we place the initial data components under an admissible frequency envelope 
according to Definition~\ref{def:admissible_envelope} 
\[
    \|\du_{0,j}\|_{H^{s}} \leq \epsilon c_j, \qquad c \in \ell^2,
\]
where the envelope $\{c_j\}$ is normalized,
\begin{equation*}
\| c\|_{\ell^2} \approx 1
\end{equation*}
and also preserves the $L^2$ norm of $\du_0$,
\begin{equation*}
  \|\du_0\|_{L^2} \approx  \epsilon \| \dyad^{-sj}c_j\|_{\ell_j^2}.
\end{equation*}

Our goal will be to establish the following frequency envelope bounds for the solution:

\begin{theorem}\label{t:boot}
Fix $-\frac12 < s < 0$.
Let $u \in C([0,T];L^2)$ be a solution for the equation \eqref{nls}
with initial data $\du_0$  which has $H^s$ size at most $\epsilon$, where $\epsilon \ll 1$.
Let $\{\epsilon c_j\}$ be an admissible frequency envelope for the initial data 
in $H^s$, with $c_j$ normalized in $\ell^2$. Then the solution $u$ satisfies 
the following bounds, where all space-time norms are taken over
$[0,T]\times\R$ and the implicit constants are independent of $T$:
\begin{enumerate}[label=(\roman*)]
\item Uniform frequency envelope bound:
\begin{equation}\label{eq:uj_ee}
    \| P_ju \|_{L^\infty_t L^2_x} \lesssim \epsilon c_j\dyad^{-sj}
\end{equation}
\item Localized Strichartz bound:
\begin{equation}\label{eq:uj_se}
    \| P_j u \|_{L^6_{t,x}} \lesssim (\epsilon c_j)^{2/3}\dyad^{(1 - 4s)j/6} ,
\end{equation}
\item Localized Interaction Morawetz:
\begin{equation}\label{eq:uj_loc_bi}
\| \partial_x (P_j u P_j \bar u)  \|_{L^2_{t,x}} \lesssim  \epsilon^2 c_j^2 
    \dyad^{(1-4s)j/2} ,
\end{equation}
\item Transversal bilinear $L^2_{t,x}$ bound:
\begin{equation} \label{eq:uj_sep_bi}
    \|\partial( P_j u P_k \bar u(\cdot+x_0))  \|_{L^2_{t,x}} 
    \lesssim \epsilon^2 c_{j} c_{k} 
    (\dyad^{j/2} + \dyad^{k/2}) \dyad^{-sj}\dyad^{-sk},
\end{equation}
 for all $x_0 \in \R$.
\end{enumerate}
\end{theorem}
Here \eqref{eq:uj_loc_bi} can be seen as a particular case of 
\eqref{eq:uj_sep_bi} with $x_0 = 0$ and $j=k$; 
we stated it separately in order 
to ease comparison with earlier work on Interaction Morawetz estimates.

To prove this theorem, we make a bootstrap assumption where we assume the same bounds but with a worse constant $C$, as follows:

\begin{enumerate}[label=(\roman*)]
\item Uniform frequency envelope bound:
\begin{equation}\label{eq:uj_ee_boot}
\| P_ju \|_{L^\infty_t L^2_x} \lesssim C \epsilon c_j\dyad^{-sj} ,
\end{equation}
\item Localized Strichartz bound:
\begin{equation}\label{eq:uj_se_boot}
    \| P_j u \|_{L^6_{t,x}} \lesssim C (\epsilon c_j)^{2/3}\dyad ^{(1 - 4s)j/6} ,
\end{equation}
\item Localized Interaction Morawetz:
\begin{equation}\label{eq:uj_loc_bi_boot}
\| \partial_x (P_j u P_j \bar u)  \|_{L^2_{t,x}} \lesssim  C \epsilon^2 c_j^2 
    \dyad^{(1-4s)j/2} ,
\end{equation}
\item Transversal bilinear $L^2_{t,x}$ bound:
\begin{equation} \label{eq:uj_sep_bi_boot}
    \|\partial( P_j u P_k \bar u(\cdot+x_0))  \|_{L^2_{t,x}} 
    \lesssim C \epsilon^2 c_{j} c_{k} 
    (\dyad^{j/2} + \dyad^{k/2}) \dyad^{-sj}\dyad^{-sk},
\end{equation}
 for all $x_0 \in \R$.
\end{enumerate}

Then we seek to improve the constant in these bounds. 
The gain will come from the fact that the $C$'s will always come paired
with extra $\epsilon$'s. 

We also remark on the need to add translations to the bilinear 
$L^2_{t,x}$ estimates. This is because, unlike the linear bounds
\eqref{eq:uj_ee_boot} and \eqref{eq:uj_se_boot} 
which are inherently invariant with respect to translations, 
bilinear estimates are not invariant 
with respect to separate translations for the two factors.
Translation invariance is required to estimate multilinear symbols with 
integrable kernels, as in Lemma~\ref{lem:sep}. This will be the only
way in which we utilize the translation invariance.

Before we continue, we provide the continuity argument which shows that it 
suffices to prove Theorem~\ref{t:boot} under the bootstrap assumptions 
\eqref{eq:uj_ee_boot}-\eqref{eq:uj_sep_bi_boot}. 

For this, we denote by $T^*$ the maximal time for which the bounds 
\eqref{eq:uj_ee_boot}-\eqref{eq:uj_sep_bi_boot}
hold in $[0,T^*]$. By the local well-posedness
result, Theorem \ref{thm:lwp}, we have $T^* > 0$. Assume by contradiction that $T^*$
is finite.
Then the bootstrap version of the theorem implies that the bounds 
\eqref{eq:uj_ee}-\eqref{eq:uj_sep_bi}
hold in $[0,T^*]$. 

In particular, $u(T^*)$ will also be controlled 
by the same frequency envelope $c_j$ coming from the initial data, 
which we have chosen to preserve $L^2_x$ as well as $H^s_x$ so 
that we may apply Theorem \ref{thm:lwp}.
This implies in turn that 
the bounds \eqref{eq:uj_ee_boot}-\eqref{eq:uj_sep_bi_boot} hold in $[T^*,T^*+T]$
with $C \approx 1$. Adding this to the bounds \eqref{eq:uj_ee}-\eqref{eq:uj_sep_bi}
in $[0,T^*]$, using  the subadditivity of the $L^6_{t,x}$ and bilinear norms over time intervals, it follows that  \eqref{eq:uj_ee_boot}-\eqref{eq:uj_sep_bi_boot} 
hold in $[0,T^*+T]$, thereby contradicting the maximality of $T^*$.


\section{Unbalanced normal form corrections}
\label{sec:unbalanced_corrections}

In this section, we turn our attention to the unbalanced part of the nonlinearity 
in 
\[
    C = C^{\bal} + C^{\unbal}
\]
where at least one input frequency is separated from the output frequency. 
Note that throughout
this section, we will unsymmetrize multilinear forms so that their 
highest frequency input is the first input, the next highest is the second 
and so on.

Unlike the analysis in Section \ref{s:nonlin_den_flux} 
for the balanced part of the nonlinearity, we 
will not need to take advantage of symmetry to control 
the resonant part of $C^{\unbal}$ since we will have enough separation 
to utilize bilinear $L^2_{t,x}$ estimates. Thus, for simplicity of 
the following argument, we choose to interpret $C^{\unbal}$ as a 
source term at the level of the equation. 

In particular, we will define resonant and nonresonant interactions 
based on the size of the source term version of the symbol $\Delta^4\xi^2$ given by
\[
    \Delta^4\xi^2 = \xi_1^2 - \xi_2^2 + \xi_3^2 - (\xi_1 - \xi_2 + \xi_3)^2
    =2(\xi_1 - \xi_2)(\xi_2 - \xi_3)
\]
where $\xi_1 - \xi_2 + \xi_3$ is the output frequency of the trilinear form, here
$C^{\unbal}$. In the unbalanced regime, we can classify 
resonant and non-resonant interactions based on the relative dyadic 
sizes of the input and output frequencies. In particular, typical 
unbalanced resonant interactions will be of the form
\[
    \sum_{h \gg j} P_j(P_h u P_h \bar u P_j u)
\]
as $\Delta^4\xi^2$ may be close to zero in this region,
while typical nonresonant interactions will be of the form 
\[
    \sum_{h \gtrsim m \gg j} P_j(P_h u P_h \bar u P_m u)
\]
with $\Delta^4\xi^2 \sim \dyad^h \dyad^m$. Note that as $s$ approaches 
$-1/2$, we will lose control of high frequencies more rapidly than 
low frequencies as 
\[
    \|P_j u\|_{L^2_x} \sim \dyad^{-sj}\|P_j u\|_{H^s}.
\]
Thus we expect both resonant interactions and nonresonant interactions which 
involve frequencies lower than the output frequency 
to be perturbative. 

Based on this,
we further divide $C^{\unbal}$ into perturbative, resonant, and non-resonant parts:
\[
    C^{\unbal} = C^{\low} + C^{\res} + C^{\nonres}
\]
where we will pin down the specifics of each of these in the following 
subsections.

As we will see in more detail in Section \ref{ss:cubic_corrections}, 
for $s$ close to $-1/2$ we lose control of $C^{\nonres}$. This prompts 
us to look for a normal form change of variables of the form
\[
    v = u + B^{3}(u,\bar u, u)
\]
where we choose the symbol of $B^3$ to satisfy 
\[
    \Delta^4\xi^2 \cdot 
    b^3 = -c^{\nonres} 
\]
because then $v$ will satisfy an improved Schr\"odinger equation 
\begin{align*}
    (i\partial_t + \partial_x^2) v &= C(u, \bar u, u) - C^{\nonres}(u, \bar u, u)\\
    &+ B^3(C(u,\bar u, u), \bar u, u) - B^3(u, \ol{C(u,\bar u, u)}, u)
    + B^3(u, \bar u, C(u,\bar u, u)) 
\end{align*}
where $C^{\nonres}$ is canceled. Unfortunately, we will 
not be able to control the fifth order terms. In Section
\ref{ss:cubic_corrections} we take advantage of the fact that 
$B^3(C^{\nonres}(u,\bar u, u), \bar u, u)$ can be corrected 
up to lower order terms by $B^3(B^3(u,\bar u, u),\bar u, u)$. 
This will prompt us 
to instead consider \emph{implicit normal forms} 
of the form 
\begin{align*}
    v &= u + B^{3}(v,\bar v, v)\\
    &= u + B^3(u, \bar u, u) + B^3(B^3(u, \bar u, u), \bar u, u) + \cdots.
\end{align*}
This improves our situation but does not fully 
remove all unbounded quintic terms.
In Section \ref{ss:quintic_corrections} we add to our implicit normal 
form corrections to account for the quintic term
\[
    B^3(C^\bal(u,\bar u, u), \bar u, u) - B^3(u, \ol{C^\bal(u,\bar u, u)}, u).
\]
to get a normal form transformation of the form
\[
    v = u + B^{3}(v,\bar v, v) + B^5(v, \bar v, v, \bar v, v).
\]
And finally in Section \ref{ss:higher_corrections}, we will 
make corrections concerning $B^3(u, \bar u, C^{\res}(u, \bar u, u))$ as 
well as uncontrollable higher order errors arising from the implicit normal form. 
These 
corrections will generate more errors up to 
some large finite order $2N+1$, depending on $s$, leading 
overall to a normal form change of variables 
\[
    v = u + \sum_{n=1}^N B^{2n+1}(v,\ldots, v). 
\]

For this change of variables to be useful, we want two properties to hold.
First, we want $u$ and $v$ to be comparable in the frequency envelope 
controlled versions of the $L^2_x$, $L^6_{t,x}$ and bilinear 
$L^2_{t,x}$ norms, given enough smallness. Second we will want 
the resulting unbalanced source term in the equation for $v$ to be small 
in the frequency envelope controlled dual norm given smallness in $v$ 
again in frequency envelope controlled $L^2_x$, $L^6_{t,x}$ and bilinear 
$L^2_{t,x}$ norms. 

That both of these are possible is given in the following theorem.

\begin{theorem}\label{thm:unbal_correction}
    Let $u$ solve \eqref{nls} and assume the bootstrap assumptions 
    \eqref{eq:uj_ee_boot}-\eqref{eq:uj_sep_bi_boot} hold for $u$ 
    under an admissible envelope $c$. 

    Then there exists an integer $N$, $2n+1$ linear forms $B^{2n+1}$
    for each $1 \leq n \leq N$, and a unique solution $v$ to 
    \[
        v = u + \sum_{n=1}^N B^{2n+1}(v, \bar v, \ldots, v)
    \]
    such that $u$ and $v$ are comparable in the bootstrap norms:
    for each $T$ we have fixed time energy estimates 
    \[
        \|P_ju(T) - P_j v(T)\|_{L^2_x} \lesssim 
        C^3\epsilon^3\dyad^{-sj}c_j,
    \]
    spacetime Strichartz estimates
    \[
        \|P_ju - P_j v\|_{L^6_{t,x}} \lesssim C^3\epsilon^{8/3}c_j^{2/3}
        \dyad^{(1-4s)j/6},
    \]
    and the following mixed bilinear $L^2_{t,x}$ estimates for all 
    $j, k$ and offsets $x_0$
    \[
        \|\partial_x(P_j(u-v) \ol{P_k(u(\cdot + x_0)-v(\cdot + x_0))})\|_{L^2_{t,x}} 
        \lesssim C^5\epsilon^{6}c_jc_k(\dyad^{j/2} + \dyad^{k/2})\dyad^{-sj}
        \dyad^{-sk}
    \]
    and 
    \[
        \|\partial_x(P_j(u-v) \ol{P_k(u(\cdot + x_0))})\|_{L^2_{t,x}} 
        +\|\partial_x(P_j(u-v) \ol{P_k(v(\cdot + x_0))})\|_{L^2_{t,x}} 
        \lesssim C^3\epsilon^{4}c_jc_k(\dyad^{j/2} + \dyad^{k/2})\dyad^{-sj}
        \dyad^{-sk};
    \]
    and $v$ solves a Schr\"odinger equation 
    \begin{equation}\label{eq:new_nls_simp}
        (i\partial_t + \partial_x^2) v
        = C^\bal(v,\bar v, v) + N^{\unbal}(v)
    \end{equation}
    for some translation invariant nonlinear form $N^{\unbal}$
    such that 
    \begin{equation}\label{eq:source_term}
\sup_{x_0\in \mathbb R}\ \big\| P_j N^{\unbal}(v)\ P_k \ol v(\cdot+x_0)\big\|_{L^1_{t,x}}
\lesssim C^2\epsilon^4\, c_j^2\, \dyad^{-2sj}, \qquad j \cong k,
\end{equation}
where $c$ is the frequency envelope in the bootstrap bounds.

\end{theorem}
\begin{remark}
    The four estimates involving $u-v$ show that if we can close the 
    bootstrap estimates for $v$ (or $u$) without the constant $C$, 
    then as long as $\epsilon$ is small enough, we can close the 
    bootstrap estimates for $u$ (or $v$) also without the constant $C$. This 
    is possible because each difference estimate has more powers of 
    $\epsilon$ than the corresponding bootstrap estimate.

    In addition, with the exception of 
    one special term which we will call $B^{3,\hihi}$ in Section 
    \ref{ss:cubic_corrections}, 
    we will be able to construct each $B^{2n+1}$ so 
    that it depends on $v$ through higher frequencies than the output 
    frequencies. This reflects the fact that the contribution of 
    high frequencies is the most problematic as $s \to -1/2$. But even 
    in $B^{3,\hihi}$, two of the three input frequencies will be much 
    higher than the output frequency.

    Finally, in the source term estimate \eqref{eq:source_term}, 
    note that while we include a translation $x_0$ so 
    that we may apply this estimate after applying Lemma \ref{lem:sep}, 
    our proofs going forward 
    will be independent of $x_0$, so we omit it from our arguments going forward
    for simplicity.
\end{remark}

In the following sections, the estimate \eqref{eq:source_term} will 
be of special importance. We will prove it by splitting $N^{\unbal}$ into 
a sum of multilinear forms
\[
    N^{\unbal}(v) = \sum_{n=1}^\infty F^{2n+1}(v)
\]
where we would like to prove \eqref{eq:source_term} for each $F$ with enough 
powers of $\epsilon$ to converge geometrically in the sum. 
It will often be more convenient to prove the following stronger 
version of \eqref{eq:source_term} for a multilinear form:

\begin{definition}\label{def:good_L6_terms}
Fix a constant $\tilde C$ and let $v$ satisfy the bootstrap estimates 
    \eqref{eq:uj_ee_boot}-\eqref{eq:uj_sep_bi_boot}.
For $n \geq 2$, we will call a $2n+1$-linear form $G$ \emph{good} 
    if we have the estimate
    \[
        \|P_jG\|_{L^1_{t,x}} \leq (\tilde C)^{2n+1} \epsilon^{2n-1} 
        \dyad^{(-1/2 -s)j}c_j
    \]
\end{definition}
\begin{remark}
The constant $\tilde C$ will be some large constant depending on $s$ 
and the regularity of the symbol $c$. The reason we insist on 
this constant is because, in the following sections, we will 
sum over $n$ and will need the sum to be at worst geometric.

We also remark that cubic terms can never be good under this definition since 
    the bootstrap estimates \eqref{eq:uj_ee_boot}-\eqref{eq:uj_sep_bi_boot} do 
    not provide a way to estimate $\|v^3\|_{L^1_{t,x}}$. However, quintic terms 
    may, in principle, 
    be estimated by 
    \[
        \|v^5\|_{L^1_{t,x}} \lesssim \|v^2\|_{L^2_{t,x}}\|v\|_{L^6_{t,x}}^3
    \]
    or 
    \[
        \|v^5\|_{L^1_{t,x}} \lesssim \|v^2\|^2_{L^2_{t,x}}\|v\|_{L^\infty_{t,x}}
    \]
    and septic and higher order terms, in addition to these estimates, 
    also have access to 
    \[
        \|v^7\|_{L^1_{t,x}} \lesssim \|v\|^6_{L^6_{t,x}}\|v\|_{L^\infty_{t,x}}.
    \]
    Note that the last of these estimates gives the power of $\epsilon$ in the 
    above definition.
\end{remark}

It is clear that a sum of good terms will 
satisfy \eqref{eq:source_term} from Bernstein's inequality and 
the bootstrap estimate \eqref{eq:uj_ee_boot}. 

This stronger estimate will be helpful as it is respected by $B$ terms 
under the following mild condition.
\begin{proposition}\label{prop:b_preserves_good}
Suppose $G$ is a good term, $v$ satisfies the bootstrap estimates of Theorem
    \ref{thm:unbal_correction}, and $B$ is an order $2N+1$ 
    term satisfying the $L^1$ based fixed time estimate 
\begin{equation}\label{eq:bn_pointwise_L1}
    \begin{aligned}
        &\|P_jB(f, \bar v, \ldots, v)\|_{L^1_x} + \cdots + 
        \|P_jB(v, \bar v, \ldots, \bar v, f)\|_{L^1_x}\\
        &\qquad\leq (\tilde C \epsilon)^{2N} \left(
        \sum_{h \geq j} \|P_h f\|_{L^1_x} + 
        \dyad^{(-1/2-s)j}c_j \sum_{\ell \leq j} \|P_\ell f\|_{L^1_x} \right),
    \end{aligned}
\end{equation}
then
\[ 
    B(G,\bar v, \ldots, v) + \cdots + B(v, \ldots,\bar v,  G)
\]
has the good estimate from Definition \ref{def:good_L6_terms} with
constants depending on $N$.
\end{proposition}
\begin{proof}
    The proof follows immediately from the fixed time estimate 
    \eqref{eq:bn_pointwise_L1}, Definition \ref{def:good_L6_terms},
    and the slowly varying condition in Definition~\ref{def:admissible_envelope}
    as long as $-1/2 < s$. For the high frequency 
    terms 
    \[
        \|P_h f\|_{L^1_x} \leq 
        \tilde C^{2n+1}\epsilon^{2n-1}
        \dyad^{(-1/2 - s)h} c_h
    \]
    which sums geometrically to the desired bound if $s > -1/2$.
For the low frequency 
    terms we do not need the frequency envelope since it is provided in 
    the estimate. We get 
    \[
        \|P_\ell f\|_{L^1_x} \leq 
        \tilde C^{2n+1}\epsilon^{2n-1}
        \dyad^{(-1/2 - s)\ell} 
    \]
    which sums geometrically to $1$ since $s > -1/2$ which suffices.

\end{proof}

Before we move on with our construction of the relevant corrections, we give 
sufficient conditions on the $B^{2n+1}$ to show existence, uniqueness, and 
comparability of the functions $u$ and $v$ above. 

\begin{definition}\label{def:good_corrections}
    A $2n+1$ linear form $B$ will be a good correction if it satisfies the 
    following estimates:

  (i)   For each $w$ which satisfies the bootstrap estimates 
    \eqref{eq:uj_ee_boot}-\eqref{eq:uj_sep_bi_boot} and appropriate $f$, 
    we have the estimate \eqref{eq:bn_pointwise_L1}

\begin{equation*}
    \begin{aligned}
        &\|P_jB(f, \bar w, \ldots, w)\|_{L^1_x} + \cdots + 
        \|P_jB(w, \bar w, \ldots, \bar w, f)\|_{L^1_x}\\
        &\qquad\leq (\tilde C \epsilon)^{2n}\left(
        \sum_{h \geq j} \|P_h f\|_{L^1_x} + 
        \dyad^{(-1/2-s)j}c_j \sum_{\ell \leq j} \|P_\ell f\|_{L^1_x} \right),
    \end{aligned}
\end{equation*}
as well as $L^2$ and $L^6$ based versions, with $p=2$ and $p=6$ below
\begin{equation}\label{eq:bn_pointwise}
    \begin{aligned}
        &\|P_jB^{2n+1}(f, \bar w, \ldots, w)\|_{L^p_x} + \cdots + 
    \|P_jB^{2n+1}(w, \bar w, \ldots, \bar w, f)\|_{L^p_x}\\
        &\qquad\lesssim (C\epsilon)^{2n} \left(
    \dyad^{j/2}\sum_{h \geq j} \dyad^{-h/2}\|P_h f\|_{L^p_x}
        + c_j\sum_{\ell \leq j} \|P_\ell f\|_{L^p_x}\right),
    \end{aligned}
\end{equation}
and a bilinear estimate
\begin{equation}\label{eq:bn_bilinear_1}
        \|P_jB^{2n+1}(w, \ldots, w)\|_{L^2_{t,x}} 
        \lesssim (C\epsilon)^{2n+1} c_j \dyad^{(-1 - s)j} 
    \end{equation}

(ii) Further, we have the following mixed bilinear estimates whenever 
    $w_1$ and $w_2$ satisfy the bootstrap estimates \eqref{eq:uj_ee_boot} 
    -\eqref{eq:uj_sep_bi_boot} and in addition mutually satisfy the bilinear 
    estimate
    \[
        \|\partial_x(P_jw_1 \ol{P_kw_2(\cdot + x_0)})\|_{L^2_{t,x}}
        \lesssim C\epsilon^2 (\dyad^{j/2} + \dyad^{k/2})\dyad^{-sj}\dyad^{-sk}
        c_j c_k
    \]
    then we have the mixed bilinear estimates for $B$:
    \begin{equation}\label{eq:bn_bilinear_2}
        \| \partial_x(P_jB^{2n+1}(w_1, \ldots, w_1) P_k \bar w_2(\cdot + x_0))\|_{L^2_{t,x}} 
        \lesssim (C\epsilon)^{2n+2} (\dyad^{j/2} + \dyad^{k/2})
        \dyad^{-sj}\dyad^{-sk}c_jc_k 
    \end{equation}
\end{definition}
Now, we are in a position to prove the existence of a 
comparable $v$ from $u$ satisfying the bootstrap assumptions if in addition 
we assume Definition \ref{def:good_corrections} for our $B$ terms. 
\begin{proposition}\label{prop:contraction}
    Suppose for each $1 \leq n \leq N$, $B^{2n+1}$ is a good correction in the sense of Definition 
    \ref{def:good_corrections}. 
Then for every  $u \in C^0_t L^2_x$ satisfying
    the bootstrap estimates \eqref{eq:uj_ee_boot}-\eqref{eq:uj_sep_bi_boot} we have a unique solution $v$ to 
    \[
        v = u + \sum_{n=1}^N B^{2n+1}(v, \ldots, v)
    \]
    which satisfies the bootstrap estimates \eqref{eq:uj_ee_boot}-\eqref{eq:uj_sep_bi_boot}
    and is comparable\footnote{Using also the $C$ constant.} to $u$ according to Theorem \ref{thm:unbal_correction}.
\end{proposition}
\begin{proof}
    The estimates for the linearization of each $B^{2n+1}$ near zero 
    \eqref{eq:bn_pointwise}, in
    addition to the slowly varying condition of the $c_j$ and the fact that 
    $s > -1/2$, show 
    that in a neighborhood of zero including $u$, the nonlinear operator
    \[
        F(v) = v - \sum_{n=1}^N B^{2n+1}(v, \ldots, v)
    \]
    has a bounded, and uniformly invertible derivative in frequency 
    envelope controlled $C^0_t L^2_x \cap L^6_{t,x}$
    provided $\epsilon$ is small enough. In 
    particular, the inverse function theorem applied in frequency envelope weighted norms implies the existence and uniqueness 
    of a $v$ which solves $u = F(v)$ which also satisfies estimates 
    comparable to the bootstrap estimates \eqref{eq:uj_ee_boot} 
    and \eqref{eq:uj_se_boot}. The difference estimates follow now 
    from the linearized estimates \eqref{eq:bn_pointwise} and the bootstrap 
    estimates \eqref{eq:uj_ee_boot} and \eqref{eq:uj_se_boot} on $v$. 
    Note that since the $L^2_x$ contraction holds pointwise 
    and $u$ is continuous in $L^2$, it follows that $v$ is continuous in $L^2_x$
    as well.

    All that remains is to use the estimates \eqref{eq:bn_bilinear_1} 
    and \eqref{eq:bn_bilinear_2} to show that 
    the bilinear bootstrap estimates \eqref{eq:uj_loc_bi_boot} and 
    \eqref{eq:uj_sep_bi_boot} hold for $v$ and that we have 
    the two bilinear difference estimates in Theorem \ref{thm:unbal_correction}. 

    To show $v$ satisfies the bilinear bootstrap estimates, 
    we follow an approximating sequence for $v$. We set 
    \[
        v_m = u + \sum_{n}B^{2n+1}(v_{m-1}, \ldots, v_{m-1})
    \]
    and $v_0 = u$. Now since $v_0$ trivially mutually satisfies the bilinear
    estimate with $u$, we can use the estimates \eqref{eq:bn_bilinear_1} 
    and \eqref{eq:bn_bilinear_2} to show that $v_1$ 
    is controlled by $u$ and that $v_1$ mutually satisfies the bilinear estimate 
    with $u$ by writing 
    \begin{align*}
        P_j v_1 \ol{P_k v_1(\cdot + x_0)} 
        &= 
        P_j (u + \sum_n B^{2n+1}(v_0)) 
        \ol{P_k (u + \sum_n B^{2n+1}(v_0))(\cdot + x_0)}\\
        &=P_j u \ol{P_k u(\cdot + x_0)}
        + \sum_n P_j B^{2n+1}(v_0) \ol{P_k u(\cdot + x_0)}\\
        &\qquad + \sum_n P_j u \ol{P_k B^{2n+1}(v_0)(\cdot + x_0)} 
        + \sum_{n,m} P_j B^{2n+1}(v_0) \ol{P_k B^{2m+1}(v_0)(\cdot + x_0)}.
    \end{align*}
    Note that in the first term above we use the corresponding bootstrap estimates 
    \eqref{eq:uj_loc_bi_boot} and \eqref{eq:uj_sep_bi_boot} for $u$; for the 
    second and third terms we use \eqref{eq:bn_bilinear_2}; and for the 
    fourth term we use \eqref{eq:bn_bilinear_1} on the higher frequency of 
    $j$ and $k$, and use the $L^2_x$ version of \eqref{eq:bn_pointwise} 
    as well as Bernstein's inequality on the lower frequency.

    By induction, we then have that each $v_m$ is controlled 
    in the frequency envelope controlled bilinear norm by $u$, and by taking 
    limits, we have the bilinear bootstrap estimates \eqref{eq:uj_loc_bi_boot}
    and \eqref{eq:uj_sep_bi_boot} for $v$.

    From here, we use the estimates \eqref{eq:bn_bilinear_1} and 
    \eqref{eq:bn_bilinear_2} on the difference. 
    In particular we can show the estimates in Theorem 
    \ref{thm:unbal_correction}, by replacing 
    \[
        u-v = \sum_n B^{2n+1}(v).
    \]
\end{proof}

After this proposition, the only remaining task in the 
proof of Theorem~\ref{thm:unbal_correction} 
will be to construct the good correction terms 
$B^{2n+1}$ in such a way that $v$ satisfies \eqref{eq:new_nls_simp} 
with a source term $N^{\unbal}$ which satisfies \eqref{eq:source_term}. This is 
the objective of the following three subsections.

\subsection{Cubic Corrections}
\label{ss:cubic_corrections}

In this section, we focus on correcting 
\[
C^{\nonres}(u,\bar u, u).
\]
Before we begin, let us pin down the precise definition of this term, which 
is meant to capture nonresonant interactions which are not perturbative. Recall
that a set of frequencies is nonresonant when $\Delta^4\xi^2$ is bounded away from 
zero
and nonperturbative when at least one input frequency is very large 
compared to the output frequency. In all cases where at least one frequency 
is much higher than the output frequency, we must have that two inputs 
have frequencies which are comparable to the highest frequency. There
are two situations where this happens which we would like to distinguish. 

First
when the first and third inputs are not comparable, by unsymmetrizing we may 
assume that the first is larger, then we must have $\xi_1 \sim \xi_2 \sim \dyad^h$
and $\xi_3 \sim \dyad^m$ where $m \ll h$. If the output frequency is
$\xi_1 - \xi_2 + \xi_3 \sim \dyad^j$
with $j \ll m$ then we have that 
$|\xi_1 - \xi_2| \sim \dyad^m$ and we calculate 
\[
    \Delta^4\xi^2 = 2(\xi_1 - \xi_2)(\xi_2 - \xi_3)  
\]
This is unambiguously large if $m \gg j$ with 
$\Delta^4\xi^2 \sim \dyad^h\dyad^m$.
Since the size of $\Delta^4\xi^2$ depends on both frequencies $h$ and $m$ 
we will denote this situation $C^{\nonres, \himed}$.

In the other situation, the first and third frequencies are comparable and so 
we have $\xi_1 \sim \xi_3 \sim \dyad^h$ and $\xi_2 \sim \dyad^m$ with 
$h \gtrsim m$. Here we instead have $\xi_1 - \xi_2 \sim \xi_3 \sim \dyad^h$ 
and similarly $\xi_2 - \xi_3 \sim \xi_1 \sim \dyad^h$ so we calculate 
\[
    \Delta^4\xi^2 = 2(\xi_1 - \xi_2)(\xi_2 - \xi_3)
    \sim \dyad^{2h}
\]
In parallel to $C^{\nonres, \himed}$ we will denote this term 
$C^{\nonres, \hihi}$.

The following definition makes this splitting precise:
\begin{definition}\label{def:c_splitting}
    We first split into cases based on the relationship between the first and 
    third inputs to $c$.
    We define 
    \[
        c^{\himed}(\xi_1,\xi_2, \xi_3) =
        \sum_{h \gg m}
        p_{h}(\xi_1)p_{m}(\xi_3) c(\xi_1,\xi_2,\xi_3)
    \]
    and 
    \[
        c^{\hihi}(\xi_1,\xi_2, \xi_3) =
        \sum_{j_1 \sim j_3}
        p_{j_1}(\xi_1)p_{j_3}(\xi_3) c(\xi_1,\xi_2,\xi_3).
    \]
    From $c^{\himed}$ we further split based on the relationship of 
    $h$ and $m$ to the output frequency $j$ corresponding to 
    $\xi_1 - \xi_2 + \xi_3$. In these terms we will make very symmetric 
    definitions to ease the analysis going forward and as such 
    we replace $\xi_3$ by $\xi_2-\xi_1$ at the cost of a perturbative error.
    An interaction will be \emph{nonresonant} when $m \gg j$
    \[
        c^{\nonres, \himed}(\xi_1,\xi_2,\xi_3) 
        = \sum_{h \gg m \gg j} p_j(\xi_1- \xi_2 + \xi_3) 
        p_m(\xi_1 - \xi_2)
        p_h(\frac12 (\xi_1 + \xi_2)) c^{\himed}(\xi_1, \xi_2, \xi_2 - \xi_1);
    \]
    will be \emph{near resonant} when $h \gg j \gtrsim m$
    \[
        c^{\res, \himed}(\xi_1,\xi_2,\xi_3) 
        = \sum_{h \gg j \gtrsim m} p_j(\xi_1- \xi_2 + \xi_3) 
        p_m(\xi_1 - \xi_2)
        p_h(\frac12 (\xi_1 + \xi_2)) c^{\himed}(\xi_1, \xi_2, \xi_2 - \xi_1);
    \]
    and \emph{perturbative} otherwise
    \begin{align*}
        c^{\low, \himed}&(\xi_1,\xi_2,\xi_3) 
        = \sum_{j \sim m \gtrsim h} p_j(\xi_1- \xi_2 + \xi_3) 
        p_m(\xi_1 - \xi_2)
        p_h(\frac12 (\xi_1 + \xi_2))c^{\himed}(\xi_1, \xi_2, \xi_3)\\
        &+ \sum_{j \gtrsim h \gg m} p_j(\xi_1- \xi_2 + \xi_3) 
        p_m(\xi_1 - \xi_2)
        p_h(\frac12 (\xi_1 + \xi_2))c^{\himed}(\xi_1, \xi_2, \xi_3)\\
        &+ 
\sum_{h \gg j, h \gg m } p_j(\xi_1- \xi_2 + \xi_3) 
        p_m(\xi_1 - \xi_2)
        p_h(\frac12 (\xi_1 + \xi_2)) 
        \left(c^{\himed}(\xi_1, \xi_2, \xi_3) - c^{\himed}(\xi_1, \xi_2,\xi_2 - \xi_1)\right)
    \end{align*}
    From $c^{\hihi}$ we compare the frequencies of $j_1, j_2$ and the output
    frequency.
    An interaction where all frequencies are comparable we will call 
    \emph{balanced}
    \[
        c^{\bal}(\xi_1,\xi_2,\xi_3) = 
        \sum_{j}\sum_{j_1 \sim j, j_2 \sim j, j_3 \sim j} 
        p_j(\xi_1-\xi_2 + \xi_3)p_{j_1}(\xi_1)p_{j_2}(\xi_2)p_{j_3}(\xi_3)
        c^{\hihi}(\xi_1,\xi_2,\xi_3);
    \]
    where the first and third frequencies are much greater than the output
    we will call \emph{nonresonant}
    \[
        c^{\nonres,\hihi}(\xi_1,\xi_2,\xi_3) = 
        \sum_{j}\sum_{j_3 \sim j_1 \gg j}
        p_j(\xi_1-\xi_2 + \xi_3)p_{j_1}(\xi_1)p_{j_2}(\xi_2)p_{j_3}(\xi_3)
        c^{\hihi}(\xi_1,\xi_2,\xi_3)
    \]
    and the remainder will be \emph{perturbative}
    \begin{align*}
        c^{\low,\hihi}(\xi_1,\xi_2,\xi_3)
        &=\sum_{j}\sum_{\substack{j_1 \lesssim j, j_2 \lesssim j, j_3 \lesssim j\\ 
        \min(j_1,j_2,j_3) \ll j}} 
        p_j(\xi_1-\xi_2 + \xi_3)p_{j_1}(\xi_1)p_{j_2}(\xi_2)p_{j_3}(\xi_3)
        c^{\hihi}(\xi_1,\xi_2,\xi_3)
    \end{align*}
    We define the combined forms
    \[
        c^{\nonres} = c^{\nonres,\himed} + c^{\nonres,\hihi},
    \]
    \[
        c^{\low} = c^{\low,\himed} + c^{\low,\hihi}
    \]
    and in parallel
    \[
        c^{\res} = c^{\res,\himed}.
    \]
    We also define, when appropriate, a name for a specific term in 
    the above sums, for instance 
    \[
        c^{\nonres,\himed}_{h,m,j}(\xi_1,\xi_2,\xi_3) = 
        p_j(\xi_1-\xi_2 + \xi_3)p_h(\frac12(\xi_1+\xi_2))p_m(\xi_1 - \xi_2) 
        c^{\himed}(\xi_1,\xi_2,\xi_2-\xi_1)
    \]
    so that 
    \[
        c^{\nonres,\himed} = \sum_{h \gg m \gg j} c^{\nonres,\himed}_{h,m,j}.
    \]
Finally, it is clear that after symmetrizing, these symbols 
sum to $c$:
\begin{align*}
    c(\xi_1, \xi_2, \xi_3) &= c^{\bal}(\xi_1, \xi_2, \xi_3)
    + \frac12 [c^{\res}(\xi_1, \xi_2, \xi_3) + c^{\res}(\xi_3, \xi_2, \xi_1)]\\
    &+ \frac12 [c^{\nonres} (\xi_1, \xi_2, \xi_3) + 
    c^{\nonres}(\xi_3, \xi_2, \xi_1)]
    + \frac12 [c^{\low}(\xi_1, \xi_2, \xi_3) + c^{\low}(\xi_3, \xi_2, \xi_1)]
\end{align*}

\end{definition}
\begin{remark}
    Note that these symbols are smooth on the dyadic scale and 
    so we can apply Lemma \ref{lem:sep} to these on each term in the above sums.

    In this section we will not take advantage of the replacement of 
    $\xi_3$ by $\xi_2 - \xi_1$ in $c^{\nonres,\himed}$ and $c^{\res, \himed}$, 
    but it will be necessary in Section \ref{ss:higher_corrections}
    and we will remark on it in more depth in Definition
    \ref{def:c2_h}. For 
    now we note that by the enhanced smoothness condition (H1s), we have 
    that the error term in $c^{\low,\himed}$ is separable with size 
    $\dyad^{-h/2}\dyad^{-m/2}\dyad^{j}$. This will be used 
    in 
    Proposition \ref{prop:b3_good_bad_splitting} to show that this term is 
    indeed perturbative. This is also the only place where we use 
    the enhanced smoothness condition (H1s).

    We also note that $c^{\low}$ captures all interactions 
    where all frequencies are comparable or less to the output frequency, 
    but at least one is much less. Of course, at least one of the inputs
    has to be at least comparable to the output. 
\end{remark}

Before we move on,
we discuss why the bootstrap estimates \eqref{eq:uj_ee_boot}-\eqref{eq:uj_sep_bi_boot} are 
insufficient to prove the source term estimate \eqref{eq:source_term} for 
$C^{\nonres}$. 

After applying Lemma \ref{lem:sep}, $C^{\nonres, \himed}$ is given by
\[
    \sum_{h \gg m \gg j} P_j(P_h u P_h \bar u P_m u).
\]
Then we can only use two bilinear $L^2_{t,x}$ estimates \eqref{eq:uj_sep_bi_boot} 
to control 
the left hand side of \eqref{eq:source_term}. Thus, the best estimate we can hope
for is 
\begin{align*}
    \|P_j C^{\nonres,\himed}(u,\bar u, u) P_j \bar u\|_{L^1_{t,x}}
    &\lesssim \sum_{h \gg m \gg j} C^2\epsilon^4
    \dyad^{(-1-2s)h}\dyad^{-sm}\dyad^{-sj}c_h^2 c_m c_j
\end{align*}
Note this sum diverges if $s < -1/3$. Heuristically, at lower regularity 
we allow much more mass at high frequencies, and we see here that 
the bilinear estimate is only good enough to control these high frequency 
interactions when $s$ is large enough. However, in situations where $m \lesssim j$,
which are captured by $C^{\res}$, the sum will converge for $s > -1/2$ which 
is sharp by scaling.

Thus, we are motivated to take advantage of the fact that these bad terms are 
nonresonant by introducing a normal form change of variables. That is, 
since $\Delta^4\xi^2$ is large on the support of $c^{\nonres}$, we can set
$b^3$ according to the division relation
\[
    \Delta^4\xi^2 b^3(\xi_1, \xi_2, \xi_3) = - c^{\nonres}(\xi_1,\xi_2,\xi_3).
\]
Based on this relation, we can calculate that 
\begin{align}\label{eq:cubic_correction}
    (i\partial_t + \partial_x^2) &B^3(u,\bar u, u) =
    - C^{\nonres}(u,\bar u, u)\\ 
\nonumber&+ B^3((i\partial_t + \partial_x^2)u,\bar u, u)
    - B^3(u,\ol{(i\partial_t + \partial_x^2)u}, u) 
+ B^3(u,\bar u, (i\partial_t + \partial_x^2)u).
\end{align}
In particular, by setting 
\[
    v = u + B^3(u,\bar u, u),
\]
we see that $C^{\nonres}$ will cancel in the Schr\"odinger equation for $v$. 

The following definition pins down the symbol of $b^3$ to respect the 
splitting we have chosen in Definition \ref{def:c_splitting}.

\begin{definition}\label{def:b3_nonres}
For $h \gg m \gg j$ we put 
\[
    b^{3,\himed}_{h,m,j}(\xi_1,\xi_2,\xi_3) = 
    -\frac{c^{\nonres,\himed}_{h,m,j} (\xi_1,\xi_2,\xi_3)}
{\xi_1^2 - \xi_2^2 + \xi_3^2 - (\xi_1 - \xi_2 + \xi_3)^2}
\]
and for $h \gg j$ and $m \lesssim h$ free we put
\[
    b^{3,\hihi}_{h,m,j}(\xi_1,\xi_2,\xi_3) = 
    -\frac{c^{\nonres,\hihi}_{h,m,j} (\xi_1,\xi_2,\xi_3)}
{\xi_1^2 - \xi_2^2 + \xi_3^2 - (\xi_1 - \xi_2 + \xi_3)^2}.
\]
Or expanding the definitions of the $c$ terms in \ref{def:c_splitting} we have 
with the same restrictions on $h$, $m$, and $j$
\[ 
    b^{3,\himed}_{h, m, j}(\xi_1,\xi_2,\xi_3) = 
    p_j(\xi_1-\xi_2+\xi_3) p_h(\frac{1}{2}\left( \xi_1 + \xi_2) \right)
    p_m(\xi_1-\xi_2)\frac{c^{\himed}(\xi_1,\xi_2,\xi_2-\xi_1)}
    {\xi_1^2 - \xi_2^2 + \xi_3^2 - (\xi_1 - \xi_2 + \xi_3)^2}
\]
and
\[ 
    b^{3,\hihi}_{h, m, j}(\xi_1,\xi_2,\xi_3) = 
    \sum_{j_3 \sim h} p_j(\xi_1 - \xi_2 + \xi_3)p_{h}(\xi_1)
    p_{m}(\xi_2)p_{j_3}(\xi_3)\frac{c^{\nonres, \hihi}(\xi_1,\xi_2,\xi_3)}
{\xi_1^2 - \xi_2^2 + \xi_3^2 - (\xi_1 - \xi_2 + \xi_3)^2}.
\]
In parallel to the definitions for $c$
we will further write 
\[
    b^{3,\himed} = \sum_{h\gg m \gg j} b^{3,\himed}_{h,m,j},
\]
\[
    b^{3,\hihi} = \sum_{h\gg j} \sum_{m\lesssim h} b^{3,\hihi}_{h,m,j},
\]
and 
    \[b^3 = b^{3,\himed} + b^{3,\hihi}\]
\end{definition}

In agreement with the discussion above, we may calculate the 
size of the symbols  $b^{3,\himed}_{h,m,j}$ 
and $b^{3,\hihi}_{h,m,j}$ to 
see the gain of the division by $\Delta^4\xi^2$.

\begin{proposition}\label{prop:b3_kernel_size}
    We have the following kernel estimates for the  
    terms in Definition~\ref{def:b3_nonres}
    
\[
    \|\check b^{3,\himed}_{h,m,j}\|_{L^1} \lesssim \dyad^{-h}\dyad^{-m}
\]
and
\[
    \|\check b^{3,\hihi}_{h,m,j}\|_{L^1} \lesssim \dyad^{-2h}.
\]
\end{proposition}
\begin{remark}
Note that this makes $b^{3,\hihi}$ more 
favorable than $b^{3,\himed}$ since the symbol size is one over the highest 
frequency squared, compared to one over high times one over medium.
\end{remark}
\begin{proof}
    The two estimates will follow from Lemma \ref{lem:sep} if 
    we can split each symbol into parts each of which we can show is 
    smooth on some scales after a linear change of coordinates. We may 
    freely introduce larger smooth cutoffs, for instance to the 
    dyadic regions of $\xi_1, \xi_2$ and $\xi_3$. This makes it clear, by 
    the regularity condition (H1s), that $c(\xi_1,\xi_2, \xi_3)$ is separable. 
    In addition the projections $p_j(\xi_1)$ will be separable. Thus all that 
    remains is the denominators of the two symbols.

    Starting with $b^{3,\himed}_{h,m,j}$, the remaining part 
    after factoring out $c^{\himed}$ is 
    \[
        \frac{p_j(\xi_1-\xi_2+\xi_3)p_m(\xi_1-\xi_2)p_h(\frac{1}{2}\left(\xi_1+\xi_2)\right)}
        {\xi_1^2 - \xi_2^2 + \xi_3^2 - (\xi_1 - \xi_2 + \xi_3)^2}.
    \]
    We rewrite the denominator as
    \[
    \frac{1}{\xi_1^2 - \xi_2^2 + \xi_3^2 - (\xi_1 - \xi_2 + \xi_3)^2} =\frac{1}{(\xi_1 - \xi_2)(2\xi_2 - 2\xi_3)}
    \]
    Now it is clear that 
    \[ 
        \frac{p_m(\xi_1-\xi_2)}{\xi_1 - \xi_2}
    \]
    is separable of size $\dyad^{-m}$ after a change of coordinates, since 
    it is smooth on scale $\dyad^m$. 

    For the remainder we write 
    \[
        2\xi_2 - 2\xi_3 = (\xi_1 + \xi_2) + (\xi_1 - \xi_2) - 2(\xi_1 - \xi_2 + \xi_3)
    \]
    If we change coordinates to $\alpha = \xi_1 + \xi_2$, 
    $\delta = \xi_1 - \xi_2$ and $\xi_j = \xi_1 - \xi_2 + \xi_3$,
    we get
    \[
        \frac{p_j(\xi_j)p_m(\delta)p_h(\alpha)}
        {\alpha + \delta -2 \xi_j}.
    \]
    Since $h \gg m \gg j$, we see that any derivative landing on the 
    denominator will produce a factor of size $\dyad^{-h}$, so the 
    integral kernel will have size $\dyad^{-h}$ by Lemma \ref{lem:sep}.
    This shows that the symbol has integral kernel of size $\dyad^{-h}\dyad^{-m}$.

    ~

    For $b^{3,\hihi}_{h,m,j}$ we introduce cutoffs
    $p_{\sim h}(\xi_1 - \xi_2)p_{\sim h}(\xi_2 - \xi_3)$ 
    which does not modify the symbol 
    since $c^{\hihi}$ is supported near $\xi_1, \xi_3 \sim \dyad^h$ and 
    $\xi_1 - \xi_2 + \xi_3 \sim \dyad^j \ll \dyad^h$.
    Thus using the same argument as above proves that 
    $b^{3,\hihi}_{h,m,j}$ has kernel of size $\dyad^{-2h}.$

\end{proof}

\begin{corollary}
   The trilinear form $B^3$  constructed above satisfies Definition 
    \ref{def:good_corrections}.
\end{corollary}
\begin{proof}
    We prove these estimates in turn from the symbol size of $B^3$ in 
    Proposition \ref{prop:b3_kernel_size}. Note that using Lemma \ref{lem:sep} 
    introduces translations, but since $L^p$ and 
    the assumed bilinear $L^2_{t,x}$ norms are 
    translation invariant, we drop these translations in the argument below.
    In each of the following estimates $B^{3,\hihi}$
    is better than $B^{3,\himed}$ so for the sake of concision, 
    we prove the estimates only for $B^{3,\himed}$.
    
    For the bound \eqref{eq:bn_pointwise_L1} we have 
    \begin{align*}
        \|P_j B^{3,\himed}(f_1, \bar f_2, f_3)\|_{L^1_x}
        &\lesssim \sum_{h \gg m \gg j} \dyad^{-h}\dyad^{-m}
        \|P_j (P_m(P_h f_1 P_h \bar f_2) P_m f_3)\|_{L^1_x}\\
        &\lesssim \sum_{h \gg m \gg j} \dyad^{-h}\dyad^{-m/2}
        \|P_h f_1 P_h \bar f_2\|_{L^1_x}\|P_m f_3\|_{L^2_x}\\
        &\lesssim \sum_{h \gg m \gg j} \dyad^{-h}\dyad^{-m/2}
        \|P_h f_1\|_{L^2_x} \|P_h \bar f_2\|_{L^2_x}\|P_m f_3\|_{L^2_x}.
    \end{align*}
    Then the bound \eqref{eq:bn_pointwise_L1} 
    follows by Bernstein's inequality and the bootstrap 
    estimate \eqref{eq:uj_ee_boot}. 

    The $L^2$ version of \eqref{eq:bn_pointwise} can be proven 
    directly from the above estimate and Bernstein's inequality, 
    but the $L^6$ version has 
    different numerology:
    \begin{align*}
        \|P_j B^{3,\himed}(f_1, \bar f_2, f_3)\|_{L^6_x}
        &\lesssim \sum_{h \gg m \gg j} \dyad^{-h}\dyad^{-m}
        \|P_j (P_m(P_h f_1 P_h \bar f_2) P_m f_3)\|_{L^6_x}\\
        &\lesssim \sum_{h \gg m \gg j} \dyad^{-h}\dyad^{-m/3}
        \|P_m(P_h f_1 P_h \bar f_2)\|_{L^{3/2}_x}\|P_m f_3\|_{L^6_x}\\
        &\lesssim \sum_{h \gg m \gg j} \dyad^{-h}\|P_h f_1\|_{L^2_x}
        \|P_h f_2\|_{L^2} \|P_m f_3\|_{L^6_x}.
    \end{align*}

    For the first bilinear estimate, we pair the high and medium frequencies 
    which are separated roughly $\dyad^h$ (or in the case of $B^{3,\hihi}$ 
    we have this separation because of the projection to lower frequency).
    \begin{align*}
        \|P_jB^{3}(u, \bar u, u)\|_{L^2_{t,x}} 
        &\lesssim \dyad^{j/2}\sum_{h \gg m \gg j}\dyad^{-h}\dyad^{-m}
        \|P_j(P_h u P_h \bar u P_m u)\|_{L^2_tL^1_x}\\
        &\lesssim \dyad^{j/2}\sum_{h \gg m \gg j}\dyad^{-h}\dyad^{-m}
        \|P_h u\|_{L^\infty_t L^2_x}
        \dyad^{-h}\|\partial(P_h \bar u P_m u)\|_{L^2_tL^2_x}\\
        &\lesssim C^2\epsilon^3\dyad^{j/2}\sum_{h \gg m \gg j}\dyad^{(-3/2-2s)h}
        \dyad^{(-1-s)m} c_h^2c_m\\
        &\lesssim C^2\epsilon^3c_j\dyad^{(-2-3s)j}
    \end{align*}
    which suffices because $s > -1/2$.

    For the mixed bilinear estimate, we calculate 
    \begin{align*}
        \|\partial(P_jB^{3}(u, \bar u, u)P_k \bar v)\|_{L^2_{t,x}} 
        &\lesssim \sum_{h \gg m \gg j}\dyad^{-h}\dyad^{-m}
        \|\partial(P_j(P_h u P_h \bar u P_m u)P_k \bar v)\|_{L^2_tL^2_x}.
    \end{align*}
    Now, depending on the balance of $h$ and $k$, we opt for two different estimates.
    If $h \gtrsim k$ then we use the above argument by simply placing $P_k v$ 
    in $L^\infty_t L^2_x$. Note that the derivative has size $\max(\dyad^j, \dyad^k)$,
    but Bernstein can be lost on the smaller of the two sizes. Thus the overall 
    loss will be $(\dyad^{j/2}+ \dyad^{k/2})\dyad^{j/2}\dyad^{k/2}$.
    \begin{align*}
        &\sum_{k \lesssim h \gg m \gg j}\dyad^{-h}\dyad^{-m}
        \|\partial(P_j(P_h u P_h \bar u P_m u)P_k \bar v)\|_{L^2_tL^2_x}\\
        &\lesssim (\dyad^{j/2}+ \dyad^{k/2})\dyad^{j/2}\dyad^{k/2}
        \sum_{k \lesssim h \gg m \gg j}\dyad^{-h}\dyad^{-m}
        \|P_j(P_h u P_h \bar u P_m u)\|_{L^2_tL^2_x}\|P_k \bar v\|_{L^\infty_tL^2_x}\\
        &\lesssim C^3\epsilon^4 (\dyad^{j/2}+ \dyad^{k/2})\dyad^{j}
        \dyad^{(1/2-s)k}c_k\sum_{k \lesssim h \gg m \gg j} \dyad^{(-3/2-2s)h}
        \dyad^{(-1-s)m} c_h^2c_m\\
        &\lesssim C^3\epsilon^4(\dyad^{j/2}+ \dyad^{k/2})\dyad^{-sj}\dyad^{(-2-3s)k}
        c_jc_k
    \end{align*}
    which suffices since $s > -1/2$.

    Otherwise, we estimate the square of the norm:
    
    \begin{align*}
        &\sum_{k \gg h \gg m \gg j}\dyad^{-h}\dyad^{-m}
        \|\partial(P_j(P_h u P_h \bar u P_m u)P_k \bar v)\|_{L^2_tL^2_x}\\
        &\lesssim \dyad^{k}
        \sum_{k \gg h \gg m \gg j}\dyad^{-h}\dyad^{-m}
        \|P_j(P_h u P_h \bar u P_m u)P_j(P_h \bar u P_h u P_m \bar u) 
        P_k\bar v P_k v\|^{1/2}_{L^1_{t,x}}\\
        &\lesssim \dyad^{k}
        \sum_{k \gg h \gg m \gg j}\dyad^{-h}\dyad^{-m}
        \|P_j(P_h \bar u P_h u P_m \bar u)\|^{1/2}_{L^\infty_{t,x}}
        \|P_m u\|^{1/2}_{L^\infty_{t,x}}
        \|P_h uP_k\bar v \|_{L^2_{t,x}}\\
        &\lesssim C^3\epsilon^4 \dyad^{j/2}
        \dyad^{(1/2-s)k}c_k\sum_{k \gg h \gg m \gg j} \dyad^{(-1-2s)h}
        \dyad^{(-1/2-s)m} c_h^2c_m\\
        &\lesssim C^3\epsilon^4 \dyad^{k/2}\dyad^{(-1-3s)j}\dyad^{-sk}
        c_jc_k
    \end{align*}
    which suffices because $s > -1/2$ and completes the proof.

\end{proof}

If we now consider the Schr\"odinger equation 
for $v = u + B^3(u,\bar u, u)$,
\[
\begin{aligned}
    (i\partial_t + \partial_x^2) v = & \ C(u, \bar u , u) - C^{\nonres}(u,\bar u , u)
    + B^3(C(u,\bar u, u), \bar u, u)  
    \\ & - B^3(u, \ol{C(u,\bar u, u)}, u)  
    + B^3(u,\bar u, C(u,\bar u, u))  
\end{aligned}
\]
we see that while we have removed the contribution of $C^{\nonres}$ to cubic 
terms, it appears in quintic terms. This will create the same problem as 
$C^{\nonres}$. We will demonstrate this for 
$B^3(C^{\nonres,\himed}(u, \bar u, u), \bar u, u)$ which we need 
to control as in \eqref{eq:source_term}.

In particular, after applying Proposition \ref{prop:b3_kernel_size} 
and separating variables, the expression

$P_jB^3(C^{\nonres,\himed}(u, \bar u, u), \bar u, u) P_j \bar u$ 
is of the form
\[
    \sum_{h_1 \gg m_1 \gg h_2 \gg m_2 \gg j} \dyad^{-h_2}\dyad^{-m_2}
    P_j(P_{h_2}(P_{h_1}uP_{h_1}\bar u P_{m_1} u) P_{h_2} \bar u P_{m_2} u).
\]
With the bootstrap estimates \eqref{eq:uj_ee_boot}-\eqref{eq:uj_sep_bi_boot}, 
the best estimate we can achieve pairs the first four frequencies in bilinear 
estimates and uses Bernstein's inequality on the fifth giving 
\[
    C^4\epsilon^6\sum_{h_1 \gg m_1 \gg h_2 \gg m_2 \gg j} 
    \dyad^{(-1 -2s)h_1}c_{h_1}^2\dyad^{-sm_1}c_{m_1}\dyad^{(-1 - s)h_2}c_{h_2}
    \dyad^{(-1/2 - s)m_2}c_{m_2}.
\]
Since $s < 0$, we see that the sum in $m_1$ of $\dyad^{-sm_1}$ converges to 
$\dyad^{-sh_1}$. And for $\dyad^{(-1-3s)h_1}$ to converge, we would need 
$s > -1/3$. 

Thus, we see that correcting $C^{\nonres}$ did not improve the convergence of our 
source terms. However, we again may check for non-resonance by calculating 
$\Delta^6\xi^2$ for 
$B^3(C^{\nonres,\himed}(u, \bar u, u), \bar u, u)$. If we localize 
to frequencies $h_1, m_1, h_2, m_2, $ and $j$ as above, we see
\begin{align*}
    \Delta^6\xi^2 &= \xi_1^2 - \xi_2^2 + \xi_3^2 - (\xi_1 - \xi_2 + \xi_3)^2
+ (\xi_1 - \xi_2 + \xi_3)^2 - \xi_4^2 + \xi_5^2\\
    &\sim \Delta^4\xi^2 + \dyad^{2h_2}\\
    &\sim \dyad^{h_1}\dyad^{m_1} + \dyad^{2h_2}\\
    &\sim \dyad^{h_1}\dyad^{m_1}.
\end{align*}
Thus, the contribution to $\Delta^6\xi^2$ is dominated by the inner factor 
coming from 
$C^{\nonres,\himed}$, which we can correct with $B^3$. In particular,
we should be able to correct $B^3(C^{\nonres}(u,\bar u, u), \bar u, u)$ by 
$B^3(B^3(u,\bar u, u), \bar u , u)$. Of course, the time derivative will 
produce higher order terms which can be corrected similarly.

This motivates a normal form change of variables of the form 
\[
    v = u + B^3(u, \bar u , u) + B^3(B^3(u, \bar u, u), \bar u, u) + \cdots
\]
which can be more cleanly written in the implicit form 
\[v = u + B^3(v,\bar v, v).\]
With this change of variables the equation \eqref{nls} becomes 
\[
(i\partial_t + \partial_x^2)(v - B^3(v,\bar v, v)) = 
C(v - B^3(v,\bar v, v),\ol{v - B^3(v,\bar v, v)},v - B^3(v,\bar v, v)).
\]
After using the relation \eqref{eq:cubic_correction} and collecting 
terms we have
\begin{equation}\label{eq:cubic_equation}
    (I-L_v)((i\partial_t + \partial_x^2)v) = C^{\bal}(v,\bar v,v) 
    + C^{\low}(v,\bar v ,v)
    + C^{\res}(v,\bar v, v) + C^5(v) + C^7(v) + C^9(v).
\end{equation}
Here the real-linear operator $L_v$ is defined by 
\[L_v(f) =  B^3(f, \bar v, v) - B^3(v ,\bar f , v) + B^3(v,\bar v, f)\]
and the nonlinear terms $C^5$, $C^7$, and $C^9$ are defined by 
\[C^5(v) = -C(B^3(v,\bar v , v), \bar v, v) - 
C(v, \ol{B^3(v,\bar v, v)}, v) - C(v, \bar v, B^3(v,\bar v , v)),\]
\[C^7(v) = C(B^3(v,\bar v , v), \ol{B^3(v,\bar v, v)}, v) + 
C(v, \ol{B^3(v,\bar v, v)}, B^3(v,\bar v , v)) + C(B^3(v,\bar v , v), \bar v, B^3(v,\bar v , v)),\]
and
\[C^9(v) = -C(B^3(v,\bar v , v), \ol{B^3(v,\bar v, v)}, B^3(v,\bar v , v)).\]
To make equation \eqref{eq:cubic_equation} look like a Schr\"odinger equation 
we have to apply $(I - L_v)^{-1}$ to both sides (which will be justified 
by the smallness of $B$ when we prove the linearized estimates 
\eqref{eq:bn_pointwise}). We can use 
the Neumann series:
\begin{equation}\label{eq:neumann}
(I - L_v)^{-1} = \sum_{k = 0}^\infty L_v^k.
\end{equation}
This gives the expansion 
\begin{equation}\label{eq:cubic_equation_expanded+}
    (i\partial_t + \partial_x^2)v = \sum_{k=0}^\infty L_v^k(C^{\bal}(v,\bar v,v)
    + C^{\low}(v,\bar v, v)+ C^{\res}(v,\bar v, v) + C^5(v) + C^7(v) + C^9(v)).
\end{equation}
Note that in \eqref{eq:cubic_equation_expanded+} the only contribution of 
$C^{\nonres}$ occurs in $C^5, C^7, $ and $C^9$ where the medium frequency 
of $C^{\nonres}$ is not the highest, so we have eliminated the worst of 
our problems.

Of course, we would like every source term in \eqref{eq:cubic_equation_expanded+} 
to obey the estimate \eqref{eq:source_term}, but there are many terms to check. 
In fact there will be several important terms which will not satisfy 
\eqref{eq:source_term}.

Luckily, many of these terms satisfy Definition \ref{def:good_L6_terms} and 
so we will immediately get the estimate \eqref{eq:source_term} for the 
sum of powers of $L_v$ since $B^3$ satisfies Definition
\ref{def:good_corrections}. The following proposition shows which 
of the source terms in \eqref{eq:cubic_equation_expanded+} are favourable: 

\begin{proposition}\label{prop:b3_good_bad_splitting}
    First, we have that $C^{\low}$ and $C^{\res}$ satisfy 
    the source term estimate \eqref{eq:source_term} of Theorem 
    \ref{thm:unbal_correction}.
    
    Further, we have that the following part of 
    $L^1_v(C^{\low} + C^{\bal} + C^{\res})$ is good:
    \begin{equation*}
\begin{aligned}
    &L^1_v(C^{\low} + C^{\bal} + C^{\res}) -
        \Big(B^{3,\himed}(C^{\bal}(v, \bar v, v) ,\bar v, v) - 
    B^{3,\himed}(v, \ol{C^{\bal}(v, \bar v, v)}, v)\Big)\\  
    &- \Big(\sum_h B^{3,\himed}(P_{\cong h}v, P_{\cong h}\bar v , 
    C^{\res}(P_{\cong h} v, P_{\cong h}\bar v, v))\Big)
\end{aligned}
    \end{equation*}
    satisfies Definition \ref{def:good_L6_terms}.

    Further, the following part of $C^5$ is good:
    \begin{align*}
        -C^5 &-\Big(\sum_{h} C^{\low}(P_{\cong h}v, \ol{B^{3,\himed}(P_{\cong h} v, 
    P_{\cong h} \bar v, v)}, P_{\cong h} v)\Big)\\
        -&\Big((C^{\nonres, \himed} + C^{\res, \himed})
    (P_{\cong h}v, P_{\cong h}\bar v , B^{3,\himed}(P_{\cong h}v, P_{\cong h}\bar v, v))\Big)
    \end{align*}
    satisfies Definition \ref{def:good_L6_terms}.

    Also, the first bad term above, 
\[
\sum_{h} C^{\low}(P_{\cong h}v, \ol{B^{3,\himed}(P_{\cong h} v, 
    P_{\cong h} \bar v, v)}, 
        P_{\cong h} v)
        \]
    still satisfies \eqref{eq:source_term}, and in addition
    \[L_v^1
\Big(\sum_{h} C^{\low}(P_{\cong h}v, \ol{B^{3,\himed}(P_{\cong h} v, 
    P_{\cong h} \bar v, v)}, 
        P_{\cong h} v)\Big)
        \]
    satisfies Definition \ref{def:good_L6_terms}.

    Lastly, $C^7$ and $C^9$ unambiguously satisfy 
    Definition \ref{def:good_L6_terms}.
 
    Note that by Proposition \ref{prop:b_preserves_good}
    $L^{k}_v$ applied to the above good terms will also be good.

\end{proposition}
\begin{remark}
    After this proposition, we have two groups of important source terms in
    \eqref{eq:cubic_equation_expanded+}: those which 
    satisfy the source term estimate \eqref{eq:source_term} but are not good,
    and those which do not satisfy the source term estimate \eqref{eq:source_term}. 
    In the first, we have the terms
    \begin{equation}\label{eq:b3_terms_to_check}
        C^{\bal}, \quad C^{\low}, \quad C^{\res}, \quad 
        \sum_{h} C^{\low,\hihi}(P_{\cong h}v, 
        \ol{B^{3,\himed}(P_{\cong h} v, P_{\cong h} \bar v, v)}, 
        P_{\cong h} v).
    \end{equation}
    (Note that $C^{\bal}$ falls in this list because of our analysis
     in Section \ref{s:energy}.)
    And in the second we have the terms
    \begin{equation}\label{eq:b3_remaining_bad_terms}
\begin{aligned}
&B^{3,\himed}(C^{\bal}(v, \bar v, v) ,\bar v, v) - 
B^{3,\himed}(v, \ol{C^{\bal}(v, \bar v, v)}, v)\\  
&+ \sum_h B^{3,\himed}(v_{\cong h}, \bar v_{\cong h} , 
    C^{\res}(v_{\cong h}, \bar v_{\cong h}, v))
    -(C^{\nonres, \himed} + C^{\res, \himed})
    (v_{\cong h}, \bar v_{\cong h} , B^{3,\himed}(v_{\cong h}, \bar v_{\cong h}, v)).
\end{aligned}
    \end{equation}
Both groups will appear in the remaining analysis, although the focus will 
    be on terms in \eqref{eq:b3_remaining_bad_terms} as they demand the most nuanced 
    analysis.
\end{remark}
\begin{proof}[Proof of Proposition~\ref{prop:b3_good_bad_splitting}]
 Note that the order of the terms above is at most a fixed finite number, 
    in this case 9, so we need not worry about the growth rates of the 
    constant in Definition \ref{def:good_L6_terms} or the source term 
    estimate \eqref{eq:source_term}. 

    Since there are many cases we will outline the general strategy
    and give details for the edge cases. 

    First, whenever there are two pairs of frequency separations at 
    the level of the highest frequency we can use two bilinear 
    estimates \eqref{eq:uj_sep_bi_boot}. This happens for instance 
    when $B^3$ falls on the highest frequency or the highest frequencies 
    of $B^3$ and $C$ are unmatched. One situation where this only barely 
    happens is in 
    \[
        P_h C^{\low}(P_{\lesssim h} v,  P_{\lesssim h}\bar v, 
        B^{3,\himed}(P_{\cong h}v, P_{\cong h}\bar v, v))
    \]
    because here we must have that 
    \[
        \dyad^h \sim \xi_1 - \xi_2 + \xi_3 - \xi_4 + \xi_5 
        \sim \xi_1 - \xi_2 + \dyad^\ell \sim \xi_1 - \xi_2.
    \]
    Thus, even though $\xi_1$ and $\xi_2$ are localized in nearby dyadic 
    regions, they must always have separation. 

    Another delicate example of this is involving the error 
    term in $C^{\low}$ (see Definition \ref{def:c_splitting}) which 
    we will call $C^{err}$:
    \[
        \sum_{h_1 \gg h_2 \gg h_3 \gg m} 
        B^{3,\himed}(P_{h_3}C^{err}(P_{\cong h_1} v,  P_{\cong h_1}\bar v, P_{h_2}v), 
        P_{h_3} \bar v, P_{m} v).
    \]
    Here, the enhanced smoothness condition (H1s) tells us 
    this symbol has size $\dyad^{-h_1/2} \dyad^{-h_2/2}\dyad^{-m}$.
    Here the strategy described above takes the following form where 
    we have omitted spatial translations for concision.
    \begin{align*}
       \sum_{h_1 \gg h_2 \gg h_3 \gg m} 
        \| B^{3,\himed}&(P_{h_3}C^{err}(P_{\cong h_1} v,  P_{\cong h_1}\bar v, P_{h_2}v), 
        P_{h_3} \bar v, P_{m} v) \|_{L^1_{t,x}}\\
        &\lesssim \sum_{h_1 \gg h_2 \gg h_3 \gg m} 
        \dyad^{-h_1/2}\dyad^{-h_2/2}\dyad^{-m}
        \|P_{\cong h_1} \bar v P_{\cong h_2} v\|_{L^2_{t,x}}
        \|P_{\cong h_1} vP_{h_3} \bar v \|_{L^2_{t,x}}
        \|P_{m} v\|_{L^\infty_{t,x}}\\
        &\lesssim C^3\epsilon^5
        \sum_{h_1 \gg h_2 \gg h_3 \gg m} 
        \dyad^{(-3/2-2s)h_1}\dyad^{(-1/2-s)h_2}\dyad^{-sh_3}\dyad^{(-1/2 - s)m}c_{h_1}^2c_{h_2}c_{h_3}c_m.
    \end{align*}
    This suffices because $h_1 \geq h_3$. 

    The other case is when a term involves the error 
    term of $C^{\low}$ when high frequencies are matched or in 
    any term involving $B^{3,\hihi}$. 
    Here, we get extra smallness from the size $\dyad^{-h/2}\dyad^{-m/2}\dyad^{j}$ 
    of the error term or respectively Proposition \ref{prop:b3_kernel_size} 
    which 
    allows us to use one bilinear estimate \eqref{eq:uj_sep_bi_boot} and three 
    $L^6$ estimates \eqref{eq:uj_se_boot}
    with only one separation. Since $B^{3,\hihi}$ and the error 
    term in $C^{\low}$ 
    always comes with a separation these are harmless. As an example we analyze
    \[
    P_j B^{3,\himed}(P_{\cong h} v, P_{\cong h} \bar v, 
    P_{m_1} C^{err}(P_{\cong h} v,P_{\cong h} \bar v, P_{m_2}v))
\]
    Here, the symbol size is 
    $\dyad^{-3h/2}\dyad^{-m_2/2}$ by Proposition \ref{prop:b3_kernel_size},
    and the enhanced smoothness condition (H1s) and 
    we may estimate by Lemma \ref{lem:sep} and the 
    bootstrap estimates \eqref{eq:uj_ee_boot}-\eqref{eq:uj_sep_bi_boot}
    \begin{align*}
        \|P_j B^{3,\himed}&(P_{\cong h} v, P_{\cong h} \bar v, 
        P_{m_1} C^{err}(P_{\cong h} v,P_{\cong h} \bar v, P_{m_2}v))\|_{L^1_{t,x}}\\
        &\lesssim \dyad^{-3h/2}\dyad^{-m_2/2} \sup_{x_0}
        \|P_{\cong h} v P_{m_2} v\|_{L^2_{t,x}}
        \|P_{\cong h} v\|_{L^6_{t,x}}^3\\
        &\lesssim C^4\epsilon^4\dyad^{(-3/2-3s)h}\dyad^{(-1/2 -s)m_2} 
        c_h^3c_{m_2}
    \end{align*}
    which suffices after summing since $s > -1/2$. Note that in this 
    term, the enhancement of (H1s) is required in this case.

    Now we consider the terms which do not have good estimates. 
    In $B^{3,\himed}(C^{\bal}(v, \bar v, v), \bar v , v)$ as well as its symmetric 
    counterpart, we find ourselves near a situation where the first four 
    frequencies are very near each other and the fifth frequency is 
    much smaller, so we have 
    only one separation at high frequency, which gains at most 
    $\dyad^{-h/2}$. 
    Since $B^{3,\himed}$ only has size $\dyad^{-h}\dyad^{-m}$ this means 
    we can only control three high frequencies, but we have four. 

    Similarly, in $B^{3,\himed}(v, \bar v, C^{\res}(v, \bar v, v))$ 
    as well as in its symmetric counterpart 
    which is
    $-(C^{\res} + C^{\nonres})(v, \bar v, B^{3,\himed})$, we may find ourselves 
    in a situation where all high frequencies are matched although here 
    they come in two pairs.
    In particular
    we do have control of this term when there is separation among the 
    high frequencies, which is why the bad term reduces to 
    \[
\Big(\sum_{h} C^{\low}(P_{\cong h}v, \ol{B^{3,\himed}(P_{\cong h} v, 
    P_{\cong h} \bar v, v)}, 
        P_{\cong h} v)\Big).
    \]
    This is also the reason why the similar term
    $B^{3,\himed}(P_{\cong h} v, P_{\cong h}\bar v, 
    C^{\low}(P_{\cong h}v, \bar v, P_{\cong h}v))$
    is in fact good, because the high frequencies of $C^{\low}$, in this regime, 
    have opposite signs, and thus separation.

    The terms $C^{\res}$ and $C^{\low}$ can never be good since they are 
    trilinear, but easily satisfy the source term estimate 
    \eqref{eq:source_term} since they both have two high frequency separations
    after being paired against as in \eqref{eq:source_term}, and have at most 
    two frequencies higher than the output frequency. In $C^{\res}$ and 
    the error part of $C^{\low}$ it is clear why we have two separations at the highest 
    frequency after pairing. 
    
    For the remaining part of $P_j C^{\low}$ we have
    all three frequencies less or comparable to $j$, with one much less. There 
    are two cases up to symmetry. If $\xi_3$ is much less, then 
    $\xi_1 - \xi_2 \sim \dyad^j$ which means $\xi_1$ and $\xi_2$ as well as 
    $\xi_3$ and the measurement frequency are separated. Otherwise, if $\xi_2$ is 
    much less, then $\xi_1 + \xi_3 \sim \dyad^j$ and either one is smaller than $\dyad^j$
    and we have two separations, or $\xi_1 \sim \xi_3 \sim \dyad^j/2 \ll \dyad^j$ which 
    puts us in the first case. In either case we may verify the source 
    term estimate \eqref{eq:source_term}:
    \begin{align*}
        \|P_j C^{\low}(v, \bar v, v) P_{\cong j} \bar v\|_{L^1_{t,x}}
        &\lesssim \sum_{j_1, j_2\lesssim j} \dyad^{-2j}\sup_{x_0, x_1}
        \|\partial_x (P_{\sim j} v P_{j_1} \bar v(\cdot + x_0))\|_{L^2_{t,x}}
        \|\partial_x (P_{\cong j} v P_{j_2} \bar v(\cdot + x_1))\|_{L^2_{t,x}}\\
        &\lesssim C^2\epsilon^4c_j^2 \dyad^{(-1-4s)j}.
    \end{align*}

    Finally, we have the outlier 
    \[
        C^{\low}(v, \ol{B^{3,\himed}(v, \bar v, v)}, v)
    \]
    where, unlike the case of $C^{\low}$ above, knowing that
    $\xi_1 + \xi_5 \sim \dyad^j$ does not enforce separation. 
    However, this term does satisfy the source term estimate \eqref{eq:source_term} 
    because if there is no separation between 
    the high frequencies, then the output frequency of this outlier term is 
    \[
    \xi_1 - \xi_2 + \xi_3 -\xi_4 + \xi_5 \sim 2\dyad^h.
    \]
    In particular, in \eqref{eq:source_term}, we measure against a much higher frequency,
    and thus there are two separations: one between the medium frequency of $B^{3,\himed}$
    and a size $\cong h$ frequency, and one against the high output frequency.

    Lastly, it is easy to see that 
    \[
        L^1_v(C^{\low}(P_{\cong h} v, \ol{B^{3,\himed}(P_{\cong h} v, P_{\cong h} \bar v, v)}, P_{\cong h} v))
    \]
    satisfies the good estimate as either the high frequency of 
    this term is unmatched with the high frequency of $B^3$ and 
    we have several lower frequency terms, or it is matched and we have 
    a separation from $C^{\low,\hihi}$ and one from $B^3$.

\end{proof}


\subsection{Balanced Quintic Corrections}
\label{ss:quintic_corrections}

Here we discuss the corrections we need for the first set of 
bad terms from \eqref{eq:b3_remaining_bad_terms}:
\[ 
    B^{3,\himed}(C^{\bal}(v, \bar v, v) ,\bar v, v) - 
    B^{3,\himed}(v, \ol{C^{\bal}(v, \bar v, v)}, v)  
\]

It will be useful to consider these terms as a single multilinear form so 
we make the following definition.
\begin{definition}\label{def:c5}
    We define the 5 linear form
\[
C^{5,\bal}(v,\bar v,v,\bar v,v) = B^{3,\himed}(C^{\bal}(v,\bar v, v),\bar v, v) 
- B^{3,\himed}(v, \ol{C^{\bal}(v,\bar v, v)}, v).
\]
Put
\[\xi_6 = \xi_1 - \xi_2 + \xi_3 - \xi_4 + \xi_5\]
and 
\[\eta_1 = \xi_1 - \xi_2 + \xi_3 - \xi_4, \quad 
\eta_2 = \xi_1 - \xi_2 - \xi_3 + \xi_4, \quad 
\eta_3 = \xi_1 + \xi_2 - \xi_3 - \xi_4, \quad 
\eta_4 = \xi_1 + \xi_2 + \xi_3 + \xi_4.\]
    Then, we may take the symbol of $c^{5,\bal}$ to be 
\begin{align*}
c^{5,\bal}(\xi_1, \xi_2, \xi_3, \xi_4, \xi_5) 
= &\frac12 c^{\bal}(\xi_1,\xi_2,\xi_3)b^{3,\himed}(\xi_1-\xi_2+\xi_3, \xi_4, \xi_5)\\
    -&\frac12 \ol{c^{\bal}}(\xi_2,\xi_3,\xi_4)b^{3,\himed}(\xi_1,\xi_2-\xi_3+\xi_4,\xi_5)\\
+&\frac12 c^{\bal}(\xi_1,\xi_4,\xi_3)b^{3,\himed}(\xi_1-\xi_4+\xi_3, \xi_2, \xi_5)\\
    -&\frac12 \ol{c^{\bal}}(\xi_2,\xi_1,\xi_4)b^{3,\himed}(\xi_3,\xi_2-\xi_1+\xi_4,\xi_5)\\
\end{align*}
    where we have symmetrized the first four inputs. 
    Here, it will be convenient to partition 
    the symbol into dyadic regions which respect the symmetry between 
    $\xi_1, \xi_2, \xi_3$ and $\xi_4$. 
    Or expanding out Definitions \ref{def:c_splitting} 
    and \ref{def:b3_nonres} and using the new coordinates, we may take
\begin{align*}
    c^{5,\bal}_{h,m,j}&(\xi_1, \xi_2, \xi_3, \xi_4, \xi_5) = 
    -p_j(\xi_6)p_{m}(\eta_1)p_h(\eta_4/4)\sum_{j_1 \sim h, j_2 \sim h, j_3 \sim h,
     j_4 \sim h} p_{j_1}(\xi_1) p_{j_2}(\xi_2)p_{j_3}(\xi_3) p_{j_4}(\xi_4)\\
&\Big[
    \frac{c^{\bal}(\xi_1,\xi_2,\xi_3)
    c^{\nonres,\himed}(\eta_1 + \xi_4,\xi_4,-\eta_1)}
{\eta_1^2 + 2\eta_1\xi_4 + \xi_5^2 - \xi_6^2}
    +\frac{c^{\bal}(\xi_1,\xi_4,\xi_3)
    c^{\nonres, \himed}(\eta_1 + \xi_2,\xi_2,-\eta_1)} {\eta_1^2 + 2\eta_1\xi_2 + \xi_5^2 - \xi_6^2}\\
    &-\frac{\ol{c^{\bal}}(\xi_2,\xi_3,\xi_4)
    c^{\nonres,\himed}(\xi_1,\xi_1-\eta_1,-\eta_1)}
{-\eta_1^2 + 2\eta_1\xi_1  + \xi_5^2 - \xi_6^2}
    -\frac{\ol{ c^{\bal}}(\xi_2,\xi_1,\xi_4)
    c^{\nonres, \himed}(\xi_3,\xi_3-\eta_1,-\eta_1)}
{-\eta_1^2 + 2\eta_1\xi_3  + \xi_5^2 - \xi_6^2}
\Big]
\end{align*}
so that 
    \[
        c^{5,\bal} = \sum_{h \gg m \gg j} c^{5,\bal}_{h,m,j}.
    \]
\end{definition}

The reason that $C^{5,\bal}$
fails to satisfy the source term estimate, namely \eqref{eq:source_term}, 
follows from the fact that since
the four high frequencies can be close to each other
we can not simultaneously use Bernstein's inequality at low 
frequencies and use the bilinear $L^2_{t,x}$ estimate to make a gain at high frequency.
In particular, the structure of the first two terms above, after measurement 
against $P_j v$, is schematically
\[\sum_{h \gg m \gg j} P_j(P_h(v_h \bar v_h v_h) \bar v_h v_m) \bar v_j.\] 

If we use the Strichartz estimate \eqref{eq:uj_se_boot},
placing 3 high frequencies in $L^6_{t,x}$, bilinear $L^2_{t,x}$ between
a $\bar v_h$ and $v_m$ and Bernstein on $\bar v_j$, we heuristically get
\begin{align*}
    \sum_{h \gg m \gg j} \|P_j(P_h(P_h v P_h \bar v P_h v) P_h \bar v P_m v) 
    P_j \bar v\|_{L^1_{t,x}}
    &\lesssim C^5\epsilon^4 \sum_{h \gg m \gg j} 
\dyad^{(-3/2 + (1 - 4s)/2 -s)h}\dyad^{(-1-s)m}\dyad^{(1/2-s)j}c_h^4c_mc_j\\
    &=C^5\epsilon^4c_j\sum_{h \gg m \gg j} 
\dyad^{(-1 -3s)h}\dyad^{(-1-s)m}\dyad^{(1/2-s)j}c_h^4c_m.
\end{align*}
This only converges if $s > -1/3$. If instead we use two bilinear estimates, 
the result would be worse, as we would be forced to place Bernstein at high 
frequency.

Thus, we are motivated to take advantage of the cancellation in $C^{5,\bal}$
between the positive and negative components.
We expect this to be possible because this difference is morally similar to 
\[(C^{\bal}(v,\bar v, v)\bar v - v\ol{C^{\bal}(v,\bar v, v)})\]
which we have seen in Section \ref{s:energy} admits a favourable division. 
The situation is complicated by the symbol of $B^{3,\himed}$ and the presence of the 
medium frequency, but we are assisted by only needing to 
do the division modulo good terms. In particular, up to a good term, we will be 
able to apply Lemma \ref{l:division}. 

\begin{proposition}\label{prop:b5bal}
    $c^{5,\bal}_{h,m,j}$ admits a decomposition of the form 
    \[ 
    c^{5,\bal}_{h,m,j} + \Delta^4\xi^2 b^{5,\bal}_{h,m,j} 
    = g^{5,\bal}_{h,m,j},
    \]
    where $\Delta^4 \xi^2$ is taken with respect to the 
    first four frequencies, and such that $G^{5,\bal}_{h,m,j}$ satisfies Definition 
    \ref{def:good_L6_terms} and 
    \[
    \|\check b^{5,\bal}_{h,m,j}\|_{L^1} \lesssim \dyad^{-3h}\dyad^{-m}.
    \]
\end{proposition}
\begin{proof}
The first step is to factor out a good term so that we may apply Lemma 
    \ref{l:division} in the first four inputs. We define 
\begin{align*}
    c^{5,\bal,simp}_{h,m,j}&(\xi_1, \xi_2, \xi_3, \xi_4, \xi_5) = 
    -p_j(\xi_6)p_{m}(\eta_1)p_h(\eta_4/4)\sum_{j_1 \sim h, j_2 \sim h, j_3 \sim h,
     j_4 \sim h} p_{j_1}(\xi_1) p_{j_2}(\xi_2)p_{j_3}(\xi_3) p_{j_4}(\xi_4)\\
    &\frac1{\eta_1}\Big[
    \frac{c^{\bal}(\xi_1,\xi_2,\xi_3)
    c^{\nonres,\himed}(\eta_1 + \xi_4,\xi_4,-\eta_1)}
{\eta_1 + 2\xi_4}
    +\frac{c^{\bal}(\xi_1,\xi_4,\xi_3)
    c^{\nonres, \himed}(\eta_1 + \xi_2,\xi_2,-\eta_1)} {\eta_1 + 2\xi_2}\\
    &-\frac{\ol{c^{\bal}}(\xi_2,\xi_3,\xi_4)
    c^{\nonres,\himed}(\xi_1,\xi_1-\eta_1,-\eta_1)}
{-\eta_1 + 2\xi_1}
    -\frac{\ol{ c^{\bal}}(\xi_2,\xi_1,\xi_4)
    c^{\nonres, \himed}(\xi_3,\xi_3-\eta_1,-\eta_1)}
{-\eta_1 + 2\xi_3}
\Big]
\end{align*}
and wish to prove that the difference will have the good estimate of 
Definition \ref{def:good_L6_terms}.

The analysis will be symmetric for each of the four terms. We note that 
    \[
        \frac{1}{\pm \eta_1^2 + 2\eta_1\xi_4 + \xi_5^2 - \xi_6^2}
        - \frac{1}{\pm \eta_1^2 + 2\eta_1\xi_4 }
        = \frac{\xi_5^2 - \xi_6^2}{(\pm \eta_1^2 + 2\eta_1\xi_4 + \xi_5^2 - \xi_6^2)
        (\pm \eta_1^2 + 2\eta_1\xi_4 )}
    \]
    which is separable by Lemma~\ref{lem:sep} with size 
    \[
        \frac{\dyad^{2m}}{\dyad^{2h}\dyad^{2m}} = \dyad^{-2h}.
    \]
    This extra smallness allows us, assuming the bootstrap estimates
    \eqref{eq:uj_se_boot} and \eqref{eq:uj_sep_bi_boot}, to estimate 
    \begin{align*}
        \|P_j(C^{5,\bal} - C^{5,\bal,simp})\|_{L^1_{t,x}}
        & \lesssim \sum_{h\gg m \gg j} \dyad^{-2h} \|P_h v\|_{L^6_{t,x}}^3
        \sup_{x_0}\|P_h v P_m \bar v(\cdot + x_0)\|_{L^2_{t,x}}\\
        & \lesssim \epsilon^4\sum_{h \gg m \gg j} \dyad^{(-5/2 + (1-4s)/2-s)h}\dyad^{-sm}
        c_h^4c_m\\
        & \lesssim \epsilon^4\dyad^{(-2 -4s)j}c_j.
    \end{align*}
    We let the sum of these symbols be $g$.

    Next, we note that $\frac{\eta_1}{p_j(\xi_6)p_m(\eta_1)} c^{5,\bal,simp}_{h,m,j}$ 
    depends only on $\xi_1, \xi_2, \xi_3, \xi_4$ and vanishes on the diagonal
    and so we may apply Lemma \ref{l:division} to get 
    \[
     c^{5,\bal}_{h,m,j} + \frac1{\eta_1} {p_j(\xi_6)p_m(\eta_1)}\Delta^4 \xi^2 b =
    \frac1{\eta_1} {p_j(\xi_6)p_m(\eta_1)}(\xi_o - \xi_e)^2 q + {p_j(\xi_6)p_m(\eta_1)}r + g
    \]
    with 
    \[\|\check b\|_{L^1} \lesssim \dyad^{-3h}\]
    \[\|\check q\|_{L^1} \lesssim \dyad^{-3h}\]
    \[\|\check r\|_{L^1} \lesssim \dyad^{-2h}.\]
    Note that compared to the result of Lemma 
    \ref{l:division}, each term has an extra factor of $\dyad^{-h}$ 
    which comes 
    from the fact that the overall symbol size of $c^{5,\bal,simp}$ has 
    an extra factor of $\dyad^{-h}$.

    Now, in the support of $c^{5,\bal}_{h,m,j}$ we have $\eta_1 \sim \dyad^m$, so by 
    multiplying through by a larger cutoff to this region, we may have this support 
    property in all terms. In particular, we need not worry about the division
    by $\eta_1,$ and may set 
    \[
        b^{5,\bal}_{h,m,j} = \frac1{\eta_1} b
    \]
    which will have the correct support properties.
    Finally we need to show that $r$ and $q$ both satisfy Definition 
    \ref{def:good_L6_terms}. 

    For $Q$ associated to the symbol $\frac1{\eta_1} (\xi_o - \xi_e)^2 q$ we have 
    \begin{align*}
        \|Q\|_{L^1_{t,x}}
        & \lesssim \sum_{h\gg m \gg j} \dyad^{-3h}\dyad^{-m} 
        \sup_{x_0}\|\partial(P_{\cong h} v P_{\cong h} \bar v(\cdot + x_0))\|_{L^2_{t,x}}^2
        \|P_m v \|_{L^\infty_{t,x}}\\
        & \lesssim \epsilon^4\sum_{h \gg m \gg j} \dyad^{(-2 -4s)h}\dyad^{(-1/2-s)m}
        c_h^4c_m\\
        & \lesssim \epsilon^4\dyad^{(-5/2 -5s)j}c_j.
    \end{align*}
    And for $R$ associated to the symbol $r$ we have
    \begin{align*}
        \|R\|_{L^1_{t,x}}
        & \lesssim \sum_{h\gg m \gg j} \dyad^{-2h} \|P_h v\|_{L^6_{t,x}}^3
        \sup_{x_0}\|P_h v P_m \bar v(\cdot + x_0)\|_{L^2_{t,x}}\\
        & \lesssim \epsilon^4\sum_{h \gg m \gg j} \dyad^{(-5/2 + (1-4s)/2-s)h}\dyad^{-sm}
        c_h^4c_m\\
        & \lesssim \epsilon^4\dyad^{(-2 -4s)j}c_j
    \end{align*}
    Thus we may set 
    \[g^{5,\bal}_{h,m,j} = i\frac1{\eta_1} (\xi_o - \xi_e)^2 q + ir + g\]
    which completes the proof.
\end{proof}

Then we set 
\[
    B^{5,\bal} = \sum_{h\gg m \gg j} B^{5,\bal}_{h,m,j}
\]
and 
\[
    G^{5,\bal} = \sum_{h\gg m \gg j} G^{5,\bal}_{h,m,j}
\]

We may use the frequency support and the above symbol size conditions of $B^{5,\bal}$
to verify the conditions of Propositions \ref{prop:contraction} and 
\ref{prop:b_preserves_good}.
\begin{corollary}
    $B^{5,\bal}$ as constructed above satisifes Definition 
    \ref{def:good_corrections}.
\end{corollary}
\begin{proof}
    These bounds follow in an almost identical fashion to the bounds 
    for $B^3$ so we omit the proof.
\end{proof}

Thus we are in a good position to use $B^{5,\bal}$ as a correction.

In particular we have
\begin{align*}
    (i\partial_t + \partial_x^2) B^{5,\bal}(v,\ldots, v) &= -C^{5,\bal}(v, \ldots, v)
    + G^{5,\bal}(v,\ldots, v) \\
                            &+B^{5,\bal}((i\partial_t + \partial_x^2)v, \bar v, \ldots, v) 
                            -B^{5,\bal}(v, \ol{(i\partial_t + \partial_x^2)v}, \ldots, v)\\
                            &+\cdots + B^{5,\bal}(v, \bar v, \ldots, (i\partial_t + \partial_x^2)v) 
\end{align*}

Then we amend our normal form correction to 
\begin{equation}\label{eq:normal_form_quintic}
    v = u + B^3(v, \bar v, v) + B^{5,\bal}(v,\bar v, v, \bar v, v)
\end{equation}
Plugging this into \eqref{nls} and expanding the Neumann series gives 
\begin{align}\label{eq:quintic_equation_expanded}
    (i\partial_t + \partial_x^2) v = C^\bal + 
    \sum_{k=0}^\infty (L_v)^k\Big(
    &B^{3,\hihi}(C^{\bal},\bar v, v) 
    -B^{3,\hihi}(v, \ol{C^{\bal}}, v) 
    +B^{3,\hihi}(v ,\bar v, C^{\bal})\\
    \nonumber &+ B^{5,\bal}(C^\bal, \bar v, v, \bar v, v) - \cdots 
    + B^{5,\bal}(v, \bar v, v, \bar v, C^{\bal})\\
    \nonumber &+ B^{3,\himed}(v, \bar v, C^\bal) + C^{\res}
    + C^{\low} + G^{5,\bal} + C^{\geq5}\Big)
\end{align}
where 
\begin{align*}
    L_{v}(f) = &B^3(f,\bar v, v) - B^3(v, \bar f, v) + B^3(v,\bar v, f)\\
    &+ B^{5,\bal}(f,\bar v, v,\bar v , v) - B^{5,\bal}(v, \bar f, v, \bar v , v) + \cdots + B^{5,\bal}(v,\bar v,v,\bar v, f)
\end{align*}
and 
\begin{align*}
    C^{\geq 5} = 
               C(v - {B^3-B^{5,\bal}}, \ol{v - B^3-B^{5,\bal}}, {v - B^3-B^{5,\bal}})
               - C(v, \bar v , v).
\end{align*}

As in Section \ref{ss:cubic_corrections}, we would like to 
determine which source terms are good. In Proposition 
\ref{prop:b3_good_bad_splitting}, we studied 
terms involving 
$B^3$. Here we will show that $B^{5,\bal}$ landing on 
the bad terms from \eqref{eq:b3_terms_to_check} 
produces good terms and that the new parts of $C^{\geq 5}$ are also good.

\begin{proposition}\label{prop:b5_good_bad_splitting}
    $B^{5,\bal}$ applied to the terms from \eqref{eq:b3_terms_to_check} 
    are good:
    \[
        B^{5,\bal}\Big(C^{\bal} + C^{\low} + C^{\res} + 
        \sum_h C^{\low}(P_{\cong h} v, 
        \ol{B^{3,\himed}(P_{\cong h} v, P_{\cong h} \bar v, v)}, P_{\cong h} v)
        , \bar v, v, \bar v, v\Big), \ldots,
    \]
and
\[
        B^{5,\bal}\Big(v, \bar v, v, \bar v, C^{\bal} + C^{\low} + C^{\res} + 
        \sum_h C^{\low}(P_{\cong h} v, 
        \ol{B^{3,\himed}(P_{\cong h} v, P_{\cong h} \bar v, v)}, P_{\cong h} v)
        \Big)
    \]
    satisfies Definition \ref{def:good_L6_terms}.

    In addition $G^{5,\bal}$ is good.
    
    Finally, the following part of $C^{\geq 5}$ is good:
    \begin{align*}
        -C^{\geq 5} &-\Big(\sum_{h} C^{\low}(P_{\cong h}v, \ol{B^{3,\himed}(P_{\cong h} v, 
    P_{\cong h} \bar v, v)},
        P_{\cong h} v)\Big)\\
        -&\Big((C^{\nonres, \himed} + C^{\res})
    (P_{\cong h}v, P_{\cong h}\bar v , B^{3,\himed}(P_{\cong h}v, P_{\cong h}\bar v, v))\Big)
    \end{align*}
    satisfies Definition \ref{def:good_L6_terms}.

    Also, after another iteration, the first bad term above is 
    good:
    \[L_v
\Big(\sum_{h} C^{\low}(P_{\cong h}v, \ol{B^{3,\himed}(P_{\cong h} v, 
    P_{\cong h} \bar v, v)}, 
        P_{\cong h} v)\Big)
        \]
    satisfies Definition \ref{def:good_L6_terms}.
\end{proposition}
\begin{remark}
    Unlike Proposition \ref{prop:b3_good_bad_splitting}, $B^{5,\bal}$ produces 
    only good terms through $L_v$. In particular the bad terms that remain are 
    the same as before except that we have now removed 
    $B^{3,\himed}(C^{\bal}, \bar v, v) - B^{3,\himed}(v, \ol{C^{\bal}}, v)$.
    That is \eqref{eq:b3_terms_to_check} remains unchanged, but 
    \eqref{eq:b3_remaining_bad_terms} is modified to 
    \begin{equation}\label{eq:b5_remaining_bad_terms}
\begin{aligned}
& \sum_h B^{3,\himed}(v_{\cong h}, \bar v_{\cong h} , 
    C^{\res}(v_{\cong h}, \bar v_{\cong h}, v))
    -(C^{\nonres, \himed} + C^{\res, \himed})
    (v_{\cong h}, \bar v_{\cong h} , B^{3,\himed}(v_{\cong h}, \bar v_{\cong h}, v)).
\end{aligned}
    \end{equation}
\end{remark}
\begin{proof}
    Again, we outline the strategy used to apply the good estimates 
    in Definition \ref{def:good_L6_terms}. The most important 
    new term which is good is $B^{5,\bal}(C^{\bal}, \bar v, v, \bar v, v)$ 
    since the corresponding term for $B^3$ in fact required correction. This 
    term can be estimated with the $L^6_{t,x}$ estimate \eqref{eq:uj_se_boot}.

    \begin{align*}
        \|P_j B^{5,\bal}(C^{\bal}(v,\bar v ,v), \bar v, v, \bar v, v)\|_{L^1_{t,x}}
        &\lesssim \sum_{h \gg m \gg j} \dyad^{-3h}\dyad^{-m}
        \|P_{\cong h} v\|_{L^6_{t,x}}^6 \|P_{m} v\|_{L^\infty_{t,x}}\\
        &\lesssim C^3\epsilon^5
        \sum_{h \gg m \gg j} \dyad^{(-2 - 4s)h}\dyad^{(-1/2 - s)m}c_h^4c_m
    \end{align*}
    which suffices by the slowly varying condition and $s > -1/2$.
    A similar estimate suffices for 
    \[B^{5,\bal}(P_{\cong h}v, P_{\cong h}\bar v, P_{\cong h}v , P_{\cong h}\bar v,
    C^{res}(P_{\lesssim h} v, P_{\lesssim h} \bar v, P_{m} v))\]
    and 
    \[C(P_{\lesssim h} v, P_{\lesssim h}, B^{5,\bal}(P_{\cong h}v, P_{\cong h}\bar v, P_{\cong h}v , P_{\cong h}\bar v, P_m v))\]

    Again, when there are two separations
    at the highest frequency, we can use two bilinear estimates 
    \eqref{eq:uj_sep_bi_boot}. For instance, 
    $B^{5,\bal}(C^{\res}, \bar v, v, \bar v, v)$ has two separations between
    the highest two frequencies $C^{\res}$ and the less high frequencies 
    of $B^{5,\bal}$ 
    In particular we have 
    \begin{align*}
        &\|P_j B^{5,\bal}(C^{\res}, \bar v, v, \bar v, v)\|_{L^1_{t,x}}\\
&\lesssim \sum_{h_1 \gg h_2 \gg m \gg j} \dyad^{-3h_2}\dyad^{-m}
        \sup_{x_2}
        \|P_{\cong h_1} v P_{\cong h_2} \bar v(\cdot + x_2))\|^2_{L^2_{t,x}}
        \|P_{\cong h_2} v\|_{L^\infty_{t,x}}^2 
        \|P_m v\|_{L^\infty_{t,x}}\\
        &\lesssim C^5\epsilon^7
        \sum_{h_1 \gg h_2 \gg m \gg j} 
        \dyad^{(-1-2s)h_1}\dyad^{(-2 - 4s)h_2}\dyad^{(-1/2 - s)m}
        c_{h_1}^2c_{h_2}^4c_m
    \end{align*}
    which suffices by the slowly varying condition and $s > -1/2$.

    When this does not happen, we resort to 
    one bilinear estimate \eqref{eq:uj_sep_bi_boot} and three $L^6_{t,x}$ estimates 
    \eqref{eq:uj_se_boot}. Here, by Proposition
    \ref{prop:b5bal}, terms involving $B^{5,\bal}$ or $G^{5,\bal}$ have 
    enough factors of $\dyad^{-h}$ to apply this strategy.

    We use as example 
    the term $C^{\low}(v, \ol{B^{5,\bal}}, v)$ whose 
    counterpart in Proposition \ref{prop:b3_good_bad_splitting} failed to be good. 

    Here, when the high frequencies are matched, 
    we may place all of $B^{5,\bal}$ in $L^1_tL^\infty_x$ by Bernstein's 
    inequality, and place the separation in $B^{5,\bal}$ in bilinear $L^2_{t,x}$:

    \begin{align*}
        &\|P_j C^{\low}(v, \ol{B^5}, v)\|_{L^1_{t,x}}\\
        &\lesssim \sum_{j\sim h \gg m \gg \ell} 
        \|P_{\sim j}v\|_{L^\infty L^2_x}\|P_{\sim j}v\|_{L^\infty_tL^2}
        \|P_\ell B^5(P_{\sim h} v, P_{\sim h}\bar v, 
        P_{\sim h}v, P_{\sim h}\bar v, P_{m} v)\|_{L^1_tL^\infty_x}\\
        &\lesssim \sum_{j \gg m \gg \ell} \dyad^{-3j}\dyad^{-m}\dyad^{\ell}
        \sup_{x_0}
        \|\partial(P_{\sim j} v P_{m} \bar v(\cdot + x_0))\|_{L^2_{t,x}}
        \|P_{\sim j}v\|^2_{L^\infty L^2_x} \|P_{\sim j}v\|^3_{L^6_{t,x}}\\
        &\lesssim C^6\epsilon^6\sum_{j \gg m \gg \ell} \dyad^{(-7/2 -3s + (1-4s)/2)j}
        \dyad^{(-1 - s)m}\dyad^{\ell}c_j^6c_m
    \end{align*}
    which again suffices since $s > -1/2$, noting that the dyadic summation becomes critical exactly as $s$ approaches  $-\frac12$.

\end{proof}


\subsection{Higher Order Corrections}
\label{ss:higher_corrections}
\newcommand\ord[1]{{2(#1)+1}}
Here we analyze the two remaining bad terms from \eqref{eq:b5_remaining_bad_terms}
\[
\sum_h B^{3,\himed}(v_{\cong h}, \bar v_{\cong h} , C^{\res}(v_{\cong h}, 
\bar v_{\cong h}, v))
-(C^{\nonres,\himed}+C^{\res,\himed})(v_{\cong h}, \bar v_{\cong h} , B^{3,\himed}(v_{\cong h}, \bar v_{\cong h}, v))
\]

Note that in the above, we are assuming pairing of the high frequencies 
because, as we saw in Proposition~\ref{prop:b3_good_bad_splitting}, when 
the high frequencies are separated, we may use bilinear estimates to get control. 

\subsubsection{Structure of the remaining terms}

Luckily, or perhaps unluckily, 
the implicit normal form employed above produces two terms which have 
parallel problems.

These terms are the result of the implicit normal form attempting to 
correct the quintic term $B^3(u, \bar u, C(u, \bar u, u))$ with a correction 
\[B^3(u, \bar u, B^3(u, \bar u, u)).\]
When we take the time derivative of this correction we get 
\[
    -C^{\nonres}(u, \bar u , B^3(u, \bar u, u)) 
    - B^3(u, \bar u, C^{\nonres}(u, \bar u, u)).
\]
The second term cancels part of the bad quintic term at the expense of an
error. 
In other words, the implicit normal form assumes that the innermost $C$ 
dominates the behavior of the term, and that the outer $B^3$ is lower order. This
is in fact a good assumption when the medium frequency from $C$ is much 
larger than the corresponding output frequency of $C$, i.e. the $C^{\nonres,\himed}$
part of $C$. 
And indeed this is exactly the 
case the implicit normal form corrects in. 

However, we have not eliminated all of the bad nonresonant part of 
$B^3(u,\bar u, C(u,\bar u, u))$. In particular, when 
the high frequencies are matched, we can have nonresonance when $C$ is resonant
when $B^3$ dominates. This happens when the medium frequency from $C$ is below
its output, but above the overall output of the quintic term. Algebraically
this is because, even though $\Delta^4\xi^2$ for the inner $C$ vanishes as 
$C$ becomes resonant, we have that $\Delta^6\xi^2$ is still bounded 
away from zero. We will continue by calculating this discrepancy more precisely.

To simplify the notation 
in what follows we make the following definitions:
\[\alpha_1 = \xi_{1} + \xi_2, \qquad \delta_1 = \xi_1 - \xi_2\]
\[\alpha_2 = \xi_{3} + \xi_4, \qquad \delta_2 = \xi_3 - \xi_4.\]
We will 
reserve simply $\xi$ for the last unpaired frequency. 
In this section there will be many frequencies. So to be a bit more clear,
we will reserve $\ell$ for the low output frequency, $m$ for the 
several medium frequencies, and $h$ for the high frequency.

That is, we will let $\dyad^\ell$ be the dyadic size of 
the overall frequency or $\xi + \delta_1 + \delta_2 $, 
$\dyad^{m_k}$ be the dyadic size of $\delta_k$, 
$\dyad^m$ be the dyadic size of the unpaired medium
frequency $\xi$, and 
$\dyad^h$ be the dyadic size of all $\alpha_k$.
Note that the $\alpha_k$ are high frequencies (since they \emph{average} nearby 
high frequencies) and the $\delta_k$ can be small (since they 
subtract nearby high frequencies) thus justifying the notation.

Using this notation, we can see clearly why $\Delta^6\xi^2$ for 
$B^{3,\himed}(u,\bar u, C(u, \bar u, u))$ is not well approximated by 
$\Delta^4 \xi^2$ for the inner $C(u, \bar u, u)$ 
if the $\alpha$'s from each are matched 
and the medium frequency from $C$ (i.e. $\delta_2$) is small.

Note that $B^{3,\himed}$ is localized to situations where 
$\ell \ll m_1.$
This means that
either $\delta_1 + \delta_2 \sim \dyad^m$, or $\dyad^m \lesssim \dyad^\ell$. In 
the second case we have control, so we focus on the first.
\begin{align*}
    \Delta^6 \xi^2 &= \xi_1^2 - \xi_2^2+ \xi_3^2- \xi_4^2 + \xi^2 
    - (\xi_1-\xi_2+\xi_3-\xi_4 + \xi)^2\\
&= \alpha_1\delta_1 + \alpha_2\delta_2  
+ \xi^2 - (\delta_1 + \delta_2 + \xi)^2\\
&= \frac12 (\alpha_1 + \alpha_2)(\delta_1+\delta_2) + \frac12 (\alpha_1 -\alpha_2)
(\delta_1 - \delta_2) 
+ \xi^2 - (\delta_1 + \delta_2 + \xi)^2\\
\end{align*}
When $\alpha_1 \approx \alpha_2$ we see that this is dominated by 
$\frac12 (\alpha_1 + \alpha_2)(\delta_1+\delta_2) \sim \dyad^h \dyad^m.$
On the other hand 
\begin{align*}
    \Delta^4 \xi^2 &= \xi_3^2- \xi_4^2 + \xi^2 - (\xi_3-\xi_4+\xi)^2\\
&= \alpha_2\delta_2 + \xi^2 - (\delta_2 + \xi)^2
\end{align*}
will have size $\alpha_2\delta_2 \sim \dyad^h\dyad^{m_2}$. 
Thus there is a discrepancy when 
$m_2 \not \sim m.$

\subsubsection{Corrections to the implicit normal form}

Even though this discrepancy is at its worst when $m_2 \not \sim m$, it will 
simplify the analysis
if we correct the term $B^3(u,\bar u, C(u, \bar u, u))$ 
more precisely over the entire correction region instead of approximating the 
correction by 
$B^3(u, \bar u, B^3(u, \bar u, u))$. We wish to replace the 
implicit normal form for these terms, but only in the situation where there 
is matching of the high frequencies and only when the correction lands on the 
medium frequency term of $B^3$. The cleanest way to accomplish this is to 
subtract the bad correction. It is easy to see how to 
do this here but inspired by higher order computations, we introduce the 
following framework to decide what must be subtracted. 

Recall the implicit normal form was given by 
\[v = u + B^3(v,\bar v, v) + B^{5,\bal}(v, \bar v, v, \bar v, v)\]
This corresponds to a series of usual normal forms by replacing 
each $v$ by the normal form which will expand one order. But we were unhappy with 
the expansion in the last variable of $B^3$. So let us not be implicit in that 
input. This gives 
\[v = u + B^3(v,\bar v, u) + B^{5,\bal}(v, \bar v, v, \bar v, v).\]
This is too brutal however, because we still want to use the implicit normal 
form when the nonlinearity hitting that term has unmatched high frequencies, so 
we try to replace the $u$ in $B^3$ by $v$ according to this equation to get more 
control. In particular, we see that 
\[(I - T) u = v - B^{5,\bal}(v, \bar v, v, \bar v, v)\]
where 
\[T u = - B^3(v, \bar v, u).\] 
Or using Neumann series, 
\begin{equation}\label{eq:normal_form_correction}
    u = \sum_{k=0}^\infty T^k (v - B^{5,\bal}) = v - B^{5,\bal} - B^3 
+ B^{3}(v, \bar v, B^{5,\bal}) + B^3(v,\bar v, B^3) + \cdots
\end{equation}

Thus, we see that not using the implicit normal form in the third variable 
is equivalent to the implicit normal form, but with extra terms:
\[v = u + B^3(v,\bar v, \sum_{k=0}^\infty T^k (v - B^{5,\bal})) + B^{5,\bal}(v, \bar v, v, \bar v, v).\]
Note that this recovers the na\"ive idea to subtract off the bad correction because 
this contains the term
\[-B^3(v, \bar v, B^3(v,\bar v, v))\]
but also has higher order terms which we will come to later:
\[
    -B^3(v, \bar v, B^{5,\bal}) + B^3(v, \bar v, B^3(v, \bar v, B^3 + B^{5,\bal}))
    - \cdots.
\]
However, we want to modify the operator $T$ to only include the part of 
$B^3(v, \bar v, B^3)$ which 
has matched high frequencies. In particular, we want to subtract 
\[
    \begin{aligned}
        -\sum_h &B^3(P_{\cong h}v, P_{\cong h}\bar v, B^3(P_{\cong h}v, 
    P_{\cong h}\bar v, v) 
        + B^{5,\bal}(P_{\cong h}v, P_{\cong h} \bar v, P_{\cong h} v, P_{\cong h} \bar v, v))\\
        &+ B^3(P_{\cong h}v, P_{\cong h}\bar v, B^3(P_{\cong h}v, P_{\cong h}\bar v, B^3(P_{\cong h}v, 
    P_{\cong h}\bar v, v) + B^{5,\bal}(P_{\cong h}v, P_{\cong h} \bar v, P_{\cong h} v, P_{\cong h} \bar v, v)))\\
        &- \cdots.
    \end{aligned}
\]

Lastly, we will need only finitely 
many terms from the $T$ sum, which we will pin down later. 

\subsubsection{First of the higher order corrections}
Going forward, to simplify the notation we will make new names 
for the forms
    $C^{\himed}, C^{\himed, \nonres},$ and $C^{\himed, \res}$
forms in a way that will better generalize to higher orders and in
the notation we have developed in the section so far.
\begin{definition}
    We define the following symbols corresponding to 
    $C^{\himed}, C^{\himed, \nonres},$ and $C^{\himed, \res}$:
\[c^{\ord1}_h(\xi) = \sum_{\ell \ll h} \sum_{h \gg m_1} p_{\ell}(\delta_1 + \xi) 
    p_{h}(\alpha_1) p_{m_1}(\delta_1) c^\himed(\alpha_1, -\delta_1),\]
\[c^{\ord1,\nonres}_h(\xi) = \sum_{\ell \ll h} \sum_{h \gg m_1 \gg \ell} 
    p_{\ell}(\delta_1 + \xi) 
    p_{h}(\alpha_1) p_{m_1}(\delta_1) c^\himed(\alpha_1, -\delta_1),\]
    and
\[c^{\ord1,\res}_h(\xi) = \sum_{m_1 \lesssim \ell \ll h} 
    p_{\ell}(\delta_1 + \xi) 
    p_{h}(\alpha_1) p_{m_1}(\delta_1) c^\himed(\alpha_1, -\delta_1),\]
where here the $(1)$ represents how many pairs of high frequencies there are,  
the $h$ represents the specific high frequency that is being matched, and we 
use the shorthand 
\[c^\himed(\alpha_i, \xi) := c^\himed(\frac12 (\alpha_i + \delta_i), \frac12 (\alpha_i - \delta_i), \xi).\]
    Further, we define the symbol $b^{\ord1}_h$ corresponding to 
    the correction $b^{3,\himed}$ 
\[b^{\ord1}_h = -\sum_{\ell \ll h} \sum_{h \gg m_1 \gg \ell} 
    p_{\ell}(\delta_1 + \xi) 
    p_{h}(\alpha_1) p_{m_1}(\delta_1) \frac{c^\himed( 
    \alpha_1, -\delta_1)}{\alpha_1\delta_1 + \delta_1^2 - 2\delta_1(\delta_1 + \xi)}.
\]
We also will define the associated multilinear forms in the usual way. 
\end{definition}
\begin{remark}
Note that since the individual high frequencies plays no role we have abused notation 
by writing
    $c^{\ord1}_h(\xi)$ as if it is a function of a single variable instead of three.

    Note that here as in Section \ref{ss:cubic_corrections}, 
    we have modified the symbol of $b^{\ord1}_h$ to depend less heavily on 
    $\xi$. This will be convenient because the symbol of 
    $B^3(v, \bar v, B^3(v, \bar v , v))$ will correspond more closely 
    to the product of the two symbols, since $\xi$ is replaced by 
    $\xi_3 - \xi_4 + \xi_5$ in such a composition. We will 
    see more clearly in Definition \ref{def:c2_h} how we will take advantage 
    of this product structure.

Unfortunately, we cannot remove dependence on $\xi$
completely since we cannot eliminate the lowest projection and we cannot remove 
$\xi$ from the factor $\Delta^4\xi^2$ in the denominator. This is because the error 
after commuting with $i\partial_t + \partial_x^2$ is of size 
\[\frac{\xi^2 - (\delta_1 + \xi)^2}{\alpha_1 \delta_1} \sim 
\frac{\dyad^m}{\dyad^h}\]
which does not have enough decay in $\dyad^m$ to get the source term estimate
    \eqref{eq:source_term}. Note
however that we have expressed $\Delta^4\xi^2$ in a way which will be helpful 
later on.
\end{remark}

We will further simplify notation by 
dropping the input to all these forms when it is simply $v$. Since 
only the last input will be important we have dropped the dependence on the 
first two. In particular, our new candidate for implicit normal form may be written
as
\[v = u + \sum_h B^{\ord1}_h - B^{\ord1}_h\circ B^{\ord1}_h 
+ B^{3,\hihi}(v,\bar v, v) + B^{5,\bal}(v, \bar v, v, \bar v, v).\]
Note that here we further abuse notation by writing 
$\sum_h B^{\ord1}_h\circ B^{\ord1}_h$ 
to mean
\[\sum_{h_1 \cong h_2} B^{\ord1}_{h_1}(v, \bar v, B^{\ord1}_{h_2}(v, \bar v, v))\]
which is suggestive of the composition of the two forms, but specifically 
in this last input.
We will continue to use this notation below.

Now, with this new notation, our equation for $v$ becomes 
\begin{align*}
    (i\partial_t + \partial_x^2) v &= 
    C^{\bal} + C^{\res} + C^{\low} + C^{\low}(v, \ol{B^{3,\himed}}, v)\\ 
    &\sum_{k=0}^\infty L_v^k \Big(G + \sum_{h} B^{\ord1}_h \circ C^{\ord1,\res}_h 
- C^{\ord1}_h \circ B^{\ord1}_h\\
    &- [i\partial_t + \partial_x^2, B^{\ord1}_h \circ
B^{\ord1}_h] + C^{\ord1}_h \circ B^{\ord1}_h \circ B^{\ord1}_h\Big)
\end{align*}
where $G$ is a sum of good terms, including errors from replacing 
$C$ with $C^{\ord 1}_h$ which we will analyze in more detail in the next section. 
We also use the notation $[i\partial_t + \partial_x^2, F^{2n+1}]$
to refer to 
\[
    (i\partial_t + \partial_x^2) F^{2n+1}(v, \ldots, v) 
    - F^{2n+1}((i\partial_t + \partial_x^2)v, \bar v, \ldots, v) 
    - F^{2n+1}(v, \ol{(i\partial_t + \partial_x^2)v}, v, \ldots, v) 
    - \cdots  
\]
which we treat as a commutator.

We compute that the commutator above is 
\[[i\partial_t + \partial_x^2, B^{\ord1}_h \circ B^{\ord1}_h] = 
- C^{{\ord1},\nonres}_h \circ B^{\ord1}_h - 
B^{\ord1}_h \circ C^{{\ord1},\nonres}_h\]
so we may simplify the bad terms in the sum as
\begin{align*}
     \sum_{h} &B^{\ord1}_h \circ C^{{\ord1}}_h 
    - C^{{\ord1}, \res}_h \circ B^{\ord1}_h +
    C^{\ord1}_h \circ B^{\ord1}_h \circ B^{\ord1}_h.
\end{align*}
Of these terms, the first will require correction; the second 
has the source term estimate \eqref{eq:source_term} because 
the uncontrolled medium frequency is below the output frequency; 
and the third is higher order which we will leave for later.

We would like to treat the first term $B^{\ord 1}_h \circ C^{\ord 1}_h$
as being parallel to $C^{\ord 1}_h$ but higher order, 
which prompts the following definition.

\begin{definition}\label{def:c2_h}
    We define $C^{\ord 2}_h$ according to the symbol 
\begin{align*}
    c^{\ord2}_h &= -\sum_{\ell \ll h_1 \cong h_2} \sum_{h_1 \gg m_1 \gg \ell} 
    p_{\ell}(\delta_1 + \delta_2 + \xi) 
    p_{h_1}(\alpha_1) p_{m_1}(\delta_1) \frac{c^{\himed}(\alpha_1, -\delta_1)}{\alpha_1\delta_1}\\
& \cdot \sum_{h_2  \gg m_2} 
    p_{h_2}(\alpha_2) p_{m_2}(\delta_2) c^{\himed}(\alpha_2 , -\delta_2).
\end{align*}
    We also split $C^{\ord 2}_h$ into near resonant and nonresonant parts according 
    to the relationship between $m_2$ and $\ell$, noting that if $m_2 \gg \ell$
    then we may symmetrize:
\begin{align*}
    c^{{\ord 2},\nonres}_h &= -\frac12 
\frac{\alpha_1\delta_1 + \alpha_2\delta_2}{\alpha_1\delta_1\alpha_2\delta_2}
\sum_{h_1\cong h_2}\sum_{h_1 \gg m_1 \gg \ell , h_2 \gg m_2 \gg \ell} 
    p_{\ell}(\delta_1 + \delta_2 + \xi) 
    p_{h_1}(\alpha_1) p_{h_2}(\alpha_2)\\
                    &\qquad \cdot p_{m_1}(\delta_1) p_{m_2}(\delta_2)
    c^{\himed}(\alpha_1, -\delta_1)c^{\himed}(\alpha_2 , -\delta_2)
\end{align*}
\begin{align*}
c^{{\ord 2},\res}_h
&=-\sum_{\ell \ll h_1 \cong h_2} \sum_{h_1 \gg m_1 \gg \ell} 
    p_{\ell}(\delta_1 + \delta_2 + \xi) 
    p_{h_1}(\alpha_1) p_{m_1}(\delta_1) \frac{c^{\himed}(\alpha_1, -\delta_1)}{\alpha_1\delta_1}\\
& \qquad \cdot \sum_{h_2 \gg m_2 \lesssim \ell} 
    p_{h_2}(\alpha_2) p_{m_2}(\delta_2) c^{\himed}(\alpha_2 , -\delta_2) 
\end{align*}
\end{definition}
\begin{remark}
    Notice that in $c^{{\ord 2}, \nonres}_h$ we have symmetrized 
    between the two pairs. This was only possible because 
    we replaced $c^{\himed}(\alpha_2, \xi)$ by $c^{\himed}(\alpha_2, -\delta_2)$ 
    in the previous definitions. This symmetrization is essential because 
    it allows us to see that $c^{{\ord 2}, \nonres}_h$ is nonresonant because 
    we may factor out $\Delta^6\xi_2 \sim \alpha_1\delta_1 + \alpha_2\delta_2$ 
    which only appeared because of symmetry. 

    This replacement is the reason we generated the error terms in $C^{\low, \himed}$
    which in turn is the reason for the enhanced smoothness condition (H1s).
\end{remark}

Of course, $C^{\ord 2}_h$ is not precisely equal to $B^{\ord 1}_h \circ C^{\ord 1}_h$.
Instead, they differ by a good error.

\begin{proposition}\label{prop:c2_h_approx}
    $C^{\ord 2}_h$ is well approximated by  $B^{\ord 1}_h \circ C^{\ord 1}_h$
    in the sense that their difference satisfies 
    Definition \ref{def:good_L6_terms}.
\end{proposition}
\begin{proof}
We analyze the symbol of $B^{\ord 1}_h \circ C^{\ord 1}_h$:
\begin{align*}
    b^{\ord 1}_h c^{\ord 1}_h &= -\sum_{\ell \ll h_1 \cong h_2} \sum_{h_1 \gg m_1 \gg \ell} 
    p_{\ell}(\delta_1 + \delta_2 + \xi) 
    p_{h_1}(\alpha_1) p_{m_1}(\delta_1) \frac{c^{\himed}(\alpha_1, -\delta_1)}
{\alpha_1\delta_1 + \delta_1^2 - 2\delta_1(\delta_1 + \delta_2 + \xi)}\\
& \cdot \sum_{j \ll h_2} \sum_{h_2 \gg m_2} p_{j}(\delta_2 + \xi) 
    p_{h_2}(\alpha_2) p_{m_2}(\delta_2) c^{\himed}(\alpha_2 , -\delta_2).
\end{align*}
    We seek to remove the projection $p_j(\delta_2 + \xi)$ and replace the 
    denominator $\alpha_1\delta_1 + \delta_1^2 - 2\delta_1(\delta_1 + \delta_2 + \xi)$
    by $\alpha_1\delta_1$.

First, we remove the projection $p_j(\delta_2 + \xi)$.
Note that the summand would be good if we allowed $j \gtrsim h$ since 
\[\dyad^j \sim (\delta_2 + \xi) + (\delta_1) \sim \dyad^j + \dyad^{m_1}\]
and since $j \ll m_1 \ll h$, if $j \gtrsim h$ either the summand is $0$
or $m_1$ is comparable to $h$ for a slightly larger definition than 
    $\sim$. But if $m_1$ is comparable to $h$ then we have a factor of
    $\dyad^{-2h}$ making this term good. 
    In summary, we may freely unrestrict the $j$ sum and use the fact that 
\[\sum_j p_j( \cdot ) = 1.\]
Finally, we may replace the denominator by $\alpha_1\delta_1$ since 
\begin{align*}
    \frac1{\alpha_1\delta_1 + \delta_1^2 - 2\delta_1(\delta_1 + \delta_2 + \xi)} - \frac1{\alpha_1\delta_1}
&= \frac{\delta_1^2 - 2\delta_1(\delta_1+\delta_2+\xi)}
{(\alpha_1\delta_1)(\alpha_1\delta_1 + \delta_1^2 - 2\delta_1(\delta_1 + \delta_2 + \xi))}\\
&= \frac{1 - 2(\delta_1+\delta_2+\xi)/\delta_1}
{(\alpha_1)(\alpha_1 + \delta_1 - 2(\delta_1 + \delta_2 + \xi))}\\
    &\lesssim {\dyad^{-2h}}
\end{align*}
since $h \gg m_1 \gg \ell$. This makes this term good as well
because of the extra factor of $\dyad^h$ in the denominator.

This concludes the proof.
\end{proof}

The final ingredient we need before generalizing this is to understand what is the optimal 
way to correct $c^{{\ord 2},\nonres}_h$ up to a good error. Recall 
that for $c^{{\ord 1},\nonres}_h$ we were forced to correct by exactly $\Delta\xi^2$. 
Here doing this would be a problem because we cannot always guarantee, 
up to a good error estimate, that
\[
    \alpha_1\delta_1 + \alpha_2\delta_2 \sim \Delta^6\xi^2
    ={\alpha_1\delta_1 + \alpha_2\delta_2 
+ (\delta_1+\delta_2)^2 - 2(\delta_1+\delta_2)(\delta_1+\delta_2 + \xi)}.\]
This is because there may be cancellation so that 
$\alpha_1\delta_1 + \alpha_2\delta_2 \ll \dyad^{2m} .$
However, this only happens when $\delta_1 \sim \delta_2$. In 
this case, we can actually handle not correcting by the full symbol 
$\Delta^6\xi^2$ since 
the error here will be at worst 
\[
    \frac{-(\delta_1 + \delta_2)^2}{\alpha_1\delta_1 \alpha_2\delta_2} 
    \sim \frac1{\alpha_1\alpha_2} \sim \dyad^{-2h}.
\]
Thus the intuition 
is that we correct by $\alpha_1\delta_1 + \alpha_2\delta_2$ when 
$\delta_1 \sim \delta_2$ and by the full symbol $\Delta\xi^2$ otherwise
\begin{align*}
    b^{{\ord 2},\nonres}_h &= \frac12 
\frac1{\alpha_1\delta_1\alpha_2\delta_2}
\sum_{h_1\cong h_2}\sum_{h_1 \gg m_1 \gg \ell , h_2 \gg m_2 \gg \ell} 
    p_{\ell}(\delta_1 + \delta_2 + \xi) 
    p_{h_1}(\alpha_1) p_{h_2}(\alpha_2)\\
                    &\qquad \cdot p_{m_1}(\delta_1) p_{m_2}(\delta_2)
    c(\alpha_1, -\delta_1)c(\alpha_2 , -\delta_2)\\
& \qquad \cdot \Big[ \chi_{m_1 \sim m_2} 
    + \chi_{m_1 \not\sim m_2} 
\frac{\alpha_1\delta_1 + \alpha_2\delta_2}
    {\alpha_1\delta_1 + \alpha_2\delta_2 +  
(\delta_1+\delta_2)^2 - 2(\delta_1+\delta_2)(\delta_1+\delta_2 + \xi)}\Big]
\end{align*}
which we will again use an implicit normal form for.

~

\subsubsection{General corrections}

This motivates the following definition:
\renewcommand\ord[1]{2{#1}+1}
\begin{definition}\label{def:bn}
    Suppose $n \geq 2$ and 
    let $\chi = \chi(m_1, \cdots, m_n)$ be 0 if $m_k$, the largest  
    of $m_1, \ldots, m_n$, is so large that 
    $\sum_{i}\alpha_i\delta_i \sim \alpha_k \delta_{k}$
    for $\delta_i$ in the support of $p_{m_i}$ and $\alpha_i$ in
    the support of $p_{\cong h}$
    and 1 otherwise. Define
\begin{align*}
    &b^{{\ord n}}_h = \frac{(-1)^n}{n!} 
\frac1{\alpha_1\delta_1\cdots \alpha_n\delta_n}
\sum_{h_k \cong h_\ell} 
p_{\ell}(\delta_1 + \cdots + \delta_n + \xi)
        \prod_{1 \leq k \leq n}\sum_{\ell \ll m_k \ll h_k} 
        p_{h_k}(\alpha_k)p_{m_k}(\delta_k) c^{\himed}(\alpha_k, -\delta_k)\\
& \qquad \cdot \Big[ \chi 
    + (1-\chi)
\frac{\alpha_1\delta_1 + \cdots + \alpha_n\delta_n}
    {\alpha_1\delta_1 + \cdots +\alpha_n\delta_n +  
(\delta_1+\cdots +\delta_n)^2 - 2(\delta_1+\cdots+\delta_n)
(\delta_1+\cdots +\delta_n + \xi)}\Big]\\\\
    &c^{{\ord n},\nonres}_h = -\frac{(-1)^n}{n!} 
    \frac{\alpha_1\delta_1 + \cdots + \alpha_n\delta_n}{\alpha_1\delta_1\cdots \alpha_n\delta_n}
\sum_{h_j \cong h_k} p_{\ell}(\delta_1 + \cdots + \delta_n + \xi)\\
    &\qquad\qquad\qquad \prod_{1 \leq k \leq n}\sum_{\ell \ll m_k \ll h_k} 
        p_{h_k}(\alpha_k)p_{m_k}(\delta_k) c^{\himed}(\alpha_k, -\delta_k)
\end{align*}
\begin{align*}
    c^{{\ord n},\res}_h &= -\frac{(-1)^n}{(n-1)!} 
    \frac{\alpha_n\delta_n}{\alpha_1\delta_1\cdots \alpha_n\delta_n}
\sum_{h_j \cong h_k} p_{\ell}(\delta_1 + \cdots + \delta_n + \xi)
    \sum_{m_n \lesssim \ell \ll h_n} 
        p_{h_n}(\alpha_n)p_{m_n}(\delta_n) c^{\himed}(\alpha_n, -\delta_n)\\
    &\qquad \prod_{1 \leq k \leq n-1}\sum_{\ell \ll m_k \ll h_k} 
        p_{h_{k}}(\alpha_k)p_{m_k}(\delta_k) c^{\himed}(\alpha_k, -\delta_k),
\end{align*}
and 
    \[C^{\ord n}_h = C^{{\ord n},\res}_h + C^{{\ord n}, \nonres}_h\]
\end{definition}
\begin{lemma}\label{lem:bnh_kernel_size}
    The summands of the symbols $b^{\ord n}_h$, $c^{{\ord n},nres}_h$ and 
    $c^{{\ord n},res}$ 
    have integrable kernels with the following sizes:
    \[
        \| \check b^{\ord n}_h(y)\|_{L^1_{y}} \lesssim \frac{1}{\dyad^{nh}}\prod 
        \frac 1{\dyad^{m_i}}
    \]
    \[
        \| \check c^{{\ord n},\nonres}_h(y)\|_{L^1_{y}} \lesssim \frac{1}{\dyad^{(n-1)h}}\prod_{i \neq max}
        \frac 1{\dyad^{m_i}}
    \]
    \[
        \| \check c^{{\ord n},\res}_h(y)\|_{L^1_{y}} \lesssim \frac{1}{\dyad^{(n-1)h}}\prod_{i \leq n-1}
        \frac 1{\dyad^{m_i}}
    \]
\end{lemma}
\begin{proof}
    We have analyzed before the kernel associated to $(1/\eta) p_\ell(\eta)$ to 
    have size $1/\dyad^\ell$, and the projections are also smooth on their respective
    scales. The polynomial numerators are also easily seen to 
    be smooth on their respective scales since they only have finitely many 
    non-zero derivatives. All that remains is this factor:
    \[
(1-\chi) 
\frac{\alpha_1\delta_1 + \cdots + \alpha_n\delta_n}
    {\alpha_1\delta_1 + \cdots +\alpha_n\delta_n +  
(\delta_1+\cdots +\delta_n)^2 - 2(\delta_1+\cdots+\delta_n)
(\delta_1+\cdots +\delta_n + \xi)}
    \]
    where $1-\chi$ is only supported on indices where the largest $m_i$ is much 
    greater than all others. Let's suppose that this largest medium
    frequency is $m_k$. Then 
    the numerator has integrable kernel of size $\dyad^h \dyad^{m_k}$ by simply using the 
    coordinates $\alpha_i,$ $\delta_i$. 

    For the denominator, first we note that the size of the denominator 
in the support of $(1-\chi)$ is $\dyad^h\dyad^{m_k}$. 

Now we consider this
to be the composition of $f(x) = 1/x$ with 
\[
g(\alpha_1, \ldots, \delta_n, \eta) = {\alpha_1\delta_1 + \cdots +\alpha_n\delta_n +  
(\delta_1+\cdots +\delta_n)^2 - 2(\delta_1+\cdots+\delta_n)
\eta}
\]
where 
$\eta = (\delta_1+\cdots +\delta_n + \xi)$ is size $\dyad^\ell$.

Now each derivative on $f$ gains a factor of the size of its input so 
that 
\[\partial^n f \sim (\dyad^{-h}\dyad^{-m_k})^{n+1}\]

For $g$ we see that 
\[\partial_{\delta_i} g = \alpha_i + 2\sum_j \delta_j - 2\eta\]
\[\partial_{\alpha_i} g = \delta_i\]
\[\partial_{\eta} g = -2\sum_j \delta_j\]
and all second derivatives are constant.
Now suppose $\partial^\gamma f \circ g$ has size 
$\dyad^{-h}\dyad^{-m_k} \dyad^{-h\gamma_\alpha} \dyad^{-m_k\gamma_\delta}.$
Note that $\partial^\gamma f \circ g$ is a sum of functions of the form 
\[f^{n}(g)(\partial^{k_1} g)(\partial^{k_2} g) \cdots.\]
Any subsequent derivative will either land on $f(g)$ which 
produces an extra factor of $\dyad^{-h}\dyad^{-m_k}$ and loses either 
$\dyad^{-m_k}$ or $\dyad^{-h}$ depending on whether the derivative was 
in the $\alpha$ or $\delta,\eta$ direction, which is good, 
or it lands on one of the factors of $g$ which also is good since the
derivatives of $g$ are polynomials. 
Thus, we have that all derivatives gain a factor corresponding to 
their size, and we have an integrable kernel by Lemma~\ref{lem:sep}.

\end{proof}

\begin{corollary}
    $B^{\ord n}_h$ as constructed above satisfies Definition 
    \ref{def:good_corrections}.
\end{corollary}
\begin{proof}
    These bounds follow in an almost identical fashion to the bounds 
    for $B^3$ so we again omit the proof.
\end{proof}

Next, we quantify the way in which $B^{\ord n}_h$ and $C^{\ord n}_h$ are related up to good 
errors. Note, while we know that $B^{\ord k}_h(G)$ is good for any 
good term $G$, it is not immediately obvious that $G(B^{\ord k}_h)$ will be 
good. This was not a problem previously because the Neumann operator 
$L_v$ only contains factors of $B$, and so good terms $G$ never appear 
on the outside. However, here we will be interested in producing good errors 
in many applications of $B^{\ord k}_h$. For instance, 
the commutator with the Schr\"odinger operator of several applications of $B^{\ord k}_h$ 
is the sum of the commutators as applied to each factor $B$. Each commutator 
$[i\partial_t + \partial_x^2, B^{\ord k}_h]$ 
we would like to say is equal to $-C^{{\ord k},\nonres}_h$ up to a good error, but we have 
to do so in this context:
\[
    [i\partial_t + \partial_x^2, B^{\ord {k_1}}_h \circ \cdots \circ B^{\ord {k_n}}_h] 
    = \cdots + 
    B^{\ord {k_1}}_h \circ \cdots \circ [i\partial_t + \partial_x^2, B^{\ord {k_j}}_h] 
    \circ \cdots \circ B^{\ord {k_n}}_h + \cdots 
\]
This leads to the following proposition which is the higher order analog of 
Proposition \ref{prop:c2_h_approx}.
\begin{proposition}\label{prop:good_errors}
    Let 
    \[D_h = B^{\ord {j_1}}_h \circ \cdots \circ B^{\ord {j_k}}_h\]
    be some arbitrary composition of the $B^{\ord {n_h}}$ functions of varying orders.

    Then the following differences are good for all $n \geq 1$
    \[C^{\ord {(n+1)}}_h\circ D_h - B^{\ord {n}}_h \circ C^{\ord {(1)}}_h \circ D_h,\]
and 
    \[[i\partial_t + \partial_x^2, B^{\ord {n}}_h] \circ D_h + 
    C^{\ord n,\nonres}_h \circ D_h\]
\end{proposition}
\begin{proof}
    We begin by considering the symbol of the first difference 
    \[(C^{\ord {(n+1)}}_h - B^{\ord{n}}_h \circ C^{\ord {(1)}}_h) \circ D_h.\]
    We write the definition of the symbol of 
    the difference 
    unsymmetrizing $C^{\ord {(n+1)},\nonres}_h$ in $C^{\ord {(n+1)}}_h$ to be parallel
    to $B^{\ord{n}}_h\circ C^{\ord {(1)}}_h$. We will also 
    define $\xi_d = \delta_{n+1} + \cdots + \delta_{N} + \xi$ 
    to be the frequency output of the frequencies coming from $D_h$.
    Then 
\begin{align*}
    &c^{\ord {(n+1)}}_h d_h - b^{\ord {n}}_hc^{\ord {(1)}}_h d_h\\
    &= -\frac{(-1)^{n+1}}{n!} 
    \frac{\alpha_{n+1}\delta_{n+1}}{\alpha_1\delta_1\cdots \alpha_{n+1}\delta_{n+1}}
\sum_{h_{k_1} \cong h_{k_2}}p_{\ell}(\delta_1 + \cdots + \xi_d)
    \prod_{1 \leq k \leq n}\sum_{\ell \ll m_k \ll h_k}
    p_{h_k}(\alpha_k)p_{m_k}(\delta_k)c^{\himed}(\alpha_k, -\delta_k) \\
        &\hspace{12em}
    \sum_{m_n \ll h_n \gg j}p_{h_n}(\alpha_n)p_{m_n}(\delta_n)
    c^{\himed}(\alpha_n , -\delta_n)d(\delta_{n+1}, \ldots, \xi)\\
&-\frac{(-1)^{n}}{n!} 
\frac1
{\alpha_1\delta_1\cdots \alpha_{n}\delta_{n}}
\sum_{h_{k_1} \cong h_{k_2}} 
p_{\ell}(\delta_1 + \cdots +  \xi_d)
    \prod_{1 \leq k \leq n}\sum_{\ell \ll m_k \ll h_k}
    p_{h_k}(\alpha_k)p_{m_k}(\delta_k) c^{\himed}(\alpha_k, -\delta_k)\\
&\cdot \Big[ \chi 
    + (1-\chi) 
    \frac{\alpha_1\delta_1 + \cdots + \alpha_{n}\delta_{n}}
    {\alpha_1\delta_1 + \cdots +\alpha_{n}\delta_{n} +  
        (\delta_1+\cdots +\delta_{n})^2 - 2(\delta_1+\cdots+\delta_{n})
(\delta_1+\cdots + \xi_d)}\Big]\\
    &\hspace{5em}\sum_{j \ll h_{n+1}}\sum_{h_{n+1} \gg m_{n+1}} p_j(\delta_{n+1} + \xi)
    p_{h_n}(\alpha_{n+1}) p_{m_n}(\delta_{n+1}) c^{\himed}(\alpha_{n+1}, -\delta_{n+1})d(\delta_{n+2}, \ldots, \xi).
\end{align*}
To make the second term look like the first we have to make two substitutions:
\[\sum_{j \ll h_{n+1}} p_j(\delta_{n+1} + \xi_d) \mapsto 1\] 
and
\[(1-\chi)\frac{\alpha_1\delta_1 + \cdots + \alpha_{n}\delta_{n}}
    {\alpha_1\delta_1 + \cdots +\alpha_{n}\delta_{n} +  
        (\delta_1+\cdots +\delta_{n})^2 - 2(\delta_1+\cdots+\delta_{n})
(\delta_1+\cdots +\delta_{n+1} + \xi_d)} \mapsto (1-\chi)\] 
After making these substitutions, the two terms will cancel 
exactly, so we need only show that each substitution produces a good term.

The strategy for both of these estimates will 
simply be to find the largest $m_k$ which is not the 
    uncorrected frequency $m_{n+1}$ among both $D_h$ and 
    $B^{\ord {n}}_h \circ C^{\ord{(1)}}_h$,
place this pair in $L^6_{t,x}$, and pair the uncorrected $m_{n+1}$ terms 
with the last factor corresponding to $\xi$ and in $L^6_{t,x}$. We 
will show this general estimate below.

Suppose 
that the size of the integrable kernel of the 
difference of one of the replacements is $K$, after localizations to 
some set of dyadic regions (which may depend on the $m$'s, $j$ and $h$).
Then the separation of variables Lemma~\ref{lem:sep} tells us we can 
estimate the difference on this specific set of dyadic regions by
\begin{align*}
    K\left(\prod_{i \neq n+1}\frac1{\dyad^h\dyad^{m_i}}\right)
    &\sup_{x_i}\|\left(\prod_{1 \leq i \leq N} 
    P_{m_i}(P_{h_1} v(x + x_{2i-1})P_{h_1}\bar v (x-x_{2i}))\right)
 P_m v(x)\|_{L^1_{t,x}}\\
    &\lesssim \frac K{\dyad^h \dyad^{m_k}}\sup_{x_n} \|P_{h_k}v\|_{L^6_{t,x}}^2
    \|P_{h_n} v\|_{L^6_{t,x}} 
    \| P_{h_n}\bar v(x+x_n) P_m v(x)\|_{L^2_{t,x}} \prod_{i\neq n+1,k} 
    \frac{1}{\dyad^h}\|P_{h_i} v\|^2_{L^2_{t,x}}
\end{align*}
For this to be good, we must have that 
\[ K \lesssim \frac{\dyad^{m_k}}{\dyad^{h/2} \dyad^{m/2}}\]
where $m_k$ is the largest medium frequency which is not the uncorrected 
frequency $m_{n+1}$, and where $m$ is the medium frequency corresponding 
to the last frequency $\xi$. We verify this for each term.

~

First, we look at the projection. This will follow from the fact 
that if $j \gtrsim h_{n+1}$, then we have a good estimate, because  
\[\sum_{j \ll h_{n+1}} p_j(\delta_{n+1} + \xi_d) + \sum_{j \gtrsim h_{n+1}} 
p_j(\delta_{n+1} + \xi_d)
= 1\]

If $j \gtrsim h_{n+1}$, then 
we must have that the largest of the first $n$ medium frequencies 
is also comparable to $h_{n+1}$ (potentially with a worse 
constant which depends on $n$) since 
\[
    \sum_{i \leq n} \dyad^{m_i} + 2^j \sim 
    \sum_{i\leq n} \delta_i + (\delta_{n+1} + \xi_d)
    \sim 2^\ell \ll \dyad^{h_{n+1}}.
\] 
Then we have the bound
\[
    1 \lesssim \frac{\dyad^{h}}{\dyad^{h/2}\dyad^{m/2}} \sim_n
\frac{\dyad^{m_k}}{\dyad^{h/2}\dyad^{m/2}}.
\]
Again $p_j$ is smooth on scale $j$, so the bounds hold for derivatives 
allowing an application of Lemma \ref{lem:sep}.

~

Finally, for the fraction we write the difference as 
\[-(1-\chi)\frac{(\delta_1+\cdots +\delta_{n})^2 - 2(\delta_1+\cdots+\delta_{n})
(\delta_1+\cdots +\delta_{n+1} + \xi_d)}
    {\alpha_1\delta_1 + \cdots +\alpha_{n}\delta_{n} +  
        (\delta_1+\cdots +\delta_{n})^2 - 2(\delta_1+\cdots+\delta_{n})
(\delta_1+\cdots +\delta_{n+1} + \xi_d)} \] 
Note that we already analyzed the separability estimate of the 
denominator, and the numerator is a polynomial and will thus have a 
good separability estimate given the overall size, so we focus on the size here.

Note that on the support of $(1-\chi)$ the largest medium 
frequencies is much larger than the rest so that this is comparable to 
\[-(1-\chi)\frac{(\dyad^{m_k})^2 - 2(\dyad^{m_k})
(\dyad^\ell)}
{\dyad^h \dyad^{m_k} + {(\dyad^{m_k})^2 - 2(\dyad^{m_k})
(\dyad^\ell)}} \lesssim \frac{\dyad^{m_k}}{\dyad^h} \lesssim 
\frac{\dyad^{m_k}}{\dyad^{h/2}\dyad^{m/2}} 
\]

Now, we analyze the other term 
\[[i\partial_t + \partial_x^2, B^{\ord {n}}_h] \circ D_h + C^{\ord n,\nonres}_h \circ D_h.\]
Since there is nothing to do when $n=1$ we focus on $n \geq 2$.

We have already seen that the commutator with the Schr\"odinger operator 
is multiplication by $\Delta \xi^2$ on the frequency side. From 
definition \ref{def:bn}, we compute that 
$[i\partial_t + \partial_x^2, B^{\ord n}_h] + C^{\ord n, \nonres}_h$ 
cancels except for a term involving $\chi$: 

\begin{align*}
    \Delta \xi^2 b^{\ord n}_h + c^{\ord n, \nonres}_h
  &= \frac{(-1)^n}{n!} 
 \frac{1}
 {\alpha_1\delta_1\cdots \alpha_n\delta_n}
\sum_{h_j \cong h_k} 
    p_{\ell}(\delta_1 + \cdots + \xi)\\
    &\quad \prod_{1\leq k \leq n}
    \sum_{\ell \ll m_k \ll h_k} p_{h_k}(\alpha_k) 
p_{m_k}(\delta_k)
    c^{\himed}(\alpha_k, -\delta_k)\\
& \quad \cdot \Big[ {(\delta_1+\cdots +\delta_n)^2 
        - 2(\delta_1+\cdots+\delta_n)
(\delta_1+\cdots +\delta_n + \xi)} \Big]\chi.
\end{align*}
Thus we just need to show that this term 
is good even when $\xi$ is replaced by the sum of frequencies 
coming from $D_h$. Here $D_h$ does not matter since 
$\xi$ only appears in the term $\sum \delta + \xi$ which 
is projected to the lowest frequency $\dyad^\ell$.

Note, everything will have integrable kernel 
since the above is a sum of products of terms which are
either a projection, $1/(\alpha\delta)$ (which we have analyzed already), 
or a polynomial. Thus, it suffices to measure the size.

Note, in the support of $\chi$ the two largest medium frequencies are comparable. 
Without loss of generality assume that the largest medium frequency 
are $\delta_1$ and $\delta_2$. Then we may estimate 
\[
\Big[ {(\delta_1+\cdots +\delta_n)^2 
        - 2(\delta_1+\cdots+\delta_n)
(\delta_1+\cdots +\delta_n + \xi)} \Big] \lesssim \dyad^{2m_1}
\]
Based on this, we arrive at the good estimate, Definition 
\ref{def:good_L6_terms} by pairing one of the $m_1$ 
high frequency $v$'s with the last medium frequency $v$, and to put 
the remaining $m_1$ and the two $m_2$ frequency $v$'s in $L^6_{t,x}$. This 
suffices since the high frequency gain from the bilinear pairing controls 
the last frequency, and the $\dyad^{m_1}$ is controlled by $\dyad^{-m_1}\dyad^{-m_2}$
coming from the denominator. Finally, there are enough high frequencies 
in the denominator: usually the commutator 
with the Schr\"odinger operator produces a term of size $\dyad^h\dyad^m$, but here 
we have only produced $\dyad^{2m}$.
\end{proof}

\newcommand\ordbal[1]{2#1 + 3}

Finally, before we state the updated normal form, we must introduce 
corrections parallel to $B^{5,\bal}$ but for $B^{\ord{n}}_{h}$. Luckily, these 
corrections follow in the exact same way as in Section \ref{ss:quintic_corrections}.

In particular, there are $n$ pairs of high frequencies in $B^{\ord n}_h$ each 
treated symmetrically, so without loss of generality we may think of  
$C^{\ordbal n, \bal}_h$ as $C^{\bal}$ landing in the first pair
of $B^{\ord n}_h$. As in
Section \ref{ss:quintic_corrections}, replacing $\bar v$ with $\ol{C^{\bal}}$ 
will come with a negative sign which will allow us to apply Lemma 
\ref{l:division}. 
In particular 
the relevant form is 
\begin{align*}
    C^{\ordbal n, \bal}_h &= B^{\ord n}_h(C^{\bal}(v, \bar v, v), \bar v, v, \cdots, v)
    - B^{\ord n}_h(v, \ol{C^{\bal}(v, \bar v, v)}, v, \cdots, v)\\
    &+ \cdots  
    + B^{\ord n}_h(v, \cdots, C^{\bal}(v, \bar v, v), \bar v, v)
    - B^{\ord n}_h(v, \cdots, v \ol{C^{\bal}(v, \bar v, v)}, v).
\end{align*}

The important fact is the following.
\begin{proposition}\label{prop:bnbal}
    There is a decomposition 
    \[
        c^{\ordbal{n},\bal}_h + \Delta^{2n+4}\xi^2 b^{\ordbal{n},\bal}_h = 
        g^{\ordbal{n}, \bal}_h
    \]
    where the corresponding multilinear form $G^{\ordbal{n},\bal}_h$ 
    to $g^{\ordbal{n},\bal}$ 
    satisfies Definition \ref{def:good_L6_terms}, and in addition 
    \[
        G^{\ordbal{n}, \bal}_h \circ B^{\ord {j_1}}_h \circ \cdots \circ 
        B^{\ord {j_k}}_h
    \]
    satisfies Definition \ref{def:good_L6_terms},
    and 
    $b^{\ordbal{n},\bal}_h$ has 
    the same support as $c^{\ordbal{n},\bal}_h$ with the estimate 
    \[
        \|\check b^{\ordbal{n},\bal}_h\|_{L^1}
        \lesssim \dyad^{(-2 -n)h}\prod_i \dyad^{-m_i}.
    \]
    after localization to some set of medium frequency pairs $m_i$.
\end{proposition}
\begin{proof}
    First we argue that it will suffice to find a decomposition 
    \[
        c^{\ordbal{n},\bal}_h + \Delta^{4}\xi^2 b^{\ordbal{n},\bal}_h = 
        g^{\ordbal{n}, \bal}_h
    \]
    where $\Delta^4\xi^2$ is the symbol associated to the four frequencies where 
    $C^{\bal}$ is injected. The largest term of the difference between the desired
    symbol and $\Delta^4\xi^2$ is controlled by 
    \[
        |\Delta^{2n+4}\xi^2 - \Delta^4\xi^2| \lesssim \dyad^h\dyad^{m_k}
    \]
    where $m_k$ denotes the largest medium frequency pairing.
    Now, if we find a decomposition with $\Delta^4\xi^2$ then 
    $b^{\ordbal{n},\bal}_h$ will have symbol size 
    $\dyad^{(-2-n)h} \prod_i \dyad^{-m_i}$. So the difference 
    between the $\Delta^{2n+4}\xi^2$ decomposition and $\Delta^4\xi^2$ decomposition
    has size 
    \[
        \dyad^{(-1-n)h} \prod_{i\neq m_k} \dyad^{-m_i}
    \]
    after localization to some set of medium frequency pairs. We need to show 
    that this difference is good according to Definition 
    \ref{def:good_L6_terms}. The strategy is to pair one of the $m_k$ pair 
    high frequencies 
    with the medium frequency, place the other $m_k$ pair high frequency in 
    $L^6_{t,x}$, and place the two extra high frequencies from $C^{\bal}$ in 
    $L^6_{t,x}$. If $k=1$, when $C^{\bal}$ lands on the highest pair, 
    then this is completely straightforward:
    \begin{align*}
        &\lesssim 
        \sum_{h \gg m_k \gg \ell} \dyad^{-2h}
        \|P_{m_k}(P_{\cong h} v P_{\cong h} \bar v P_{\cong h} v
        P_{\cong h} \bar v) P_{\lesssim m_k} v\|_{L^1_{t,x}}
        \prod_{i \neq k} \sum_{h \gg m_i \gg \ell} \dyad^{-h}\dyad^{-m_i}
        \|P_{m_i}(P_{\cong h} v P_{\cong h} \bar v)\|_{L^\infty_{t,x}}\\
        &\lesssim 
        \sum_{h \gg m \gg  \ell} \dyad^{-2h}
        \|P_{\cong h} v\|_{L^6_{t,x}}^3 \|P_{\cong h} \bar v P_m v\|_{L^2_{t,x}}
        (C^2 \epsilon^2 h \dyad^{(-1 -2s)h})^{n-1}\\
        &\lesssim C^{2n} \epsilon^{2n + 2} 
        \sum_{h \gg m \gg  \ell} \dyad^{(-2 -3s)h} \dyad^{-sm} c_h^3c_m\\
        &\lesssim C^{2n} \epsilon^{2n + 2} c_\ell
    \end{align*}
    Otherwise, it is important that we apply Bernstein's inequality on 
    the $m_1$ pair before splitting it into two $L^6_{t,x}$ estimates 
    and two $L^\infty_tL^2_x$ estimates.
    \begin{align*}
        &\lesssim 
        \sum_{h \gg m_k \gg \ell} \dyad^{-3h}\dyad^{-m_1}
        \|P_{m_1}(P_{\cong h} v P_{\cong h} \bar v P_{\cong h} v
        P_{\cong h} \bar v)\|_{L^{3}_{t}L^\infty_x} 
        \|P_{m_k}(P_{\cong h} v P_{\cong h} \bar v) P_{\lesssim m_k} v\|_{L^{3/2}_t L^1_x}
        (C^2 \epsilon^2 h \dyad^{(-1 -2s)h})^{n-2}\\
        &\lesssim C^{2n-4}\epsilon^{2n-4}
        \sum_{h \gg m \gg  \ell} \dyad^{-3h}
        \|P_{\cong h} v\|_{L^6_{t,x}}^3 \|P_{\cong h} v\|_{L^\infty_tL^2_x}^2
        \|P_{\cong h} \bar v P_m v\|_{L^2_{t,x}} \\
        &\lesssim C^{2n} \epsilon^{2n + 2} 
        \sum_{h \gg m \gg  \ell} \dyad^{(-2 -3s)h} \dyad^{-sm} c_h^3c_m\\
        &\lesssim C^{2n} \epsilon^{2n + 2} c_\ell
    \end{align*}
    
    Now, we wish to proceed with a decomposition using $\Delta^4\xi^2$, which 
    will follow directly from Lemma \ref{l:division} using 
    the same argument as Proposition \ref{prop:b5bal} if we can 
    replace the symbol 
    \[ (1-\chi)\frac{\alpha_1\delta_1 + \cdots + \alpha_n\delta_n}
    {\alpha_1\delta_1 + \cdots +\alpha_n\delta_n +  
(\delta_1+\cdots +\delta_n)^2 - 2(\delta_1+\cdots+\delta_n)
(\delta_1+\cdots +\delta_n + \xi)}
\]
by $1-\chi$ up to a good term. 

The difference of these symbols is 
\[
    (1-\chi)\frac{(\delta_1+\cdots +\delta_n)^2 - 2(\delta_1+\cdots+\delta_n)
    (\delta_1+\cdots +\delta_n + \xi)}
    {\alpha_1\delta_1 + \cdots +\alpha_n\delta_n +  
(\delta_1+\cdots +\delta_n)^2 - 2(\delta_1+\cdots+\delta_n)
(\delta_1+\cdots +\delta_n + \xi)}.
\]
    Recall that in the support of $(1-\chi)$ the largest $m_i$, which 
    we will call $m_k$ is so much larger than the rest that 
\[
    (1-\chi)\frac{(\delta_1+\cdots +\delta_n)^2 - 2(\delta_1+\cdots+\delta_n)
    (\delta_1+\cdots +\delta_n + \xi)}
    {\alpha_1\delta_1 + \cdots +\alpha_n\delta_n +  
(\delta_1+\cdots +\delta_n)^2 - 2(\delta_1+\cdots+\delta_n)
    (\delta_1+\cdots +\delta_n + \xi)} \sim \dyad^{m_k}\dyad^{-h}.
\]
Thus, the symbol size of the difference is again
    $\dyad^{(-1-n)h} \prod_{i \neq k} \dyad^{-m_i}$
    and so the above strategy applies to show that this difference is also good.

    Now that this difference is controlled, the remaining term is separable 
    in $\delta_i$ and $\alpha_i$ (except for the overall projection 
    to the lowest frequency) and we may directly apply Lemma \ref{l:division} 
    as in Proposition \ref{prop:b5bal}.

    The extra good estimate for 
    \[
        G^{\ordbal{n}, \bal}_h \circ B^{\ord {j_1}}_h \circ \cdots \circ 
        B^{\ord {j_k}}_h
    \]
    follows exactly from the estimate above, except that there are more 
    pairs of medium frequencies which are controlled by Bernstein's inequality.
\end{proof}
As usual we have the following corollary of the symbol size estimates whose proof 
we omit because of its similarity to the corresponding proof in Section 
\ref{ss:cubic_corrections}.
\begin{corollary}
    $B^{\ordbal{n},\bal}_h$ satisfies the required definition 
    for corrections, Definition \ref{def:good_corrections}.
\end{corollary}

Now,
Proposition \ref{prop:bnbal} tells us that
\[
    [i\partial_t + \partial_x^2, B^{\ordbal{n},\bal}_h]
    = - C^{\ordbal{n},\bal}_h + G^{\ordbal{n},\bal}_h
\]
and so these symbols suitably correct $C^{\ordbal{n},\bal}_h$.

To unify the notation we will refer to $B^{5,\bal}$ by $B^{\ordbal{(1)}, \bal}_h$ 
going forward.  

\subsubsection{General implicit normal form}

Now we are in a position to show that given the corrections constructed 
above up to some order $N$, that all terms of $L^k_v (\cdot)$ for $k \leq N$ 
are a certain controllable term plus a good remainder.

Inspired by our corrections to the normal form in \eqref{eq:normal_form_correction}
we use the normal form variable 
\begin{align}\label{eq:normal_form}
    v &= u + B^{3,\hihi}(v,\bar v, v)\\ 
      \nonumber &+ \sum_h B^{\ord {(1)}}_h + B^{\ordbal{(1)},\bal}_h\\
      \nonumber &+ \sum_h B^{\ord {(2)}}_h 
      - B^{\ord {(1)}}_h \circ B^{\ord{(1)}}_h + B^{\ordbal{(2)},\bal}_h - 
      B^{\ordbal{(1)},\bal}_h \circ B^{\ord {(1)}}_h 
      - B^{\ord {(1)}}_h \circ B^{\ordbal{(1)},\bal}_h\\
      \nonumber &+ \cdots \\
      \nonumber &+ \sum_h \sum_{j_1+\cdots+j_k=N} (-1)^{k+1} B^{\ord {j_1}}_h \circ 
        \cdots \circ B^{\ord {j_k}}_h + \sum_{1 \leq i \leq k} 
        (-1)^{k+1} B^{\ord {j_1}}_h \circ 
        \cdots \circ B^{\ordbal{j_i },\bal}_h \circ \cdots \circ B^{\ord {j_k}}_h
\end{align}
Note that we sum over all ordered partitions of $N$, including the single 
term partition. In fact the single term partition captures the 
most basic correction $B^{\ord n}_h$.

We let $L_v$ as usual and write the NLS equation for $v$.

\begin{align}\label{eq:new_nls}
    (i\partial_t &+ \partial_x^2) v
    = C^\bal + C^{\low} + \sum_hC^{\low}(P_{\cong h}v, \ol{
    \sum_{j_1 + \cdots + j_k \leq N} (-1)^{k} B^{\ord {j_1}}_h \circ \cdots \circ 
    B^{\ord {j_k}}_h }, P_{\cong h}v) \\
\nonumber &+ 
    \sum_{k=0}^\infty (L_v)^k\Big(
    G + \sum_h C^{\ord {(1)},\res}_h 
+ \sum_{n=2}^N \left[i\partial_t + \partial_x^2,  
          \sum_{j_1+\cdots+j_k=n} (-1)^{k+1} B^{\ord {j_1}}_h \circ \cdots \circ
      B^{\ord {j_k}}_h\right]\\
    \nonumber           & -\sum_{n=1}^N \sum_{j_1+\cdots+j_k=n} 
    (-1)^{k+1}C^{\ord {(1)}}_h\circ B^{\ord {j_1}}_h \circ \cdots \circ B^{\ord {j_k}}_h
\Big)
\end{align}
where here $G$ collects all good terms and is given by 
\begin{equation}\label{eq:post_b5_good}
\begin{aligned}
    G &= L_v(C^{\low} + 
    \sum_hC^{\low}(P_{\cong h}v, \ol{
    \sum_{j_1 + \cdots + j_k \leq N} (-1)^{k} B^{\ord {j_1}}_h \circ \cdots \circ 
    B^{\ord {j_k}}_h }, P_{\cong h}v)) \\
    & + L_v(C^{\bal})- \sum_h \sum_{j_1+\cdots+j_k\leq N} \sum_{1 \leq i \leq k} 
        (-1)^{k+1} B^{\ord {j_1}}_h \circ 
        \cdots \circ (C^{\ordbal{j_i},\bal}_h - G^{\ordbal{j_i },\bal}_h) 
        \circ \cdots \circ B^{\ord {j_k}}_h\\
    &+ C(u, \bar u, u) - C(v, \bar v, v) - 
    \sum_hC^{\low}(P_{\cong h}v, \ol{
    \sum_{j_1 + \cdots + j_k \leq N} (-1)^{k} B^{\ord {j_1}}_h \circ \cdots \circ 
    B^{\ord {j_k}}_h }, P_{\cong h}v) \\
&\qquad +\sum_{j_1+\cdots+j_k\leq N} (-1)^{k+1}C^{\ord{(1)}}_h\circ 
            B^{\ord {j_1}}_h \circ \cdots \circ B^{\ord {j_k}}_h.
\end{aligned}
\end{equation}
Here the first line captures the good terms which come from 
the perturbative part of $C$ interacting with each $B$ term. The second 
line captures $C^\bal$ interacting with each $B$ term, where 
we have subtracted off the correction coming from the $B^{\ordbal{n},\bal}_h$ 
terms and are left with a good part only. 
And the third and fourth lines capture all terms coming from the implicit 
normal form except the bad ones explicitly contained in \eqref{eq:new_nls} again 
leaving us with a good term.

In addition to $G$, we will want to know that the only way $L_v^k(C^{\ord n,\res}_h)$ 
can form a bad term, is if $L_v$ provides only unbalanced corrections 
$B^{\ord {j_1}}_h \circ \cdots \circ B^{\ord {j_k}}_h$, 
places $C^{\ord n,\res}$ on the medium frequency,
and matches the high frequencies. And finally, we have more outlier terms as in 
Proposition \ref{prop:b3_good_bad_splitting}, which we want to show are good 
after an application of $L_v$.

The following proposition proves these claims.
\begin{proposition}\label{prop:bn_good_bad_splitting} 
    The term $G$ defined above satisfies Definition \ref{def:good_L6_terms}. 

    In addition 
    \begin{align*}
        L_v(C^{\ord n,\res}_h &\circ B^{\ord {i_1}}_h \circ \cdots \circ B^{\ord {i_{k'}}}_h) \\
        &- 
        \sum_{1\leq N} \sum_{j_1 + \cdots + j_k \leq N} 
        (-1)^{k+1} B^{\ord {j_1}}_h \circ \cdots \circ B^{\ord {j_k}}_h \circ 
        C^{\ord n,\res}_h \circ B^{\ord {i_1}}_h \circ \cdots \circ B^{\ord {i_{k'}}}_h
    \end{align*}
    satisfies the good estimate of Definition \ref{def:good_L6_terms}.

    Finally, the outlier term 
    \[
\sum_hC^{\low}(P_{\cong h}v, \ol{
    \sum_{j_1 + \cdots + j_k \leq N} (-1)^{k} B^{\ord {j_1}}_h \circ \cdots \circ 
    B^{\ord {j_k}}_h }, P_{\cong h}v) 
    \]
    satisfies the source term estimate \ref{eq:source_term} and is good after 
    an application of $L_v$:
    \[
L_v\left(\sum_hC^{\low}(P_{\cong h}v, \ol{
    \sum_{j_1 + \cdots + j_k \leq N} (-1)^{k} B^{\ord {j_1}}_h \circ \cdots \circ 
    B^{\ord {j_k}}_h }, P_{\cong h}v) \right)
    \]
    satisfies Definition \ref{def:good_L6_terms}.
\end{proposition}
\begin{proof}
    This uses the same ideas as in Propositions \ref{prop:b3_good_bad_splitting}
    and \ref{prop:b5_good_bad_splitting} with the new corrections 
    $B^{\ord n}_h$ and $B^{\ordbal n,\bal}_h$. 
    Ultimately, the proofs are identical because 
    $B^{\ord n}_h$ and $B^{\ordbal n,\bal}_h$ have the same structure as 
    $B^{3,\himed}$ and $B^{5,\bal}$ up to having many medium frequency pairs which are easily controlled in $L^\infty_{t,x}$ by 
    Bernstein's and the symbol sizes from Lemma \ref{lem:bnh_kernel_size} and 
    Proposition \ref{prop:bnbal}. 

    The important ideas are as follows, we can have at most two higher frequencies 
    than have corresponding correction symbols in terms 
    such as in 
    $B(C^{\res}, \bar v, \cdots)$. Here we may use two bilinear 
    estimates to control these two new high frequencies, placing many pairs 
    of medium frequencies in $L^\infty_{t,x}$ as needed.
    
    Similarly, in terms with two medium frequencies instead of the usual one, 
    we can get control by breaking up a pair of high frequencies 
    and using two bilinear estimates against the medium frequencies. This 
    happens for instance when $C^{\bal}$, $C^{\low}$, or 
    $C^{\res}$ (with unmatched highest frequencies) lands on the medium frequency 
    of a $B$ term.

    In terms with $B^{\ordbal n,\bal}_h$ and $C^{\bal}$, we have enough extra 
    gain of high frequencies to control the extra high frequencies as in 
    the proof of Proposition \ref{prop:bnbal}. Importantly, we cannot 
    control $C^{\bal}$ landing in $B^{\ord n}_h$, 
    but this never happens because of the 
    cancellation coming from $B^{\ordbal n,\bal}_h$.

    As in Proposition \ref{prop:b3_good_bad_splitting}, the outlier terms 
    \[
\sum_hC^{\low}(P_{\cong h}v, \ol{
    \sum_{j_1 + \cdots + j_k \leq N} (-1)^{k} B^{\ord {j_1}}_h \circ \cdots \circ 
    B^{\ord {j_k}}_h }, P_{\cong h}v) 
    \]
    are bad because there are two new frequencies near $h$ which both have the same 
    sign in the output frequency. This causes a situation where there is no
    separation between the frequencies near $h$, but when this happens, the output 
    frequency is adjacent to $2\dyad^h$, which gives a natural separation 
    after measurement as in the source term estimate \eqref{eq:source_term}. This 
    separation, along with the separation against the medium frequency, gives 
    two high frequency separations so we may estimate these with two 
    bilinear estimates.

    Further, just as in Proposition \ref{prop:b3_good_bad_splitting}, after any 
    application of $L_v$ we will have lots of separations because the output 
    frequency of the outlier terms is higher than any component frequency.

    We omit the specific details for concision.
\end{proof}


At this point we collect the terms of order $2n+1$ in the sum and cancel 
using Proposition \ref{prop:good_errors}.

\begin{proposition}
    Let $F^{\ord n}$ be the sum of all of the order $\ord n$ 
    terms in \eqref{eq:new_nls} in
    the sum not in $G$ for $1 \leq n \leq N$. Note this includes 
    terms from $L_v^k F^{\ord \ell}$. Then 
    \[
        F^{\ord n} = G^{\ord n} + 
        \sum_{j_1+\cdots+j_k = n} (-1)^{k+1}C^{\ord {j_1},\res}_h
        B^{\ord {j_2}}_h\cdots B^{\ord {j_k}}_h
    \]
    where $G^{\ord n}$ has a good estimate. Note here that if there is only 
    one term in the partition, then it must be $j_1$ so that there 
    is always a $C^{\ord {j_1},\res}_h$ on the outside.
\end{proposition}
\begin{proof}
    We argue by induction. The base case $n=1$, is trivially true.

    Now let us assume the formula for all $m \leq n-1$. We see 
    from \eqref{eq:new_nls} that 
    \begin{align*}
        F^{\ord n} = &\pi_n(L_v (\sum_{1 \leq k \leq n} F^{\ord k}))
        +  \left[i\partial_t + \partial_x^2,  
          \sum_{j_1+\cdots+j_k=n} (-1)^{k+1} B^{\ord {j_1}}_h \circ \cdots \circ
        B^{\ord {j_k}}_h\right]\\
        \nonumber           & \qquad -C^{\ord {(1)}}_h \circ \left(
          \sum_{j_1+\cdots+j_k=n-1} (-1)^{k+1} B^{\ord {j_1}}_h \circ \cdots \circ
        B^{\ord {j_k}}_h\right)
    \end{align*}
    where here by $\pi_n$, we mean take only the order $2n+1$ terms.

    First we focus on the commutator term and inductive term.
    We may compute using Proposition \ref{prop:good_errors} 
    and the chain rule that 
    \begin{align*}
        &\left[i\partial_t + \partial_x^2, 
          \sum_{j_1+\cdots+j_k=n} (-1)^{k+1} B^{\ord {j_1}}_h \circ \cdots \circ
      B^{\ord {j_k}}_h\right] = G \\
        &\qquad -\sum_{j_1+\cdots+j_k=n} \sum_{1 \leq i \leq k} (-1)^{k+1} B^{\ord {j_1}}_h \circ 
      \cdots \circ C^{\ord {j_i},\nonres}_h \circ \cdots \circ 
      B^{\ord {j_k}}_h 
    \end{align*}
    
    where $G$ is good.

    Now we turn to $\pi_n L_v F^m$.
    We compute that 
    \begin{align*}
        &\pi_n L_vF^{\ord m} = 
        G \\&+ \sum_h \sum_{\tilde j_1 + \cdots + \tilde j_{\tilde k} = 
        n-m} (-1)^{\tilde k+1}B^{\ord {\tilde j_1}}_h\circ \cdots \circ 
        B^{\ord {\tilde j_{\tilde k}}}_h \circ 
        \sum_{j_1+\cdots+j_k = m} (-1)^{k+1}C^{\ord {j_1}, \res}_h
        \circ B^{\ord {j_2}}_h \circ \cdots \circ B^{\ord {j_k}}_h.
    \end{align*}
    
    Here $G$ is good because if $L_v$ applies $F^{\ord m}$ into any 
    position but the last of a $B^{\ord {j_1}}_h$, then we are left with a good 
    term by Proposition~\ref{prop:bn_good_bad_splitting}.

    This may be simplified into 
    \[
        \sum_{1 \leq m \leq n-1} \pi_n L_vF^{\ord m} = G - 
        \sum_h \sum_{j_1 + \cdots + j_{k} = 
        n} \sum_{1 < i \leq k}
        (-1)^{k+1}B^{\ord {j_1}}_h\circ \cdots \circ 
        C^{\ord {j_i},\res}_h \circ \cdots \circ B^{\ord {j_k}}_h
    \]
    Note here $i > 1$ since $L_v$ always applies at least one $B$ to 
    $F$. 

    We use the linearity of these multilinear forms and the 
    fact that $C^{\ord n, \res}_h + C^{\ord n, \nonres}_h = C^{\ord n}_h$ to 
    combine the commutator terms and these inductive terms to get 
    \begin{align*}
        \pi_n(L_v &(\sum_{1 \leq k \leq n} F^{\ord k}))
        +  \left[i\partial_t + \partial_x^2,  
          \sum_{j_1+\cdots+j_k=n} (-1)^{k+1} B^{\ord {j_1}}_h \circ \cdots \circ
        B^{\ord {j_k}}_h\right]\\
        = &-\sum_h \sum_{j_1 + \cdots + j_{k} = 
        n} \sum_{1 < i \leq k}
        (-1)^{k+1}B^{\ord {j_1}}_h\circ \cdots \circ
        C^{\ord {j_i}}_h \circ \cdots \circ B^{\ord {j_k}}_h\\
        &-    \sum_h \sum_{j_1 + \cdots + j_{k} = 
        n} 
        (-1)^{k+1}C^{\ord {j_1},\nonres}_h \circ B^{\ord {j_2}}_h \circ \cdots 
        \circ B^{\ord {j_k}}_h
    \end{align*}

Now, when $j_i = 1$ since $i > 1$ we may use Proposition \ref{prop:good_errors}
to write $B^{\ord {j_{i-1}}}_h \circ C^{\ord {(1)}}_h = 
C^{\ord {(j_{i-1}+1)}}$ up to a good error. Note 
that $j_{i-1} + 1 \geq 2$. Thus the terms where $j_i = 1$ become 
\[
+\sum_h \sum_{j_1 + \cdots + j_{k} = 
        n} \sum_{1 \leq i \leq k, j_i \geq 2}
        (-1)^{k+1}B^{\ord {j_1}}_h\circ \cdots \circ 
        C^{\ord {j_i}}_h\circ \cdots \circ B^{\ord {j_k}}_h\\
\]
after relabeling. Thus, we may write the 
sum of the inductive and commutator terms up to a good error as 
\begin{align*}
    \pi_n(L_v &(\sum_{1 \leq k \leq n} F^{\ord k}))
        +  \left[i\partial_t + \partial_x^2,  
          \sum_{j_1+\cdots+j_k=n} (-1)^{k+1} B^{\ord {j_1}}_h \circ \cdots \circ
        B^{\ord {j_k}}_h\right]\\
        &=\sum_h \sum_{j_1 + \cdots + j_{k} = 
        n} \sum_{1 \leq i \leq k, j_i \geq 2}
        (-1)^{k+1}B^{\ord {j_1}}_h \circ \cdots \circ
        C^{\ord {j_i}}_h \circ \cdots \circ B^{\ord {j_k}}_h\\
&-\sum_h \sum_{j_1 + \cdots + j_{k} = 
        n} \sum_{1 < i \leq k, j_i \geq 2}
        (-1)^{k+1}B^{\ord {j_1}}_h \circ \cdots \circ
        C^{\ord {j_i}}_h \circ \cdots \circ B^{\ord {j_k}}_h\\
&-    \sum_h \sum_{j_1 + \cdots + j_{k} = 
        n} 
        (-1)^{k+1}C^{\ord {j_1},\nonres}_h\circ B^{\ord {j_2}}_h \circ \cdots 
        \circ B^{\ord {j_k}}_h
\end{align*}
Now it is clear that the first and second 
terms cancel except when $i=1$. (The first line has $j_i \geq 2$ since 
it is formed from $B^{\ord k} \circ C^{\ord {(1)}}$, 
whereas the second has $j_i \geq 2$ since 
it is what is left after removing all $C^{\ord {(1)}}$ terms.)

Now we add in the third term from our formula for $F^{\ord n}$ to see 
up to a good error that 
\begin{align*}
    F^{\ord n} &=\sum_h \sum_{j_1 + \cdots + j_{k} = n, j_1 \geq 2} 
        (-1)^{k+1}C^{\ord {j_1}}_h\circ B^{\ord {j_2}}_h \circ 
        \cdots \circ B^{\ord {j_k}}_h\\
&-    \sum_h \sum_{j_1 + \cdots + j_{k} = 
        n} 
        (-1)^{k+1}C^{\ord {j_1},\nonres}_h \circ B^{\ord {j_2}}_h \circ \cdots 
        \circ B^{\ord {j_k}}_h\\
    &-C^{\ord {(1)}}_h \circ \left(
          \sum_{j_1+\cdots+j_k=n-1} (-1)^{k+1} B^{\ord {j_1}}_h \circ \cdots \circ
      B^{\ord {j_k}}_h\right)
\end{align*}
We can combine the first and third terms since bringing the $C^{\ord {(1)}}_h$ into 
the partition of $n$ increases $k$ by one, flipping the sign. This
unrestricts the assumption that $j_1 \geq 2$:
\begin{align*}
    F^{\ord n} &=\sum_h \sum_{j_1 + \cdots + j_{k} = n} 
        (-1)^{k+1}C^{\ord {j_1}}_h \circ B^{\ord {j_2}}_h\circ
        \cdots \circ B^{\ord {j_k}}_h\\
&-    \sum_h \sum_{j_1 + \cdots + j_{k} = 
        n} 
        (-1)^{k+1}C^{\ord {j_1},\nonres}_h \circ B^{\ord {j_2}}_h  \circ
        \cdots \circ B^{\ord {j_k}}_h
\end{align*}
again up to a good error.
Finally the sum of these is by definition 
\begin{align*}
    F^{\ord n} &=G^{\ord n} + \sum_h \sum_{j_1 + \cdots + j_{k} = n} 
        (-1)^{k+1}C^{\ord {j_1},\res}_h \circ B^{\ord {j_2}}_h\circ 
        \cdots \circ B^{\ord {j_k}}_h
\end{align*}
as desired.
\end{proof}
This allows us to write 
\begin{align}
    (i\partial_t &+ \partial_x^2) v
    = C^\bal + C^{\low} + 
    \sum_hC^{\low}(P_{\cong h}v, \ol{
    \sum_{j_1 + \cdots + j_k \leq N} (-1)^{k} B^{\ord {j_1}}_h \circ \cdots \circ 
    B^{\ord {j_k}}_h }, P_{\cong h}v) \\
    \nonumber & + \sum_h \sum_{j_1+\cdots+j_k \leq N} (-1)^{k+1}
    C^{\ord {j_1},\res}_h\circ B^{\ord {j_2}}_h \circ \cdots \circ B^{\ord {j_k}}_h + 
    \sum_{k=0}^\infty (L_v)^k\Big(G + C^{\geq N}\Big)
\end{align}
where $G$ is the sum of all the good terms accumulated so far and 
$C^{\geq N}$ is the sum of terms $L_v^1(C^{\ord {j_1},\res}_h
\circ B^{\ord {j_2}}_h \circ \cdots \circ B^{\ord {j_k}}_h)$
such that the total order is greater than or equal to $N$. We use the following 
notation to simplify the situation. We use $F^{\ord n}$ as above to 
notate the bad terms of order ${\ord n}$, both when $n \leq N$ and $n > N$; when 
$n \leq N$ they have a particularly nice form, and when $n > N$ they 
will be of the form $L^k_v F^{\ord m}$ where $k+m = n$ and do not have 
the good estimate of Definition \ref{def:good_L6_terms}.

We clarify this in the following definition:
\begin{definition}\label{def:good_bad_collection}
    We let 
    \[F^3 = C^{\res}
    \]
    and 
    \[F^{\low} = C^{\low} 
        + \sum_hC^{\low}(P_{\cong h}v, \ol{
    \sum_{j_1 + \cdots + j_k \leq N} (-1)^{k} B^{\ord {j_1}}_h \circ \cdots \circ 
    B^{\ord {j_k}}_h }, P_{\cong h}v) 
    \]
    For $2 \leq n \leq N$ we let 
    \begin{align*}
        F^{\ord n} &= 
        \sum_h\sum_{j_1+\cdots+j_k = n} (-1)^{k+1}C^{\ord {j_1},\res}_h \circ
        B^{\ord {j_2}}_h \circ \cdots \circ B^{\ord {j_k}}_h
    \end{align*}
    and for $N+1 \leq n$ we let
    \begin{align*}
        F^{\ord n} &= \sum_h
        \sum_{j_1+\cdots+j_k = n} (-1)^{k+1}B^{\ord {j_1}}_h \circ \cdots \circ 
        C^{\ord {j_i}}_h \circ 
        \cdots \circ B^{\ord {j_k}}_h\\
        &+ \sum_{\substack{j_1+\cdots+j_i = n-N \\ j_{i+1}+\cdots+j_k = N}}
        (-1)^k B^{\ord {j_1}}_h \circ \cdots \circ B^{\ord {j_i}}_h \circ C^{\ord{(1)}}_h \circ \cdots \circ B^{\ord {j_k}}_h
    \end{align*}
    
    where in the above the sum is taken over all ordered partitions of $N$ 
    such that each $j \leq N$.

    Then, we let $G^{\ord n}$ be the remaining order $2n+1$ terms 
    in \eqref{eq:new_nls} which we have argued have the good estimate. 
\end{definition}

Then \eqref{eq:new_nls}
becomes
\begin{align}\label{eq:new_nls_simp+}
    (i\partial_t + \partial_x^2) v
    &= C^\bal(v,\bar v, v) + F^{\low} + \sum_{n=1}^\infty F^{\ord n}(v) + G^{\ord n}(v)\\
    \nonumber &:= C^\bal(v,\bar v, v) + N^{\unbal}(v)
\end{align}

Now we are finally in a position to complete the proof of Theorem
\ref{thm:unbal_correction}. We have already seen that the $B$ terms 
we have constructed satisfy Proposition \ref{prop:contraction} which 
proves the first part of the theorem.  
We see from Definition \ref{def:good_L6_terms} that 
$\sum_n G^{\ord n}$ will satisfy the source term estimate 
\eqref{eq:source_term} as long as the constant in $G^{\ord n}$ does not grow faster 
than $\tilde C^n$. This is the case because we only produced finitely many,
depending on $N$, base 
good terms in the argument above and all other good terms are $L_v$ applied 
to a good term. Now $L_v$ has finitely many applications depending on $N$ so
the growth rate of the constants is at most $N^n$. 

All that remains is $\sum_n F^{\ord n}$. Note 
that the estimates for $n > N$ will require 
\[ 
s > \frac{-N}{2N + 1}
\]
and $N$ can be chosen large enough to capture any $s > -1/2$.

\begin{proposition}\label{prop:unbal_L1}
    Let $F^{\ord n}(v)$ be as defined above and suppose $v$ satisfies the 
    bootstrap estimates \eqref{eq:uj_ee_boot}-\eqref{eq:uj_sep_bi_boot}. 
    Fix $N$ 
    so that 
    \[ 
        s > \frac{-N}{2N + 1}.
    \]
    Then,
    there exists a constant 
    $\tilde C$ independent of $n$ (but depending on $N$) such that 
    \begin{align*}
        \|P_{\ell} F^{\ord n}(v) P_{\cong \ell} \bar v\|_{L^1_{t,x}} \lesssim 
        (\tilde C \epsilon)^{2n+2} c_\ell^2 \dyad^{-2s\ell}
    \end{align*}
\end{proposition}
\begin{remark}
    It is clear that if $\epsilon$ is made small depending on $\tilde C$ 
    then $\sum_n F^{\ord n}$ will satisfy the desired bound \eqref{eq:source_term}.
    
    Since we have seen in Propositions \ref{prop:b3_good_bad_splitting} and 
    \ref{prop:bn_good_bad_splitting} that $F^{\low}$ satisfies the 
    source term estimate \eqref{eq:source_term} this completes the proof of 
    Theorem \ref{thm:unbal_correction}.
\end{remark}
\begin{proof}[Proof of Proposition~\ref{prop:unbal_L1}]
    We have already argued that 
        $C^{\res}$
    satisfies the source term estimate \eqref{eq:source_term} in Proposition 
    \ref{prop:b3_good_bad_splitting}. 

    Next we turn our attention to 
    $F^{\ord n}$ when $2 \leq n \leq N$. These are of the form 
    \[
        F^{\ord n} = 
        \sum_h\sum_{j_1+\cdots+j_k = n} (-1)^{k+1}C^{\ord {j_1},\res}_h
        \circ B^{\ord {j_2}}_h \circ \cdots \circ B^{\ord {j_k}}_h.
    \]
    For these, since $C^{\ord {j_1}, \res}$ is applied on the outside, 
    after projecting 
    to frequency $\ell$, the lowest pair of high frequencies in 
    $C^{\ord {j_1}, \res}_h$ 
    is projected to frequency at most $\ell$, so we are happy to apply Bernstein's 
    inequality on this projection. Let us relabel so that this low 
    frequency is $m_1$. Every other pair in $C^{\ord {j_1}, \res}$ or 
    $B^{\ord {j_k}}_h$ projected to frequency $\dyad^m$ has corresponding symbol size 
    $\dyad^{-m}\dyad^{-h}$. 
    Then we split the pair with the largest frequency, say $m_2$, 
    and use the bilinear $L^2_{t,x}$ estimate since the final unpaired frequency 
    must be at most $m_2$. Every other pair we will place in Bernstein's which 
    will produce $\dyad^m$ which cancels the symbol gain.

    Since $N$ is fixed, 
    we do not need to worry about the dependence of the constant on $n$.
    \begin{align*}
        \|P_\ell F^{\ord n}(v) P_{\cong\ell} \bar v\|_{L^1_{t,x}}
        &\lesssim \sum_{h \gg \ell \gtrsim m_1} 
        \sum_{\ell \ll m_2, \ldots, m_n \ll h}\dyad^{-h}\dyad^{-m_2}\sup_{x_0, x_1}
        \|P_h v P_{\cong \ell} \bar v(\cdot + x_0)\|_{L^2_{t,x}}\\
        &\hspace{6em}\|P_h v P_{\lesssim m_2} \bar v(\cdot + x_1)\|_{L^2_{t,x}}
         (\dyad^{\ell} \|P_h v\|_{L^\infty_t L^2_x}^2)
        (\dyad^{-h}\|P_h v\|_{L^\infty_t L^2_x}^2)^{n-2}\\
        &\lesssim C^n\epsilon^{2n+2} \dyad^{(1-s)\ell}\sum_{h \gg m_2 \gg \ell } 
        h^{n-2}\dyad^{(-1-2s)nh} \dyad^{(-1-s)m_2} c_h^{2n} c_{m_2} c_\ell \\
        &\lesssim C^n\epsilon^{2n+2} \dyad^{-2s\ell}\sum_{h \gg \ell} 
        \dyad^{(-1-2s )nh}\dyad^{\delta(h-\ell)} c^2_\ell \\
        &\lesssim C^n\epsilon^{2n+2} \dyad^{-2s\ell} \dyad^{(-1-2s)n\ell}c^2_{\ell}
    \end{align*}
    Note that the frequency envelope loss and the log loss from the extra sums 
    have been absorbed because $s > -1/2$.

    ~

    Finally we turn our attention to $F^{\ord n}$ for $n \geq N+1$. As for
    the dependence of the constant on $n$, there 
    are only finitely many types of multilinear forms since we only go up to 
    $B^{\ord N}_h$ 
    and $C^{\ord N}_h$, so the constants in the symbol sizes are bounded. The number 
    of ordered partitions of $n$ of max size $N$ is controlled by $N^{n}$ 
    so we have no issue here.

    In both of the terms of $F^{\ord n}$ for $n \geq N+1$ there is one pair of high 
    frequencies which does not have a corresponding symbol size of 
    $\dyad^{-h}\dyad^{-m}$. 
    Let us relabel this frequency to $m_1$.
    Thus, this is the same as the 
    previous case except we do not know if $m_1$ is smaller than $\ell$. 
    Without loss of generality, we may assume that $m_1$ is the largest of the 
    medium frequencies as we may move the symbol of the highest pair to control 
    $m_1$ otherwise.

    In this case we will simply use the $m_1$ pair in the bilinear estimate noting 
    that the unpaired frequency can be at most size $m_1$. Again note 
    the log loss from the extra medium frequency sums.
\begin{align*}
    \|P_\ell F^{\ord n}(v) P_{\cong \ell} \bar v\|_{L^1_{t,x}}
        &\lesssim \tilde C^n\sum_{m_1 \ll h \gg \ell } 
        \sup_{x_0, x_1}
        \|P_h v P_{\cong \ell} \bar v(\cdot + x_0)\|_{L^2_{t,x}}
        \|P_h v P_{\lesssim m_1} \bar v(\cdot + x_1)\|_{L^2_{t,x}}\\
        & \qquad 
        (h\dyad^{-h}\|P_h v\|_{L^\infty_t L^2_x}^2)^{n-1}\\
        &\lesssim \tilde C^n\epsilon^{2n+2} \dyad^{-s\ell}\sum_{m_1 \ll h \gg \ell } 
        h^{n-1}\dyad^{(-1-2s)nh} \dyad^{-sm_1} c_h^{2n} c_{m_1} c_\ell \\
        &\lesssim \tilde C^n\epsilon^{2n+2} \dyad^{-2s\ell}\sum_{h \gg \ell} 
        h^{n-1}\dyad^{(-n-(2n+1)s)h}\dyad^{\delta(h-\ell)} c^2_\ell \\
        &\lesssim \tilde C^n\epsilon^{2n+2} \dyad^{-2s\ell} \dyad^{(-1-2s)n\ell}
    \end{align*}
    Note that the above only converges if $-n-(2n+1)s < 0$ which is true as long 
    as $n > N$. Further, note that the log loss does not produce more than 
    an exponentially growing constant in $n$ as each extra factor of $h$ 
    comes with an additional decaying factor $\dyad^{-1 -2s h}$. This completes the 
    proof.

\end{proof}



\section{The frequency envelope bounds}
\label{s:fe-bounds}
The aim of this section is to prove the frequency envelope bounds in 
Theorem~\ref{t:boot}, given the bootstrap assumptions 
\eqref{eq:uj_ee_boot}-\eqref{eq:uj_sep_bi_boot}. 
Our proof will follow in five broad steps. 
\begin{itemize}
\item 
First, an immediate corollary of Theorem 
\ref{thm:unbal_correction} is to translate our bootstrap in the variable $u$ 
to an equivalent bootstrap in the normal form variable $v$. 
\item
Second, we prove fixed time 
estimates for the correction terms $B^\bal$ and spacetime estimates for the 
errors $R^{\geq 6}$, $Q^{4,\bal}$ using the bootstrap 
        assumptions in $v$ and also for the source term 
        $N^{\unbal}$ using Theorem \ref{thm:unbal_correction}. 
\item 
Third, we will 
close the energy estimate bootstrap in $v$ (which also closes $u$ by step one) 
from estimates in the second step. 
\item 
Fourth, we will close the 
localized $L^6_{t,x}$ based and bilinear $L^2_{t,x}$ based estimates using 
the interaction Morawetz identity which will depend more precisely on 
the structure of only the balanced corrections. 
\item 
Finally, we will close the separated, or transversal, bilinear 
estimate from the interaction Morawetz identity, again depending only on 
the balanced correction terms and the estimates from step two.
\end{itemize}

Since the first step is an immediate corollary of Theorem \ref{thm:unbal_correction}
we may proceed as if the bootstrap assumptions 
\eqref{eq:uj_ee_boot}-\eqref{eq:uj_sep_bi_boot} are in the variable $v$.

\subsection{Unbalanced space-time \texorpdfstring{$L^1$}{} bounds}
Here we use the bootstrap assumptions to prove estimates on  
the unbalanced part of the nonlinearity $N^{\unbal}(v)$ in the $v$ equation 
\eqref{eq:new_nls_simp}. 

We begin with estimates of the unbalanced interaction terms 
$N^{\geq 4,\unbal}_{m,j}$ 
and $N^{\geq 4, \unbal}_{p,j}$ defined in the exposition leading up to 
\eqref{eq:dens_flux_mj} and \eqref{eq:dens_flux_pj} 
respectively which follow immediately from Theorem \ref{thm:unbal_correction}.
\begin{corollary}\label{l:Nm_space_time}
    Assume that the bootstrap bounds \eqref{eq:uj_ee_boot}-\eqref{eq:uj_sep_bi_boot}
    hold, that $j\geq 0$ and that $\epsilon$ is small 
    depending on $C$, $s$, and the regularity of the symbol $c$.

    Then we have the following space-time estimate 
    \begin{equation}
        \|N^{\geq 4, \unbal}_{m,j}\|_{L^1_{t,x}} \lesssim C^2\epsilon^4 c_j^2 \dyad^{-2sj}
    \end{equation}
    where the implicit constant is independent of the choice of the bootstrap 
    constant $C$.
\end{corollary}
\begin{corollary}\label{l:Np_space_time}
    Assume that the bootstrap bounds \eqref{eq:uj_ee_boot}-\eqref{eq:uj_sep_bi_boot}
    holds, that $j\geq 0$ and that $\epsilon$ is small 
    depending on $C$, $s$, and the regularity of the symbol $c$.

    Then we have the following space-time estimate 
    \begin{equation}
        \|N^{\geq 4, \unbal}_{p,j}\|_{L^1_{t,x}} \lesssim C^2\epsilon^4 c_j^2 \dyad^j\dyad^{-2sj}
    \end{equation}
    where the implicit constant is independent of the choice of the bootstrap 
    constant $C$.
\end{corollary}


\subsection{Balanced spatial and space-time \texorpdfstring{$L^1$}{} bounds}
Here we consider the corrections $B^{4,\bal}_{m,j}$ and errors 
$Q^{4,\bal}_{m,j}(|\partial |v|^2|^2)$,
$R^{\geq 6}_{m,j}$ and their momentum counterparts
from the density-flux relations \eqref{dens-flux-m}, \eqref{dens-flux-p}.
These bounds will be repeatedly used in each of the following subsections.

\begin{lemma}\label{l:B4_multi}
Assume that the bootstrap bound \eqref{eq:uj_ee_boot} holds. Let $j \geq 0$.

Then we have the fixed time estimate 
\begin{equation}\label{eq:b4_mj}
    \| B^{4,\bal}_{m,j}(v) \|_{L^1_x} \lesssim C^4\epsilon^4 
    c_j^2 \dyad^{(-1-2s)j}\dyad^{-2sj}.
\end{equation}
\end{lemma}
The corresponding bound for the momentum follows as a corollary, as long 
as we account for the size of the momentum symbol:

\begin{corollary}\label{c:B4-multi}
Assume that the bootstrap bound \eqref{eq:uj_ee_boot} holds and 
    let $j \geq 0$. 
Then we have the fixed time estimate 
\begin{equation}\label{eq:b4_pj}
    \| B^{4,\bal}_{p,j}(v) \|_{L^1_x} \lesssim \epsilon^4 C^4 c_j^2 \dyad^{(-1-2s)j}
    \dyad^{(1-2s)j}
\end{equation}
\end{corollary}
Note that in both of these $\dyad^{-1-2s} \leq 1$.

\begin{proof} 
The bounds \eqref{eq:b4_mj} and \eqref{eq:b4_pj} are similar, the only difference arises 
from the additional $\dyad^j$ factor in the size of the symbol $p_j$. 
    So we will prove 
the first bound.

    Recall that $b^{4,\bal}_{m,j}$ 
    has all four frequencies supported in adjacent dyadic intervals to $\dyad^j$ by 
Proposition~\ref{prop:symbols}. Further, by Lemma~\ref{lem:sep}, the 
regularity condition in Proposition~\ref{prop:symbols}, and 
the bootstrap bound \eqref{eq:uj_ee_boot} we have the 
estimate 
\begin{align*}
    \| B^4_{m,j}(v) \|_{L^1_x} &\lesssim 
\dyad^{-2j}\sum_{j_1,\ldots, j_4 \cong j} 
\sup_{x_1,x_2,x_3}\|P_{j_1}(\bar v(\cdot + x_1))P_{j_2}(v(\cdot + x_2))
P_{j_3}(\bar v(\cdot + x_3))P_{j_4}(v)\|_{L^1_x}\\
&\lesssim 
\dyad^{-j}
\|P_{\cong j}(v)\|_{L^2_x}^4\\
&\lesssim C^4\epsilon^4 
\dyad^{(-1 - 4s)j}c_j^4
\end{align*}
Finally, by the slowly varying condition and boundedness of the frequency envelope 
in Definition~\ref{def:admissible_envelope} we have 
\[
    \| B^{4,\bal}_{m,j}(v) \|_{L^1_x} \lesssim C^4  \epsilon^4 c_j^2 \dyad^{(-1-4s)j}
\]
\end{proof}

Next we turn our attention to $Q^{4,\bal}_{m,j}(|\partial |v|^2|^2)$
and its momentum counterpart.

\begin{lemma}\label{l:Q4_space_time}
Assume that the bootstrap bound \eqref{eq:uj_loc_bi_boot} holds. Let $j \geq 0$.

Then we have the space-time estimate 
\begin{equation}\label{eq:q4_mj}
    \| Q^{4,\bal}_{m,j}(|\partial |v|^2|^2) \|_{L^1_{t,x}} \lesssim C^2\epsilon^4 
    c_j^2 \dyad^{(-1-2s)j}\dyad^{-2sj}.
\end{equation}
\end{lemma}
The corresponding bound for the momentum is again a corollary:
\begin{corollary}\label{c:Q4_space_time}
Assume that the bootstrap bound \eqref{eq:uj_loc_bi_boot} holds. Let $j \geq 0$.

Then we have the space-time estimate 
\begin{equation}\label{eq:q4_pj}
    \| Q^{4,\bal}_{p,j}(|\partial |v|^2|^2) \|_{L^1_{t,x}} \lesssim C^2\epsilon^4 
    c_j^2 \dyad^{(-1-2s)j}\dyad^{(1-2s)j}.
\end{equation}
\end{corollary}
\begin{proof}
    Again, the symbol size of $q^{4,\bal}_{p,j}$ differs from 
    $q^{4,\bal}_{m,j}$ by $\dyad^j$, so it suffices to prove the 
    bound for $q^{4,\bal}_{m,j}.$

    By Proposition \ref{prop:symbols} as well as Lemma \ref{lem:sep} we have that 
    \begin{align*}
\| Q^{4,\bal}_{m,j}(|\partial |v|^2|^2) \|_{L^1_{t,x}}
        \lesssim \dyad^{-2j}\left(\sum_{j_1,j_2 \sim j}
        \|\partial (P_{j_1} v P_{j_2} \bar v)\|_{L^2_{t,x}}
        \right)^2\\
        \lesssim C^2 \epsilon^4 \dyad^{-2j}\dyad^{(1-4s)j} c_j^2
    \end{align*}
\end{proof}

Next we turn our attention to $R^{\geq 6}_{m,j}$, which we estimate as follows:

\begin{lemma}\label{l:R6_space_time}
Assume that the bootstrap bounds \eqref{eq:uj_ee_boot}-\eqref{eq:uj_sep_bi_boot} 
hold 
and that $j \geq 0$.

Then we have the space-time bound
\begin{equation}\label{R6_m_bd}
    \|R^{\geq 6}_{m,j}\|_{L^1_{t,x}} \lesssim \epsilon^4 C^6 c_j^2
    \dyad^{(-1-2s)j} \dyad^{-2sj}.
\end{equation}
\end{lemma}

As above, we also have a similar bound for the momentum:

\begin{corollary}\label{c:R6-AB}
Assume that the bootstrap bounds \eqref{eq:uj_ee_boot}-\eqref{eq:uj_sep_bi_boot} 
hold and that $j \geq 0$.

Then we have the space-time bound
\begin{equation}\label{R6_p_bd}
    \|R^{\geq 6}_{p,j}\|_{L^1_{t,x}}
    \lesssim \epsilon^4 C^6 c_j^2 
    \dyad^{(-1-2s)j} \dyad^{(1-2s)j}.
\end{equation}
\end{corollary}

\begin{proof}
    Recall the definition of $R^{\geq 6}_{m,j}$ which 
    is the localized version of \eqref{R6-m-bal}:
    \[
        R^{\geq 6}_{m,j} = 
        B^{4,\bal}_{m,j}(C^{\bal}(v,\bar v, v) + N^{\unbal}(v), \bar v, v, \bar v) + \cdots + B^{4,\bal}_{m,j}(v , \bar v, v, \ol{C^{\bal}(v,\bar v, v)} + \ol{N^{\unbal}(v)}).
    \]
    Again, since the symbol size of $r^{\geq 6}_{p,j}$ differs from 
    $r^{\geq 6}_{m,j}$ by $\dyad^j$ (since this 
    is true for $B^{4,\bal}$), it suffices to prove the 
    bound for $r^{\geq 6}_{m,j}.$

    By the symmetry of the symbol of $B^{4,\bal}_{m,j}$ it will suffice 
    to prove the above bound for each of the following two terms:
    \[
        B^{4,\bal}_{m,j}(C^{\bal}(v,\bar v, v), \bar v, v, \bar v)
        + B^{4,\bal}_{m,j}(N^{\unbal}(v), \bar v, v, \bar v)
    \]

    The first term we control with the Strichartz estimate \eqref{eq:uj_se_boot}
    as well as Proposition \ref{prop:symbols} and Lemma \ref{lem:sep}. Note that 
    in the first term all $6$ frequencies are near $j$ because of the support 
    properties of $C^{\bal}$ and $B^{4,\bal}_{m,j}$. We have 
    \begin{align*}
        \| B^{4,\bal}_{m,j}(C^{\bal}(v,\bar v, v) ,\bar v, v , \bar v)\|_{L^1_{t,x}}
        \lesssim& \ \dyad^{-2j}\sum_{j_1,\ldots, j_6 \sim j}
        \|P_{j_1} v\|_{L^6_{t,x}}\cdots\| P_{j_6} \bar v\|_{L^6_{t,x}}
        \\
        \lesssim& \ C^6 \epsilon^4 \dyad^{-2j}\dyad^{(1-4s)j} c_j^4
        \\
        \lesssim& \ C^6 \epsilon^4 \dyad^{(-1-4s)j} c_j^2
    \end{align*}
    The second term we can easily control with 
    Proposition\ref{prop:symbols}, Lemma~\ref{lem:sep}, 
    Theorem \ref{thm:unbal_correction}, Bernstein's inequality 
    and the bootstrap estimate \eqref{eq:uj_ee_boot}:
    \begin{align*}
        \| B^{4,\bal}_{m,j}(N^{\unbal}(v), \bar v, v, \bar v) \|_{L^1_{t,x}}
        &\lesssim \dyad^{-2j}\sum_{j_1 \cong j_2 \cong j_3 \cong j_4}
        \sup_{x_0}\|P_{j_1} N^{\unbal}(v) P_{j_2} \bar v(x + x_0)\|_{L^1_{t,x}}
        \|P_{j_3} v\|_{L^\infty_{t,x}}\|P_{j_4} v\|_{L^\infty_{t,x}}\\
        &\lesssim C^4 \epsilon^6 c_j^2 \dyad^{(-1-4s)j}
    \end{align*}
    This completes the proof.
\end{proof}


\subsection{The energy estimate}
Here we close the bootstrap for \eqref{eq:uj_ee}. We then remark that, 
after this is proved, we may 
drop the factors of $C$ in Lemma~\ref{l:B4_multi}.

To close this estimate, we fix some dyadic 
region $j$ and consider the mass localized to frequency 
$j$: 
\[
\bM_j(v) = \| P_j v\|_{L^2_x}^2.
\]
To prove \eqref{eq:uj_ee} we need to get a uniform bound on $\bM_j$ 
in time:
\begin{equation}\label{unif}
    \bM_j(v) \lesssim \epsilon^2 c_j^2 \dyad^{-2sj}.
\end{equation}
To achieve this 
we consider
the density flux relation \eqref{eq:dens_flux_mj},
\[
 \partial_t \ms_j(v) = \partial_x(\bbP_{j}(v)
    + R^{4,\bal}_{m,j}(v)) + Q^{4,\bal}_{m,j}(|\partial|v|^2|^2) 
    + N^{\unbal}_{m,j}(v) 
    + R^{\geq 6}_{m,j}(v),
\]
where 
\[
    \ms_j(v,\bar v) = \bbM_j(v,\bar v) + B^{4,\bal}_{m,j}(v).
\]
To prove \eqref{unif} we integrate the 
above density-flux relation in $t,x$ to obtain:
\begin{equation}\label{en-ident}
    \left. \int \bbM_j(v) + B^{4,\bal}_{m,j}(v) \, dx  \right|_0^T 
    = \int_{0}^T \int_\R Q^{4,\bal}_{m,j}(|\partial|v|^2|^2) 
    + N^{\geq 4, \unbal}_{m,j}(v) 
    + R^{\geq 6}_{m,j}(v) \ dx dt.
\end{equation}
Finally, we can estimate the contributions of $B^{4,\bal}_{m,j}$, $Q^{4,\bal}_{m,j}$,
$N^{\geq 4, \unbal}_{m,j}$, and $R^{\geq 6}_{m,j}$ 
using respectively Lemma~\ref{l:B4_multi}, Lemma~\ref{l:Q4_space_time},
Lemma~\ref{l:Nm_space_time}, and Lemma~\ref{l:R6_space_time}.

\begin{remark}
For later use, we observe that once the energy bounds \eqref{eq:uj_ee}
have been established, then they can be used instead of 
the bootstrap assumption \eqref{eq:uj_ee_boot} in the proof
of Lemma~\ref{l:B4_multi}. This leads to a stronger form of 
\eqref{eq:b4_mj}, \eqref{eq:b4_pj}, with the constant $C$ removed:
\begin{equation}\label{B4m_L1_re}
    \| B^{4,\bal}_{m,j}(v)\|_{L^\infty_t L^1_x}    
    \lesssim  \epsilon^4 c_j^2 \dyad^{-2sj} .
\end{equation}
\begin{equation}\label{B4p_L1_re}
    \| B^{4,\bal}_{p,j}(v)\|_{L^\infty_t L^1_x}    
    \lesssim  \epsilon^4 c_j^2 \dyad^{(1-2s)j} .
\end{equation}

\end{remark}

\subsection{The localized interaction Morawetz}
Here we prove the bounds \eqref{eq:uj_se}
and \eqref{eq:uj_loc_bi} using our bootstrap assumptions. 

Fix $j \geq 0$. It will suffice to prove the following bounds 
on the interaction Morawetz functional described in Section 
\ref{s:Morawetz}

\begin{equation}\label{Ij_bound}
    |\bI_{j}(v,v)| \lesssim \epsilon^4 c_j^4 \dyad^{(1-4s)j}, 
\end{equation}
\begin{equation}\label{J4_formula}
  \int_0^T\bJ^4_{j}(v,v)\,dt = 4\| \partial_x |P_j v|^2\|_{L^2_{t,x}}^2   , 
\end{equation}
\begin{equation}\label{J6-bound}
    \int_0^T \bJ^6_j(v,v)\, dt  = 
    \| (C^{\bal}(D,D,D) P_j^4(D))^\frac16 v\|_{L^6_{t,x}}^6
+ O( C^4 \epsilon^5 c_j^4\dyad^{(1-4s)j}) ,
\end{equation}
where $C^{\bal}(D,D,D)P^4_j(D)$ 
is the multiplier associated to the symbol
\[
    c^\bal(\xi, \xi, \xi)p_j^4(\xi),
\]
\begin{equation}
\int_0^T \bJ^8_j(v,v)\, dt = O(C^6 \epsilon^6  c_j^4\dyad^{(1-4s)j}), 
\end{equation}
\begin{equation}
    \int_0^T\bK^{\geq 6}_j(v,v)\, dt = O(C^8 \epsilon^5 c_j^4\dyad^{(1-4s)j}). 
\end{equation}

This allows us to straightforwardly 
estimate the localized interaction Morawetz term in \eqref{eq:uj_loc_bi}
provided $\epsilon$ is small enough.
The localized $L^6_{t,x}$ norm is easily estimated but with 
$(C^\bal(D,D,D)P_j^4(D))^{1/6}$ 
instead of the usual projection $P_j(D)$.
However, by the defocusing property (H3) and regularity property 
(H1s) of $C^\bal$, the discrepancy
between these two projections is a bounded symbol which is smooth 
on scale $\dyad^j$, and thus is bounded on $L^6_x$ by separation of variables, 
see Lemma~\ref{lem:sep}:
\[
\begin{aligned}
    \|P_j v\|_{L^6_{t,x}} &= \|(\frac{P_j^2(D)}{C^\bal(D,D,D)})^{1/6}
(C^\bal(D,D,D)P_j^4(D))^{1/6} v\|_{L^6_{t,x}}\\
    &\lesssim 
    \|(C^\bal(D,D,D)P^4_j(D))^{1/6} v\|_{L^6_{t,x}}
\end{aligned}
\]
(Note the converse can also be controlled using the slowly varying frequency
envelopes.)

There is nothing to do for $\bJ^4_{j}$ so we consider the remaining contributions:

\subsubsection{The $\bI_j$ bound}
The interaction Morawetz functional
$\bI_j$ is as in \eqref{Ia-sharp-def}
\begin{equation*}
\bI_{j} =   \iint_{x > y} \ms_j(v)(x) \ps_{j}(v) (y) -  
\ps_{j}(v)(x) \ms_{j}(v) (y)\, dx dy
\end{equation*}
with
\[
    \ms_j(v)= \bbM_j(v) + B^{4,\bal}_{m,j}(v),
    \qquad 
    \ps_j(v)= \bbP_j(v) + B^{4,\bal}_{p,j}(v).
\]
For $B^{4,\bal}_{m,j}$ and $B^{4,\bal}_{p,j}$ we have the $L^{\infty}_tL^1_x$
bound \eqref{B4m_L1_re} and \eqref{B4p_L1_re}. For $\bbM_j(v)$ and $\bbP_j(v)$
we have the straightforward uniform in time bounds
\begin{equation}\label{M-L1}
    \|\bbM_j(v)\|_{L^\infty_t L^1_x}\lesssim \epsilon^2 c_j^2 \dyad^{-2sj}.
\end{equation}

\begin{equation}\label{P-L1}
    \|\bbP_j(v)\|_{L^\infty_t L^1_x} \lesssim \epsilon^2 c_j^2 \dyad^{(1-2s)j}.
\end{equation}
Combining these, the estimate 
\eqref{Ij_bound} immediately follows.

\subsubsection{The $\bJ^6_j$ bound} 
This is a $6$-linear
expression whose expression we recall from \eqref{eq:J6_def},
\[
    \bJ^6_{j} =  2 \int -( \bbP_{j} B^{4,\bal}_{p,j}
    + \bbP_{j} R^{4,\bal}_{m,j}) + ( 
    \bbM_j R^{4,\bal}_{p,j} +\bbE_{j} B^{4,\bal}_{m,j})\, dx.
\]

We first notice that the mass $\bbM_j$, momentum $\bbP_j$ and energy $\bbE_j$ 
as well as the corrections $B^{4,\bal}_{m,j},$ $B^{4,\bal}_{p,j}$,
$ R^{4,\bal}_{m,j},$ and $R^{4,\bal}_{p,j}$ are all localized to a region 
where all frequencies are in adjacent dyadic intervals to $j$.

Next we recall the size of the symbols from
Proposition~\ref{prop:symbols}: 
\[
    |b^{4,\bal}_{m,j}| \lesssim \dyad^{-2j},
    \qquad 
    |b^{4,\bal}_{p,j}| \lesssim \dyad^{-j},
\]
\[
|r^{4,\bal}_{m,j}| \lesssim \dyad^{-j},
\qquad
|r^{4,\bal}_{p,j}| \lesssim 1,
\]
and similarly for their derivatives. 
And also the size of the symbols of $\bbM_j, \bbP_j,$ and $\bbE_j$ which 
simply come from their polynomial prefactors:
\[
    |\bbm_j| \lesssim 1, \qquad 
    |\bbp_j| \lesssim \dyad^j, \qquad 
    |\bbe_j| \lesssim \dyad^{2j}.
\]
We see that in \eqref{eq:J6_def} each $\dyad^\alpha$ is paired with 
$\dyad^{-\alpha}$, so by Lemma~\ref{lem:sep} we can ignore 
these contributions going forward.
Also importantly, 
we know that on the diagonal 
\[
\{ \xi_1 = \xi_2 = \xi_3 = \xi_4 = \xi_5 = \xi_6 \},
\]
we have 
\[
j^6_j(\xi) = p_j^4(\xi) c^\bal(\xi, \xi, \xi).
\]

It follows that  we can write the symbol $j^6_j$ in the form
\[
j^6_j(\xi_1,\xi_2,\xi_3,\xi_4,\xi_5,\xi_6) = b(\xi_1) b(\xi_2) b(\xi_3) b(\xi_4) b(\xi_5) b(\xi_6)  + 
j^{6,rem}_j(\xi_1,\xi_2,\xi_3,\xi_4,\xi_5,\xi_6) ,
\]
where $b(\xi) = p_j(\xi)^\frac23 c^\bal(\xi, \xi, \xi)^\frac16$ and $j^{6,rem}_j$
vanishes when all $\xi$'s are equal. 

Because we may take the symbol $j^{6,rem}_j$ to be 
symmetric in odd frequencies and separately even frequencies, we can write
$j^{6,rem}_j$ as a linear combination of terms $\xi_{odd} - \xi_{even}$ 
with smooth coefficients; this is a simpler version of 
the division lemma, \ref{l:division} where we only factor 
out \emph{one} derivative in the relevant directions. Factoring out 
this derivative will produce a factor of $\dyad^{-j}$ since 
all symbols are smooth on the dyadic scale.

The first term 
yields the desired $L^6_{t,x}$ norm,
\[
\int_0^T\bJ^6_j(v)\, dt = \| B(D) v\|_{L^6_{t,x}}^6 + \int_0^T\bJ^{6,rem}_j(v) dt.
\]

On the other hand the contribution $\bJ^{6,rem}_j$ of the second term can be 
estimated using a bilinear $L^2_{t,x}$ bound \eqref{eq:uj_sep_bi_boot}, 
three $L^6_{t,x}$ bounds \eqref{eq:uj_se_boot} 
and one $L^\infty_{t,x}$ via Bernstein's inequality and \eqref{eq:uj_ee} (which
we have already closed),
\begin{align*}
    \left|\int_0^T\bJ^{6,rem}_j(v)\, dt \right| 
    &\lesssim \dyad^{-j}
    \sup_{x_0} \|\partial(P_{\cong j} v(x) P_{\cong j} \bar v(x+x_0))\|_{L^2_{t,x}}
    \|P_{\cong j} v\|_{L^6_{t,x}}^3 \|P_{\cong j} v\|_{L^\infty_{t,x}}\\
    &\lesssim C^4\epsilon^5 c_j^5
    \dyad^{-j}\dyad^{(1/2 - 2s)j}(\dyad^{(1-4s)j/6})^3\dyad^{(1/2 -s)j}\\
    &\lesssim C^4\epsilon^5 c_j^4 \dyad^{(1-4s)j} \dyad^{(-1/2-s)j}
\end{align*}
which suffices since $-1/2 - s < 0$.

\subsubsection{The bound for $\bJ^8_j$}
We recall that $\bJ^8_j$ is defined in \eqref{eq:J8_def} which we copy 
here:
\[
\bJ^8_{j}(v,v) = 2  \int 
    B^{4,\bal}_{m,j}(v) R^{4,\bal}_{p,j}(v) 
    - R^{4,\bal}_{m,j}(v) B^{4,\bal}_{p,j}(v)    
 \, dx .
\]

For this we need to show that 
\[
\left| \int_0^T\bJ^8_j\, dt\right| \lesssim C^6\epsilon^6c_j^4\dyad^{(1-4s)j}.
\]
Again, all 8 frequencies will be localized to adjacent regions,
each comparable to $j$. We see from the symbol sizes in 
Proposition~\ref{prop:symbols}, that each term will come with a factor 
comparable to $\dyad^{-2j}$ by Lemma~\ref{lem:sep}. Then, we can 
estimate this using six $L^6_{t,x}$ bounds \eqref{eq:uj_se_boot} 
and two $L^\infty_{t,x}$ via Bernstein's inequality by \eqref{eq:uj_ee}.

\begin{align*}
    \left| \int_0^T\bJ^8_j\, dt\right| 
    &\lesssim  \dyad^{-2j}
    \|P_{\cong j} v\|_{L^6_{t,x}}^6
    \|P_{\cong j} v\|_{L^\infty_{t,x}}^2 \\
    &\lesssim C^6\epsilon^6 c_j^6
    \dyad^{-2j}(\dyad^{(1 - 4s)j/6})^6(\dyad^{(1/2-s)j})^2\\
    &\lesssim C^6\epsilon^6 c_j^4 \dyad^{(1-4s)j} \dyad^{(-1-2s)j}
\end{align*}
which suffices since $-1 - 2s < 0$.


\subsubsection{The bound for $\bK^{\geq 6}_j$} 
We recall from \eqref{eq:Kg6_def} that $\bK^{\geq 6}_j$ has the form
\begin{align*}
    \bK^{\geq 6}_j &= \iint_{x > y} \ms_j(v)(x) 
    (Q^{4,\bal}_{p,j}(|\partial|v|^2|^2)(y) + 
    N_{p,j}^{\geq 4, \unbal}(v)(y) + R^{\geq 6}_{p,j}(v)(y))\\
    &\qquad + \ps_j(v)(y) (Q^{4,\bal}_{m,j}(|\partial|v|^2|^2)(x) 
    + N_{m,j}^{\geq 4, \unbal}(v)(x) + R^{\geq 6}_{m,j}(v)(x)) \, dx dy \\ 
    \quad &- \iint_{x > y}
    \ms_j(v)(y) (Q^{4,\bal}_{p,j}(|\partial|v|^2|^2)(x) 
    + N_{p,j}^{\geq 4, \unbal}(v)(x) + R^{\geq 6}_{p,j}(v)(x))\\
    &\qquad + \ps_j(v)(x) (Q^{4,\bal}_{m,j}(|\partial|v|^2|^2)(y) 
    + N_{m,j}^{\geq 4, \unbal}(v)(y) + R^{\geq 6}_{m,j}(v)(y)) \, dx dy,
\end{align*}

The time integral of $\bK^{\geq 6}_j(v)$ is estimated directly using the 
$L^1_{t,x}$ bound for $Q^{4,\bal}_{m,j}$ in Lemma~\ref{l:Q4_space_time}, 
$N_{m,j}$ in Lemma~\ref{l:Nm_space_time}, $R^{\geq 6}_{m,j}$ 
in Lemma~\ref{l:R6_space_time} and the uniform $L^1_x$ bound
for $\ms$ and $\ps$, provided by Lemma~\ref{l:B4_multi} and 
together with the simpler bound \eqref{M-L1} as well 
as their momentum counterparts given as corollaries below each lemma.

\subsection{Near parallel interactions}

Here we briefly discuss the bilinear $L^2_{t,x}$ bound \eqref{eq:uj_sep_bi}
in the case when $j$ and $k$ are comparable. 
This can be viewed on one hand as a slight generalization of the argument 
in the previous subsection, where instead of $w = v$ we take $w = v(\cdot +x_0)$. 
The only difference in the proof is that, because of the translations, 
we can no longer use the defocusing property to control the sign of the diagonal 
$\bJ^{6}$ contribution. 
However, this is not a problem because the localized $L^6_{t,x}$ norm of $u_j$ 
has already been estimated in the previous subsection.

\subsection{The transversal bilinear \texorpdfstring{$L^2_{t,x}$}{}  estimate} 
Here we prove the bilinear $L^2_{t,x}$ bound \eqref{eq:uj_sep_bi} in 
the case where $k \ll j$.
This repeats the same analysis as before, but
using the interaction Morawetz functional associated to two separated 
dyadic frequency intervals corresponding to $\dyad^k$ and $\dyad^j$.
Here we no longer take $w = v$, and instead we let $w = v(\cdot+x_0)$. 
The parameter $x_0 \in \R$
is arbitrary and the estimates are uniform in $x_0$.

Since $x_0$ does not play any role in the analysis, we simply drop it from our 
notations. 

We copy the interaction functional from \eqref{interaction-bi} below:
\begin{equation*}
 \bI_{jk}(v,w) = \iint_{x > y} \ms_j(v)(x) \ps_{k}(w)(y)    
    - \ps_{j}(v)(x) \ms_k(w)(y) \,dx dy .
\end{equation*}
Its time derivative is given in \eqref{interaction-xi-AB}, by
\begin{equation*} \frac{d}{dt} \bI_{jk} =  \bJ^4_{jk} + \bJ^6_{jk} + \bJ^8_{jk} + 
    \bK^{\geq 6}_{jk} .
\end{equation*}
Following the same pattern as in the earlier case of the localized interaction 
Morawetz argument, we will estimate each of these
terms as follows:
\begin{equation}\label{Ijk-bound}
    |\bI_{jk}(v,w)| \lesssim  (\dyad^{j} + \dyad^k) \dyad^{-2sj}\dyad^{-2sk}
    \epsilon^4 c_j^2 c_k^2  ,
\end{equation}
\begin{equation}\label{J4jk-formula}
  \int_0^T\bJ^4_{jk}(v,w) \, dt= 4 \| \partial_x (P_j v P_k\bw)\|_{L^2_{t,x}}^2   , 
\end{equation}
\begin{equation}\label{J6jk-bound}
\left|\int_0^T\bJ^6_{jk}\, dt \right| 
    \lesssim C^2 \epsilon^6 (\dyad^{j} + \dyad^k) \dyad^{-2sj}\dyad^{-2sk}
     c_j^2 c_k^2 ,
\end{equation}
\begin{equation}\label{J8jk-bound}
\left|\int _0^T\bJ^8_{jk}\, dt\right| \lesssim  C^2 \epsilon^8  c_j^2 c_k^2
    (\dyad^{j} + \dyad^k) \dyad^{-2sj}\dyad^{-2sk} ,
\end{equation}
\begin{equation}\label{Kg6jk-bound}
    \left| \int_0^T \bK^{\geq 6}_{jk}\, dt \right| 
    \lesssim \epsilon^6  C^6  c_j^2 c_k^2 
    (\dyad^{j} + \dyad^k) \dyad^{-2sj}\dyad^{-2sk}.
\end{equation}

Again, the $J^4_{jk}$ bound follows from the linear computation
in Section~\ref{s:Morawetz-lin}.

\subsubsection{ The fixed time estimate for $\bI_{jk}$}
This estimate follows in exactly the same way as 
in the localized case by using 
the bounds \eqref{B4m_L1_re}, \eqref{B4p_L1_re}, \eqref{M-L1} and 
\eqref{P-L1}. 

\subsubsection{ The bound for $\bJ^6_{jk}$}
Here we prove the bound for $\bJ^6_{jk}$ in \eqref{J6jk-bound}. 
We recall from \eqref{eq:J6jk_def} that $\bJ^6_{jk}$ has the form
\begin{equation*}
\begin{aligned}
    \bJ^6_{jk}(v,w) &= \int -( \bbP_{j}(v) B^{4,\bal}_{p,k}(w)
    + \bbP_{k}(w) R^{4,\bal}_{m,j}(v)) + ( 
    \bbM_j(v) R^{4,\bal}_{p,k}(w) +\bbE_{k}(w) B^{4,\bal}_{m,j}(v))\\  
    &\qquad -( \bbP_{k}(w) B^{4,\bal}_{p,j}(v)
    + \bbP_{j}(v) R^{4,\bal}_{m,k}(w))+ ( 
    \bbM_k(w) R^{4,\bal}_{p,j}(v) +\bbE_{j}(v) B^{4,\bal}_{m,k}(w))\, dx.
\end{aligned}
\end{equation*}
Again, note that all $v$ frequencies will be comparable 
to $\dyad^j$ and all $w$ frequencies will be comparable 
to $\dyad^k$ because of the choices in Proposition~\ref{prop:symbols}.

The symbols for the $\bbM_j$, $\bbP_j$ and $\bbE_j$ factors have size 
$1$, $\dyad^j$ and $\dyad^{2j}$ respectively and similarly for $k$. 
Conversely, by Proposition~\ref{prop:symbols}, 
$B^{4,\bal}_{m,j}$, $B^{4,\bal}_{p,j}$, $R^{4,\bal}_{m,j}$
and $R^{4,\bal}_{p,j}$ have sizes
$\dyad^{-2j}$, $\dyad^{-j}$, $\dyad^{-j}$, and $1$ respectively and 
similarly for $k$.

Since we have chosen $j \gg k$, the term with the 
largest symbol size will be $\bbE_j(v) B^{4,\bal}_{m,k}$
with symbol size $\dyad^{2j}\dyad^{-2k}$. One might view 
this as the worst term because of this symbol size, but note 
that we also have four $w$'s at low frequency in the above. However, 
according to the structure above when we have many copies of $v$ 
at high frequency, we also have a better symbol size. These effects 
will cancel each other. We thus prove just the two extreme cases
$\bbE_j(v) B^{4,\bal}_{m,k}$ and $\bbM_k(w)R^{4,\bal}_{p,j}$
for brevity.

For both of these we use two bilinear $L^2_{t,x}$ bounds by 
\eqref{eq:uj_sep_bi_boot}
and two $L^\infty_{t,x}$ bounds via Bernstein's inequality by 
\eqref{eq:uj_ee}. Note that while there is no spatial derivative 
present for $\eqref{eq:uj_sep_bi_boot}$, since the frequencies 
are separated we will use the fact that $\dyad^{-j}\partial $ is 
a size one multiplier on $(P_j v P_k w)$.
\begin{align*}
\left|\int_0^T\int \bbE_j(v) B^{4,\bal}_{m,k}\, dx dt \right| 
    &\lesssim \dyad^{2j} \dyad^{-2k} \|P_{j} v P_{\cong k} w\|_{L^2_{t,x}}^2
    \|P_{\cong k} w\|_{L^\infty_{t,x}}^2\\
    &\lesssim C^2 \epsilon^6 c_j^2 c_k^4\dyad^{2j} \dyad^{-2k} 
    (\dyad^{-j/2}\dyad^{-sj}\dyad^{-sk})^2
    (\dyad^{k/2}\dyad^{-sk})^2\\
    &\lesssim C^2 \epsilon^6 c_j^2 c_k^2\dyad^{j} \dyad^{-2sj}\dyad^{-2sk}
    \dyad^{(-1 - 2s)k} 
\end{align*}
which suffices because $-1 - 2s < 0$.

\begin{align*}
\left|\int_0^T\int \bbM_k(w)R^{4,\bal}_{p,j}\, dx dt \right| 
    &\lesssim \|P_{\cong j} v P_{k} w\|_{L^2_{t,x}}^2
    \|P_{\cong j} v\|_{L^\infty_{t,x}}^2\\
    &\lesssim C^2 \epsilon^6 c_j^4 c_k^2 
    (\dyad^{-j/2}\dyad^{-sj}\dyad^{-sk})^2
    (\dyad^{j/2}\dyad^{-sj})^2\\
    &\lesssim C^2 \epsilon^6 c_j^2 c_k^2\dyad^{j} \dyad^{-2sj}\dyad^{-2sk}
    \dyad^{(-1 - 2s)j} 
\end{align*}
which again suffices because $-1 - 2s < 0$.


\subsubsection{ The bound for $\bJ^8_{jk}$}
Here we prove the bound \eqref{J8jk-bound}.
We recall from \eqref{eq:J8jk_def} that $\bJ^8_{jk}$ has the form
\[
    \bJ^8_{jk} = \iint B^{4,\bal}_{m,j}(v) R^{4,\bal}_{p,k}(w) 
    - B^{4,\bal}_{p,k} (w) R^{4,\bal}_{m,j}(v)
+ B^{4,\bal}_{m,k}(w) R^{4,\bal}_{p,j}(v) - B^{4,\bal}_{p,j}(v) R^{4,\bal}_{m,k}(w) 
\, dxdt.
\]
Here there is unambiguously a worst term which is 
$B^{4,\bal}_{m,k}(w) R^{4,\bal}_{p,j}(v)$ 
since the symbol size is $\dyad^{-2k}$ and every term 
has four $v$'s and four $w$'s. Thus it will suffice to 
control this term. We use two bilinear $L^2_{t,x}$ bounds by 
\eqref{eq:uj_sep_bi_boot} and four $L^\infty_{t,x}$ bounds 
via Bernstein's inequality by \eqref{eq:uj_ee}.

\begin{align*}
\left|\int_0^T\int B^{4,\bal}_{m,k}(w) R^{4,\bal}_{p,j}(v)\, dx dt \right| 
    &\lesssim  \dyad^{-2k} \|P_{\cong j} v P_{\cong k} w\|_{L^2_{t,x}}^2
    \|P_{\cong j} v\|_{L^\infty_{t,x}}^2\|P_{\cong k} w\|_{L^\infty_{t,x}}^2\\
    &\lesssim C^2 \epsilon^8 c_j^4 c_k^4 \dyad^{-2k} 
    (\dyad^{-j/2}\dyad^{-sj}\dyad^{-sk})^2
    (\dyad^{j/2}\dyad^{-sj})^2(\dyad^{k/2}\dyad^{-sk})^2\\
    &\lesssim C^2 \epsilon^8 c_j^2 c_k^2\dyad^{j} \dyad^{-2sj}\dyad^{-2sk}
    \dyad^{(-1 - 2s)j} \dyad^{(-1 - 2s)k} 
\end{align*}
which suffices since $-1-2s < 0$.

\subsubsection{ The bound for   $\bK^{\geq 6}_{jk}$} 
Again, we use the 
$L^1_{t,x}$ bounds for $Q^{4,\bal}_{m,j}$ in Lemma~\ref{l:Q4_space_time}, 
$N_{m,j}^{\geq 4, \unbal}$ in Lemma~\ref{l:Nm_space_time}, $R^{\geq 6}_{m,j}$ 
in Lemma~\ref{l:R6_space_time} and the uniform $L^1_x$ bound
for $\ms$ and $\ps$, provided by Lemma~\ref{l:B4_multi} and 
together with the simpler bound \eqref{M-L1} as well 
as their momentum counterparts given as corollaries below each lemma.

Unlike the localized case, here some terms are worse since 
the symbol size of the momentum may land on the larger 
frequency. However, this is exactly accounted for 
in \eqref{Kg6jk-bound} by the fact that 
$(\dyad^j + \dyad^k) \sim \dyad^j$.


\section{Global bilinear and Strichartz estimates}
\label{s:global}
Our objective in this last section is to complete the 
proofs of Theorems \ref{t:main} and \ref{t:NLS-Hs}. 

For Theorem 
\ref{t:main}, the $H^s$ bound follows directly from \eqref{eq:uj_ee} in Theorem
\ref{t:boot}. All that remains is to supplement the frequency localized 
bilinear $L^2_{t,x}$ and Strichartz estimates of Theorem \ref{t:boot}
with their more global counterparts. Unlike in previous work 
\cite{IT-global}, 
since the scale on which the bilinear $L^2_{t,x}$ bound holds depends on
the frequency, we cannot write a global bilinear bound which 
depends on $|u|^2$. Further, this discrepancy in the size where 
we control the $L^2_{t,x}$ bound causes a log loss when the frequencies are 
separated.

To capture this 
we can use the paraproduct notation:
\[
    T_f g = \sum_{j \ll k} P_j f P_k g
\]
and 
\[
    \Pi(f, g) = \sum_{j \sim k} P_j f P_k g.
\]
Now we are ready to state the global estimates.

\begin{proposition}
The global small data solutions $u$ for \eqref{nls}
obtained from Theorem~\ref{t:boot}  and the continuation argument satisfy the following bounds: 

\begin{itemize}
    \item Strichartz estimate:
    \begin{equation}\label{eq:global_se}
        \| \la D\ra^{-(1-4s)/6} u\|_{L^6_{t,x}}^6 \lesssim \epsilon^4,
    \end{equation}
    \item Bilinear $L^2_{t,x}$ bound:
    \begin{equation}\label{eq:global_bi_para}
    \begin{aligned}
        \|\partial_x T_{\la D\ra^{s-\delta_0} \bar u} \la D\ra^{s-1/2}u\|_{L^2_{t,x}}^2 
        &\lesssim \epsilon^4.\\
        \|\partial_x \Pi(\la D\ra^{s-1/4}u,\la D\ra^{s-1/4}\bar u)\|_{L^2_{t,x}}^2 
        & \lesssim \epsilon^4.
    \end{aligned}
    \end{equation}
    where $\delta_0 > 0$ is an arbitrary constant.

\end{itemize}
\end{proposition}

\begin{proof}
    We prove the estimates in turn.
    \bigskip

    \emph{A. The global $L^6_{t,x}$ estimate.} We prove the $L^6_{t,x}$ estimate 
    \eqref{eq:global_se} using \eqref{eq:uj_ee}, \eqref{eq:uj_se}, and 
    \eqref{eq:uj_sep_bi}. 

First, we use a spatial dyadic decomposition  in the integral
\begin{align*}
I = \iint_{\R \times \R} |\la D\ra^{-(1-4s)/6} u|^6 \, dx dt
= \sum_{j_1,\ldots, j_6} \iint_{\R \times \R} 
    \la D\ra^{-(1-4s)/6} P_{j_1}u\cdots \la D\ra^{-(1-4s)/6}P_{j_6}\bar u\, dx dt.
\end{align*}

Notice that for a particular combination of dyadic intervals 
to contribute to the sum, we must have that the highest two are comparable. We 
split the remaining contributing cases into three cases:
\begin{enumerate}
    \item All frequencies are comparable $j_1 \sim \cdots \sim j_6$.
    \item Exactly one frequency is much smaller than the highest, without loss of 
        generality assume it is $j_6$. (Note $\|u \bar v\|_{L^2_{t,x}} = \| u v\|_{L^2_{t,x}}$).
    \item Two or more frequencies are much smaller than the highest two. Without 
        loss of generality, we assume $j_1$ and $j_2$ are highest and that 
        $j_5$ and $j_6$ are smaller.
\end{enumerate}
Based on this we split 
\[I = I_1 + I_2 + I_3\]
where 
\[I_1 = \sum_{j_1 \sim \cdots \sim j_6}
\iint_{\R \times \R} 
    \la D\ra^{-(1-4s)/6} P_{j_1}u\cdots \la D\ra^{-(1-4s)/6}P_{j_6}\bar u\, dx dt,
\]
\[I_2 = \sum_{j_1 \sim \cdots \sim j_5} \sum_{j_6 \ll j_1}
\iint_{\R \times \R} 
    \la D\ra^{-(1-4s)/6} P_{j_1}u\cdots \la D\ra^{-(1-4s)/6}P_{j_6}\bar u\, dx dt,
\]
\[I_3 = \sum_{j_1 \sim j_2} \sum_{j_5,j_6 \ll j_1}\sum_{j_3,j_4 \lesssim j_1}
\iint_{\R \times \R} 
    \la D\ra^{-(1-4s)/6} P_{j_1}u\cdots \la D\ra^{-(1-4s)/6}P_{j_6} \bar u\, dx dt,
\]
We analyze each of these separately. 

For $I_1$ we use the $L^6_{t,x}$ bound:
\[I_1 \lesssim \sum_{j_1} \epsilon^4 c_{j_1}^4 \lesssim \epsilon^4,\]
which suffices. Note that while here we required $\ell^4$ summability of 
the frequency envelope $c_j$, for the remaining two estimates 
the loss of derivatives in the $L^6_{t,x}$ estimate will 
allow for these to sum regardless of summability in $c_j$.

    For $I_2$ we use one bilinear $L^2_{t,x}$ bound between $P_{j_1}$ and 
    $P_{j_6} u$, one $L^\infty_{t,x}$ on $P_{j_2}u$ via Bernstein's 
    inequality, and the remaining factors we use $L^6_{t,x}$: 
\begin{align*}
    I_2 &\lesssim \epsilon^5 
    \sum_{j_1} \sum_{j_6 \ll j_1}\dyad^{-5(1-4s)j_1/6}
    \dyad^{-(1-4s)j_6/6}(\dyad^{(-1/2-s)j_1}\dyad^{-sj_6})\dyad^{(1/2 - s)j_1}
    \dyad^{(1-4s)j_1/2}
    c_{j_1}^4 c_{j_6}\\
        &=\epsilon^5 
\sum_{j_1} \dyad^{-(1+2s)j_1/3}c_{j_1}^4 
    \sum_{j_6 \ll j_1} \dyad^{- (1+2s)j_6/6} c_{j_6}\\
\end{align*}
Now, since $-1 - 2s < 0$, both dyadic sums converge giving
\[
    I_2 \lesssim \epsilon^5
\]
which suffices. 

Finally, for $I_3$ we use two bilinear $L^2_{t,x}$ bounds \eqref{eq:uj_sep_bi} 
pairing $P_{j_1}uP_{j_5}u$ 
and $P_{j_2}uP_{j_6}u$ and two $L^\infty_{t,x}$ $\eqref{eq:uj_ee}$ 
bounds via Bernstein's inequality on 
$P_{j_3}u$ and $P_{j_4}u$. We get 
\begin{alignat*}{2}
    I_3 &\lesssim \epsilon^6\sum_{j_1}\sum_{j_3,j_4 \lesssim j_1}
    \sum_{j_5,j_6 \ll j_1}&&
    \dyad^{-2(1-4s)j_1/6}\dyad^{-(1-4s)j_3/6}\dyad^{-(1-4s)j_4/6}
\dyad^{-(1-4s)j_5/6}\dyad^{-(1-4s)j_6/6}\\
    & &&(\dyad^{(-1 - 2s)j_1}\dyad^{-sj_5}\dyad^{-sj_6}c_{j_1}^2c_{j_5}c_{j_6})
    (\dyad^{(1/2-s)j_3}\dyad^{(1/2-s)j_4}c_{j_3}c_{j_4}) \\
    &\lesssim \epsilon^6\sum_{j_1}\dyad^{-j_1 + (-1 - 2s)j_1/3}
    \sum_{j_3,j_4 \lesssim j_1}&& \dyad^{j_3/2 + (-1 - 2s)j_3/6}\dyad^{j_4/2 
    + (-1 - 2s)j_4/6}
    \sum_{j_5,j_6 \ll j_1}
\dyad^{(-1-2s)j_5/6}\dyad^{-(1+2s)j_6/6}\\
        &\lesssim \epsilon^6\sum_{j_1} \dyad^{-\frac23(1 + 2s)j_1}\\
        &\lesssim \epsilon^6
\end{alignat*}
which suffices.

~

\emph{B. The global bilinear $L^2_{t,x}$ estimate.} 
We prove 
the estimate \eqref{eq:global_bi_para}
using the localized estimates \eqref{eq:uj_loc_bi} 
and \eqref{eq:uj_sep_bi}.

\emph{B1:} We start with the nearby piece:

We expand the definition of $\Pi$:

\begin{align*}
    \|\partial_x \Pi(\la D \ra^{s-1/4} \bar u, 
    \la D\ra^{s-1/4}u)\|_{L^2_{t,x}}
    &= \|\sum_{j_1 \sim k_1}
    \partial_x \left(\la D\ra^{s-1/4} P_{j_1}\bar u 
    \la D\ra^{s-1/4}P_{k_1}u\right)\|_{L^2_{t,x}}
\end{align*}
Here we estimate this as 
\begin{align*}
    \|\partial_x \Pi(\la D\ra^{s-1/4} \bar u, 
    \la D\ra^{s-1/4}u)\|_{L^2_{t,x}}
    &= \|\sum_{j \sim k}
    \partial_x \left(\la D\ra^{s-1/4} P_{j}\bar u 
    \la D\ra^{s-1/4}P_{k}u\right)\|_{L^2_{t,x}}\\
    &\leq \sum_{j \sim k} \|
    \partial_x \left(\la D\ra^{s-1/4} P_{j}\bar u 
    \la D \ra^{s-1/4}P_{k}u\right)\|_{L^2_{t,x}}\\
    &\lesssim \sum_{j \sim k} \dyad^{-j/2}\dyad^{2sj}
    \|\partial(P_{j}u P_{k}\bar u)\|_{L^2_{t,x}}\\
    &\lesssim \sum_{j \sim k} \epsilon^2 c_j^2\\
    &\lesssim \epsilon^2 
\end{align*}

\emph{B2:} Next we handle the separated piece. In this piece we have 
a $\delta_0$ loss from the separation of the low frequencies:

We expand the definition of $T$ in:
\begin{align*}
    I &=  \iint_{\R \times \R} \lvert\partial_x T_{\la D\ra^{s-\delta_0} \bar u} \la D
    \ra^{s-1/2}u\rvert^2 \, dxdt\\
    &= \sum_{j_1 \ll k_1}\sum_{j_2 \ll k_2}
    \iint_{\R \times \R} \partial_x \left(\la D\ra^{s-\delta_0} P_{j_1}\bar u 
    \la D\ra^{s-1/2}P_{k_1}u\right)\partial_x\left( \la D\ra^{s-\delta_0} P_{j_2}u 
    \la D \ra^{s-1/2}P_{k_2} \bar u\right) \, dxdt
\end{align*}
Because of the orthogonality of the Littlewood Paley projectors, 
for a particular combination of $j_1, k_1, j_2, k_2$ to 
contribute to the sum, we must have that $k_1 \sim k_2$. Then 
without loss of generality, let us assume that $j_1 \leq j_2$.

Then we use two bilinear $L^2_{t,x}$ estimates \eqref{eq:uj_sep_bi}.

We can see that 
\begin{align*}
    I \lesssim &\sum_{j_1 \leq j_2 \ll k_2 \sim k_1}
    \dyad^{k_1/2}\dyad^{j_1(s - \delta_0)}\dyad^{k_1s}
    \dyad^{k_2/2}\dyad^{j_2(s - \delta_0)}\dyad^{k_2s}
    \sup_{x_0, x_1}\|P_{k_1}uP_{j_1}\bar u(\cdot + x_0)\|_{L^2_{t,x}}  
    \|P_{k_2}uP_{j_2}\bar u(\cdot + x_1)\|_{L^2_{t,x}}\\
\lesssim &\epsilon^4 \sum_{j_1 \leq j_2 \ll k}
    c_{k}^2 \dyad^{-\delta_0j_1}\dyad^{-\delta_0j_2}c_{j_1} c_{j_2}\\
\lesssim &\epsilon^4 \sum_{k} c_{k}^2\\
\lesssim &\epsilon^4 
\end{align*}

\end{proof}

Finally, we prove Theorem \ref{t:NLS-Hs} in the case of \eqref{nls3}$(+)$ by 
rescaling. In particular, we use the scaling symmetry
\[
u^{\lambda}(x,t) = \lambda^{-1} u(x/\lambda,  t/\lambda^2).
\]
of \eqref{nls3}$(+)$. 

Since Theorem \ref{t:main} is an inhomogeneous result, the low 
frequencies will behave differently than the high frequencies. 
In particular, our proof of Theorem \ref{t:boot} allows arbitrary
mixing of frequencies smaller than $1$ which was 
important since we required negative dyadic factors like $\dyad^{(-1/2-s)j}$
to be bounded for all $j$. However, rescaling changes where frequency $1$ lives,
and thus, controls the threshold between dyadic frequencies being 
controlled by initial frequency envelopes and low frequency mixing. Thus, 
the most precise large data global frequency result uses 
inhomogeneous Sobolev spaces adapted to the frequency 
scale on which mixing happens.
We see this in the 
following proposition, from which Theorem \ref{t:NLS-Hs} will be a corollary.

\begin{proposition}\label{prop:large}
Consider the defocusing 1-d cubic NLS problem \eqref{nls3}$(+)$ with $H^s$
initial data $\du_0$ and fix $-\frac12 < s < 0$. Let 
\[
    \lambda = (1+\|\du_0\|_{H^s})^{\frac2{1+2s}}
\]
   and put 
    \[
        \la \xi \ra_\lambda = (|\xi|^2 + \lambda^2)^{1/2} \qquad \mbox{with } 
        \la D\ra_\lambda \mbox{ the associated
        Fourier multiplier.}
    \]
    Then the global solution $u$ 
satisfies the following bounds:
\begin{enumerate}[label=(\roman*)]
\item Uniform $H^s$ bound:
\begin{equation}\label{main-Hs-model+}
    \|\la D\ra_\lambda^{s} u \|_{L^\infty_t L^2_x} 
    \lesssim \|\du_0\|_{H^s_x}.
\end{equation}

\item Strichartz bound:
\begin{equation}\label{main-s-Str-model+}
    \|\la D\ra_\lambda^{-\frac{1-4s}{6}} u \|_{L^6_{t,x}} 
    \lesssim \|\du_0\|_{H^s_x}^{2/3}. 
\end{equation}

\item Bilinear Strichartz bound: for every $\delta_0 > 0$ 

\begin{equation}\label{main-bi-Hs-model+}
\begin{aligned}
    \big\| \partial_x\, T_{\la D\ra_\lambda^{s-\delta_0}\bar u}\, 
    \la D\ra_\lambda^{s-\frac12} u
    \big\|_{L^2_{t,x}} &\lesssim \|\du_0\|_{H^s}^2
    (1 + \|\du_0\|_{H^s}^{2\delta_0/(1+2s)})^{-1},\\[2pt]
    \big\| \partial_x\, \Pi\big(\la D\ra^{s-\frac14}_{\lambda} u,\, 
    \la D\ra^{s-\frac14}_{\lambda}\bar u\big)
\big\|_{L^2_{t,x}} &\lesssim \|\du_0\|_{H^s}^2.
\end{aligned}
\end{equation}

\end{enumerate}
\end{proposition}
\begin{proof}
First, to access $H^s$ initial data instead of $L^2$ initial data, we use 
the result of \cite{HGKV} which shows that \eqref{nls3} is globally well posed
in the sense of continuous dependence on initial data. We take an $L^2$ 
approximation of our initial data and apply Theorem \ref{t:main} to this 
approximation. Taking limits we then arrive at Theorem \ref{t:main} for $H^s$ 
initial data.

    If $\|\du_0\|_{H^s} \leq \epsilon$ then the proposition follows directly from 
    Theorem \ref{t:main}. Otherwise, we set  
    \[
        \lambda = (\|\du_0\|_{H^s}/\epsilon)^{2/(1+2s)}
    \]
     where $\epsilon$ is the smallness constant\footnote{which is universal, and thus harmlessly discarded in the statement of the theorem}  in Theorem 
    \ref{t:main} 
    and consider the rescaled initial data 
    \[
\du_0^{\lambda}(x) = \lambda^{-1} \du_0(x/\lambda).
    \]
    Note that 
    \[
\hat \du_0^{\lambda}(\xi) = \hat\du_0(\lambda \xi).
    \]
    We calculate that 
    \begin{align*}
        \|\du_0^{\lambda}\|_{H^s} 
        &\sim \|\hat\du_0^\lambda\|_{L^2_{\xi}([-1,1])} 
        + \||\xi|^s \hat \du_0^\lambda\|_{L^2_\xi((-\infty, -1] \cup [1, \infty)))}\\
        &=\lambda^{-1/2}\|\hat\du_0\|_{L^2_{\xi}([-\lambda,\lambda])} 
        + \lambda^{-1/2-s}\||\xi|^s \hat \du_0\|_{L^2_\xi((-\infty, -\lambda] \cup [\lambda, \infty)))}.\\
        &\leq \lambda^{-1/2-s} \|\du_0\|_{H^s} \leq \epsilon.
    \end{align*}
    This holds because 
    $\lambda \geq 1$, and so on $[-\lambda, -1] \cup [1,\lambda]$ we 
    may replace
    \[1 \leq \lambda^{-s}|\xi|^s.\]
    Thus, we may apply Theorem \ref{t:main} to $\du_0^{\lambda}$ to 
    get a solution $u^\lambda$. It will be convenient to note that 
    \[
        \la \xi \ra^\alpha = (1+\xi^2)^{\alpha/2} 
        = \lambda^{-\alpha}(\lambda^2 + (\lambda \xi)^2)^{\alpha/2}
        = \lambda^{-\alpha}\la \lambda \xi\ra_{\lambda}^\alpha
    \]
    and so 
    \[
        \langle D\rangle^\alpha u^\lambda(x) = 
        \lambda^{-1-\alpha}(\langle D \rangle^\alpha_\lambda u)(x/\lambda)
    \]

    (i) For the energy estimate we have 
    \begin{align*}
        1 &\gtrsim \|\la D\ra^s u^\lambda\|_{L^\infty_tL^2_x}\\
         &= \lambda^{-1-s}\|\la D\ra^s_\lambda u(x/\lambda)\|_{L^\infty_tL^2_x}\\
         &= \lambda^{-1/2-s}\|\la D\ra^s_\lambda u(x)\|_{L^\infty_tL^2_x}\\
    \end{align*}
    In other words
    \[ 
    \|\la D\ra^s_\lambda u\|_{L^\infty_tL^2_x} \lesssim \lambda^{1/2+s}
    \sim \|\du_0\|_{H^s}.
    \]
    (ii) For the Strichartz estimate we have 
    \begin{align*}
        1 &\gtrsim \|\la D\ra^{-(1-4s)/6} u^\lambda\|_{L^6_{t,x}}\\
         &= \lambda^{-1+(1-4s)/6}\|\la D \ra_\lambda^{-(1-4s)/6} 
         u(x/\lambda, t/\lambda^2)\|_{L^6_{t,x}}\\
         &= \lambda^{(-1-2s)/3}\|\la D\ra^{-(1-4s)/6}_\lambda u(x,t)\|_{L^6_{t,x}}\\
    \end{align*}
    Thus, 
    \[
         \|\la D\ra^{-(1-4s)/6}_\lambda u(x,t)\|_{L^6_{t,x}}
         \lesssim \|\du_0\|^{2/3}_{H^s}
    \]
    
    (iii) Finally for the bilinear estimate we have 
    \begin{align*}
        1 &\gtrsim \|\partial_x T_{\la D\ra^{s-\delta_0} \ol{u}^\lambda} \la D\ra^{s-1/2}u^\lambda\|_{L^2_{t,x}}\\
         &= \lambda^{-5/2-2s+\delta_0}\|\partial_x T_{\la D\ra^{s-\delta_0}_\lambda
         \ol{u}(x/\lambda, t/\lambda^2)}  \la D\ra^{s-1/2}_{\lambda}u(x/\lambda, t/\lambda^2)  \|_{L^2_{t,x}}\\
         &= \lambda^{-1-2s+\delta_0}\|\partial_x T_{\la D\ra^{s-\delta_0}_\lambda
         \ol{u}(x, t)}  \la D \ra^{s-1/2}_\lambda u(x, t)\|_{L^2_{t,x}}.
    \end{align*}
    And thus 
    \[
         \|\partial_x T_{\la D\ra^{s-\delta_0}_\lambda
         \ol{u}}  \la D\ra^{s-1/2}_\lambda u\|_{L^2_{t,x}} 
         \lesssim \|\du_0\|_{H^s}^{2 - 2\delta_0/(1+2s)}.
    \]
    Note that when $\|\du_0\|_{H^s} \leq \epsilon$ and setting 
    $\lambda =1$ we get the estimate 
    \[
    \|\partial_x T_{\la D\ra^{s-\delta_0}_\lambda
         \ol{u}}  \la D\ra^{s-1/2}_\lambda u\|_{L^2_{t,x}} \lesssim \|\du_0\|_{H^s}^2
    \]
    directly from Theorem \ref{t:main}. Thus, we only see the $\delta_0$ gain when 
    $\|\du_0\|_{H^s} \gg \epsilon$ and so, up to a constant, this gives the unified bound
    \[
    \|\partial_x T_{\la D\ra^{s-\delta_0}_\lambda
         \ol{u}}  \la D\ra^{s-1/2}_\lambda u\|_{L^2_{t,x}} \lesssim \|\du_0\|_{H^s}^2
         (1 + \|\du_0\|_{H^s}^{2\delta_0/(1+2s)})^{-1}.
    \]
    For the balanced bilinear bound we have 
    \begin{align*}
        1 &\gtrsim \|\partial_x \Pi( \la D\ra^{s-1/4} u^\lambda, \la D\ra^{s-1/4}\ol{u}^\lambda)\|_{L^2_{t,x}}\\ 
         &= \lambda^{-5/2-2s}\|\partial_x \Pi( \la D\ra^{s-1/4}_\lambda 
         u(x/\lambda, t/\lambda^2), 
         \la D \ra^{s-1/4}_\lambda\ol{u}(x/\lambda, t/\lambda^2))\|_{L^2_{t,x}}\\
         &= \lambda^{-1-2s}\|\partial_x \Pi( \la D \ra^{s-1/4}_\lambda 
         u(x, t), 
         \la D\ra^{s-1/4}_\lambda\ol{u}(x, t))\|_{L^2_{t,x}}
    \end{align*}
    And thus
    \[
         \|\partial_x \Pi( \la D\ra^{s-1/4}_\lambda 
         u, 
         \la D\ra^{s-1/4}_\lambda\ol{u})\|_{L^2_{t,x}}
         \lesssim \|\du_0\|_{H^s}^2.
    \]
    
\end{proof}

From here we can access the usual $H^s$ norm by using the 
fact that for any $\alpha < 0$ and any $\lambda,$
\[
    \la \xi\ra^\alpha \leq (1+\lambda^{-\alpha})\la \xi \ra_{\lambda}^\alpha.
\]
Of course, for large initial data, we do not expect to control 
the evolution of frequencies smaller than 
\[
    \lambda = \|\du_0\|_{H^s}^{\frac{2}{1+2s}}
\]
by the initial data. Thus the usual $H^s$ norm will have a loss 
corresponding to this frequency scale. In particular, 
the polynomial dependence of the solution size on the 
initial data will blow up as $s$ goes to $-1/2$. This is the content of 
Theorem \ref{t:NLS-Hs}, which we restate here as a corollary of 
Proposition \ref{prop:large}.

\begin{corollary}
Consider the defocusing 1-d cubic NLS problem \eqref{nls3}$(+)$ with 
    $H^s$ initial data $\du_0$ for $-1/2 < s < 0$. 
    Then the global solution $u$
    satisfies the following bounds:
\begin{enumerate}[label=(\roman*)]
\item Uniform $H^s$ bound:
\begin{equation*}
    \| u \|_{L^\infty_t H^s_x} 
    \lesssim \|\du_0\|_{H^s_x}(1 + \|\du_0\|_{H^s_x}^{\frac{-2s}{1+2s}}).
\end{equation*}

\item Strichartz bound:
\begin{equation*}
    \|\la D\ra^{-\frac{1-4s}{6}} u \|_{L^6_{t,x}} 
    \lesssim \|\du_0\|_{H^s_x}^{2/3}(1 + \|\du_0\|_{H^s_x}^{\frac{1-4s}{3+6s}}). 
\end{equation*}

\item Bilinear Strichartz bound: for every $\delta_0 > 0$ 
\begin{equation*}
\begin{aligned}
    \big\| \partial_x\, T_{\la D\ra^{s-\delta_0}\bar u}\, 
    \la D\ra^{s-\frac12} u
    \big\|_{L^2_{t,x}} &\lesssim \|\du_0\|_{H^s}^{2}
    (1 + \|\du_0\|_{H^s_x}^{\frac{1 - 4s}{1+2s}}),\\[2pt]
    \big\| \partial_x\, \Pi\big(\la D\ra^{s-\frac14} u,\, 
    \la D\ra^{s-\frac14}\bar u\big)
\big\|_{L^2_{t,x}} &\lesssim \|\du_0\|_{H^s}^2
    (1 + \|\du_0\|_{H^s_x}^{\frac{1 - 4s}{1+2s}}).
\end{aligned}
\end{equation*}

\end{enumerate}
\end{corollary}
\begin{proof}

    This is a straightforward application of the fact that 
\[
    \la \xi\ra^\alpha \leq (1+\lambda^{-\alpha})\la \xi \ra_{\lambda}^\alpha
\]
    for $\alpha < 0$ and Proposition \ref{prop:large}. 
    Note that the $\delta_0$ loss cancels when 
$\lambda$ is large so that it does not appear in the final result.

\end{proof}

\bibliography{1d-global}

\bibliographystyle{plain}

\end{document}